\RequirePackage{fix-cm}
\documentclass[smallextended,envcountsame,envcountsect]{svjour3}
\smartqed  %
\usepackage{amsmath}
\usepackage{amssymb}
\usepackage{graphicx}
\usepackage{xcolor}
\usepackage{tikz}
\usetikzlibrary{arrows.meta,positioning,calc}
\usepackage{subcaption}
\usepackage{booktabs}
\usepackage{pifont}
\numberwithin{equation}{section}
\usepackage[hypertexnames=false]{hyperref}

\usepackage{geometry}
\newcommand{\R}{{\mathbb{R}}}

\journalname{}


\newcommand{\doi}[1]{\urlstyle{rm}\url{https://doi.org/#1}}

\makeatletter
\spn@wtheorem{assumption}{Assumption}{\bfseries}{\itshape}
\makeatother

\ifpdf
\hypersetup{
   pdftitle={A true single-level reformulation for pessimistic bilevel optimization},
   pdfauthor={Oliver Stein and Alain Zemkoho}
}
\fi

\begin{document}

\title{A true single--level reformulation for pessimistic bilevel optimization}

\author{Oliver Stein         \and
        Alain Zemkoho
}

\institute{Oliver Stein \at
              Institute for Operations Research (IOR)\\
              Karlsruhe Institute of Technology\\
              \email{stein@kit.edu}\\
              ORCID: \url{https://orcid.org/0000-0001-9514-6317}
           \and
           Alain Zemkoho \at
            School of Mathematical Sciences\\
            University of Southampton\\
            \email{a.b.zemkoho@soton.ac.uk}\\
            ORCID: \url{https://orcid.org/0000-0003-1265-4178}
}

\date{}

\maketitle

\begin{abstract}
We propose a single-level reformulation (SLR) for the pessimistic bilevel optimization problem that does not rely on complementarity conditions or optimal value functions. For this reason, we refer to it as a \emph{true} single-level reformulation (tSLR). A remarkable consequence is that this reformulation can satisfy the classical linear independence constraint qualification, despite the fact that even the weaker Mangasarian--Fromovitz constraint qualification is known to systematically fail for standard single-level reformulations of both optimistic and pessimistic bilevel programs. We leverage on these constraint qualifications to construct new necessary optimality conditions for pessimistic bilevel optimization. The reformulation also has a striking limitation: under the assumptions of our analysis, the classical second-order sufficient condition fails at every Karush--Kuhn--Tucker point of the problem.
Nevertheless, preliminary numerical experiments demonstrate that algorithms based on the proposed tSLR can outperform existing approaches for pessimistic bilevel optimization. Overall, the proposed framework suggests that pessimistic bilevel programs may be considerably more tractable than previously believed and need not be inherently more difficult to solve than their optimistic counterparts.
\end{abstract}
\begin{acknowledgements}
The authors are grateful to Alireza Kabgani and Martin Schmidt for fruitful discussions on an earlier version of this manuscript. The work of the second author was partly funded by the Alexander von Humboldt Foundation, through an Alexander von Humboldt Research Fellowship for Experienced Researchers held within the Continuous Optimization Chair at the Institute for Operations Research, Karlsruhe Institute of Technology (KIT).
\end{acknowledgements}
\setcounter{tocdepth}{2}

\section{Introduction}\label{sec:Intro}
In this paper, we consider the bilevel optimization problem 
\begin{equation}\label{P}\tag{BOP}
   \text{``}\underset{x}\min\text{''}~F(x,y) \; \mbox{ s.t. }\;  x\in X,\;\; y\in S_L(x):=\underset{\quad z\in Y(x)}{\text{argmin}}~f(x,z),
\end{equation}  
{where $F : \mathbb{R}^n\times \mathbb{R}^m \rightarrow \mathbb{R}$  (resp. $f : \mathbb{R}^n\times \mathbb{R}^m \rightarrow \mathbb{R}$)  corresponds to the upper-level (resp. lower-level) objective function, while  the set-valued mapping $Y: X \rightrightarrows \mathbb{R}^m$ describes the lower-level feasible set}
\begin{equation}\label{eq:Y(x)}
Y(x):=\left\{y\in \mathbb{R}^m\left|\;\, g(x, y)\leq 0, \;\; h(x, y)=0\right.\right\}    
\end{equation}
for all $x\in X$. Here, $g$ and $h$ are vector-valued functions from $\mathbb{R}^n\times \mathbb{R}^m$ to $\mathbb{R}^p$ and $\mathbb{R}^q$, respectively. Note that the upper-level (resp. lower-level) player is often also called the \textit{leader} (resp. \textit{follower}). 

{In problem \eqref{P}, the vector $x$ (resp. $y$) represents the upper-level (resp. lower-level) variable. Throughout the paper, the lower-level variable is \textit{continuous}. However, for the upper-level variable, it can be \textit{mixed-integer}; i.e., we could assume that the upper-level feasible set is such that $X\subset \mathbb{Z}^{n_I} \times \mathbb{R}^{n_C}$ with $n_I + n_C =n$ for the construction of the main model proposed in this paper. Hence, throughout the text, unless otherwise stated, it would be assumed that $X\subset \mathbb{Z}^{n_I} \times \mathbb{R}^{n_C}$.} 

Overall, problem \eqref{P} corresponds to the \textit{upper-level problem}, while the set-valued mapping $S_L: X \rightrightarrows \mathbb{R}^m$ collects all the optimal solutions of the \textit{lower-level problem}
\begin{equation}\label{eq:LL}\tag{\text{$LL(x)$}}
    \min_z\,f(x,z)\ \text{ s.t. }\ z\in Y(x)
\end{equation}
for a given upper-level variable $x\in X$.

If the lower-level player has a unique optimal solution for all choices of the upper-level player, i.e., $|S_L(x)|=1$ (with $|.|$ denoting the cardinality of the corresponding set) for all $x\in X$, then problem \eqref{P} can be reduced to the optimization problem
\begin{equation}\label{eq:P-i}
    \min\,\mathcal{F}(x)\; \mbox{ s.t. }\; x\in X
\end{equation}
with real-valued objective $\mathcal{F}(x):= F(x, y(x))$ and $S_L(x)=\{y(x)\}$ for all $x\in X$. Given that the analytical expression of the function $y(\cdot): X \rightarrow \mathbb{R}^m$ is generally unknown, \eqref{eq:P-i} is usually referred to as the implicit function reformulation of  problem \eqref{P}. Problem \eqref{eq:P-i} has been one of the main frameworks to develop solution algorithms to solve bilevel programs (see, e.g., \cite{dempe2002foundations}) and has become even more widely used recently in the context of machine learning applications \cite{grazzi2026bilevel}.

If it is not possible to ensure that the lower-level problem has a unique optimal solution for all upper-level variables, then problem \eqref{P} is not well-posed; hence, the quotation marks on the min operator in \eqref{P} are commonly used to reflect the ambiguity in the upper-level minimization in this case. To make the problem mathematically tractable in this situation, two interpretations to capture the nature of the interaction between the upper- and lower-level players have been widely used in the literature. The most popular one is the optimistic model
\begin{equation}\label{eq:Optimistic-Problem}\tag{\text{$P_o$}}
      \min\,\varphi_o(x)\; \mbox{ s.t. }\; x\in X
\end{equation}
with
\begin{align}\label{eq:def_phi_o}
    \varphi_o(x):=\underset{y\in S_L(x)}{\min}F(x, y),
\end{align}
where it is assumed that whenever the lower-level player has more than one option for a given choice of the upper-level player, they pick one that is in favor of the latter; hence, problem \eqref{eq:Optimistic-Problem} captures full cooperation between both players. 

If cooperation is not possible between the two players, then the leader, as a risk-averse player, will try to protect themself against potential worst choices from the follower by solving the pessimistic bilevel optimization problem 
\begin{equation}\label{eq:Pessimistic-Problem}\tag{\text{$P_p$}}
      \min\,\varphi_p(x)\; \mbox{ s.t. }\; x\in X
\end{equation}
with
\begin{align}\label{eq:def_phi_p}
\varphi_p(x):=\underset{y\in S_L(x)}{\max}F(x, y).
\end{align}
The positions considered in problems \eqref{eq:Optimistic-Problem} and \eqref{eq:Pessimistic-Problem} can be seen as extreme, as they reflect either a situation where there is cooperation or not, as represented by problems \eqref{eq:Optimistic-Problem} and \eqref{eq:Pessimistic-Problem}, respectively. 
In consideration of this, many papers have recently considered a partial cooperation model, which could be obtained by minimizing a convex combination of $\varphi_o$ and $\varphi_p$ under the upper-level constraint. For more details on such models; see, e.g., \cite{aboussoror2017strong}. Also see \cite{zemkoho2016solving} for a set-valued optimization approach to tackle problem \eqref{P} when $|S_L(x)|>1$ for some value(s) of $x\in X$.

Overall, in terms of approaches to tackle \eqref{P}, the most widely used models are \eqref{eq:P-i}, especially since the recent breakthroughs in machine learning, and  problem \eqref{eq:Optimistic-Problem}, however, instead through its standard optimistic reformulation \cite{dempe2020bilevel}. As for the pessimistic bilevel optimization problem \eqref{eq:Pessimistic-Problem}, it has not attracted the same level of attention as the optimistic one. However, recently, there has started to be some growing interest in the development of algorithms for  problem \eqref{eq:Pessimistic-Problem}; the next section of this paper provides a detailed overview of the state of the literature on approaches to tackle the problem. 
As it will become clear there, the existing approaches to tackle problem \eqref{eq:Pessimistic-Problem} can be organized in two main categories: (i) moving the difficult component of the problem (i.e., in particular, the inclusion $y\in S_L(x)$) to the feasible set of a constrained optimization-based transformation of the problem through a semi-infinite programming model \cite{wiesemann2013pessimistic} or models closely related to the standard pessimistic problem introduced in \cite{lampariello2019standard}; and (ii) viewing the problem \eqref{eq:Pessimistic-Problem} as that of effectively minimizing the two-level value function (TLVF) $\varphi_p$ subject $x\in X$; see, e.g., 
\cite{benchouk2025scholtes,cao2026single,vcervinka2013computation}.

Fewer works have been dedicated to the development of methods for \eqref{eq:Pessimistic-Problem} because it is seen as a very difficult problem class. 
A main challenge in solving a pessimistic bilevel problem \eqref{eq:Pessimistic-Problem} lies in the fact that $\varphi_p$ may not be lower semicontinuous, resulting in unsolvability due to a non-attained finite infimum \cite{benchouk2025scholtes,zemkoho2016solving}. We discuss this issue and its relation to the proposed single--level reformulation in Subsection~\ref{sec:jumps}. %
Hence, it is very easy to find examples of bilevel programs where problem \eqref{eq:Pessimistic-Problem} has no optimal solutions, while its optimistic counterpart \eqref{eq:Optimistic-Problem} does have optimal solutions. In this paper, we introduce a reformulation of problem \eqref{eq:Pessimistic-Problem}, which is likely to change this perception. In fact, leveraging the Wolfe duality associated to the intermediate-level problem
\begin{equation}\label{eq:IL}\tag{\text{$IL(x)$}}
    \max_y\,F(x,y)\ \text{ s.t. }\ y\in S_L(x),
\end{equation}
we construct a new reformulation of the pessimistic bilevel program with the following key features:
\\[1ex]
\textbf{(A)} Our reformulation of problem \eqref{eq:Pessimistic-Problem} is a constrained optimization problem that involves neither complementarity constraints nor an optimal value function. For this reason, we refer to it as a \emph{true single-level reformulation} (tSLR). To the best of our knowledge, all existing single-level reformulations of \eqref{eq:Pessimistic-Problem} rely either on complementarity conditions or on an optimal value function; see Section~\ref{sec:Existing reformulations and numerical algorithms}. The same observation applies to optimistic bilevel optimization, where existing reformulations are based either on complementarity conditions or on the lower-level value function
\begin{align}\label{eq:phi_L}
\varphi_L(x):=\min_{z\in Y(x)}f(x,z).
\end{align}
\textbf{(B)} The tSLR introduced in this paper does not involve implicit variables; that is, variables that are part of the  constraint set but not the objective function. Such variables are a major source of numerical difficulties when solving several existing SLRs of the standard optimistic bilevel optimization problem \cite{dempe2025duality}. {Moreover, our proposed tSLR can accommodate mixed-integer upper-level variables.}
${}$\\[1ex]
\textbf{(C)} A remarkable feature of the proposed tSLR is its compatibility with classical constraint qualifications. In particular, we prove that the linear independence constraint qualification (LICQ) can hold for a broad class of pessimistic bilevel programs. This stands in sharp contrast with existing SLRs of bilevel optimization, for which the weaker Mangasarian--Fromovitz constraint qualification (MFCQ) is known to fail systematically, both in the optimistic and pessimistic settings; {see, e.g., \cite{dempe2020bilevel,dempe2002foundations,ye2010new}.}

Leveraging the possibility that these constraint qualifications can hold for our tSLR, we derive completely new classes of necessary optimality conditions for \eqref{eq:Pessimistic-Problem}. However, this paper reveals a fundamental limitation for our tSLR: we prove that it systematically fails to satisfy the classical second-order sufficient condition (see Section~\ref{sec:Preliminaries} for the definition of the concept) at its Karush-Kuhn-Tucker points within the framework considered in this paper. Hence, despite its compatibility with classical constraint qualifications, our tSLR  violates the classical second-order sufficient condition.

Before continuing further, note that the set-valued mapping $S_p: X \rightrightarrows \mathbb{R}^m$ collects all the optimal solutions of the intermediate-level problem \eqref{eq:IL} for any given $x\in X$; that is, %
\begin{equation}\label{eq:Two-level-Function-S_p}
S_p(x):=\underset{y\in S_L(x)}{\arg\max}\,F(x, y).
\end{equation}
To focus our attention on the main ideas, we use the following blanket assumption throughout the paper. Sufficient conditions for it in terms of the problem data may be found, e.g., in \cite{lampariello2019standard}.
\begin{assumption}\label{ass:Sp_domain} For all $x\in X$, it holds that $S_p(x) \neq \emptyset$.
\end{assumption}
{Under this assumption, not only the two-level value function $\varphi_p$  \eqref{eq:def_phi_p}
is real-valued on $X$ but, in view of $S_p(x)\subset S_L(x)$ for all $x\in X$, this is also the case for 
the lower-level value function $\varphi_L$ \eqref{eq:phi_L}.}

Throughout the paper, we use the following basic optimal solution concept for problem \eqref{eq:Pessimistic-Problem}. 
\begin{definition}\label{def:local_optimal_sol}
A point $\bar{x}\in X$ will be said to be a local optimal solution of problem \eqref{eq:Pessimistic-Problem} if there exists a neighborhood $U$ of $\bar{x}$ such that 
\begin{equation}\label{eq:LocalOptimal_Sol}
\varphi_{p}(\bar{x})\leq\varphi_{p}(x) \;\,\mbox{ for all }\;\,  x\in X\cap U.
\end{equation}
\end{definition}
Similarly, $\bar{x}\in X$ will be said to be a global optimal solution for \eqref{eq:Pessimistic-Problem}
  if \eqref{eq:LocalOptimal_Sol} holds with $U=\mathbb{R}^n$.

Of course, this is not the only optimal solution notion for problem \eqref{eq:Pessimistic-Problem}; for a detailed study of solution concepts for the problem, interested readers are referred to the article \cite{aussel2019pessimistic}. 

The following example gives a flavor of the extraordinary nature of the true single-level reformulation of problem \eqref{eq:Pessimistic-Problem} introduced in this paper. %
\begin{example}\label{ex:quadratic}
Consider a class of problem \eqref{eq:Pessimistic-Problem} with the upper- and lower-level objective functions $F :\mathbb{R}^n\times \mathbb{R}^m \rightarrow \mathbb{R}$ and $f :\mathbb{R}^n\times \mathbb{R}^m \rightarrow \mathbb{R}$ respectively defined  by
\begin{equation}\label{eq:F(x,y)}
    F(x,y):=\frac{1}{2}
\begin{bmatrix}
	x^\top & y^\top
\end{bmatrix} 
\left[\begin{array}{lr}
   Q_{11}  &  Q_{12}\\[1ex]
   Q^\top_{12}  & Q_{22}
      \end{array}\right]
\begin{bmatrix}
	x \\ y
\end{bmatrix}  -  c^\top \begin{bmatrix}
	x \\ y
\end{bmatrix} \;\, \mbox{ and } \;\, f(x,y):=(c^f)^\top y
\end{equation}
with the data vectors $c:=(c^\top_x, c^\top_y)^\top\in \mathbb{R}^{n+m}$ and $c^f \in \mathbb{R}^m$, the matrices $Q_{11}\in \mathbb{R}^{n\times n}$, $Q_{12}\in\R^{n\times m}$, as well as a negative semi-definite symmetric  matrix $Q_{22}\in \mathbb{R}^{m\times m}$. Additionally, let the upper- and lower-level feasible sets be given by
\begin{equation}\label{eq:X-and-Y(x)-Linear}
   {X:=\left\{x\in \mathbb{Z}^{n_I} \times \mathbb{R}^{n_C}\left|~A^Gx \leq b^G \right.\right\}}\;\mbox{ and }\; Y(x):=\left\{y\in \mathbb{R}^m\left|~A^gx + B^gy \leq b^g \right.\right\},
\end{equation}
respectively, with the data vectors $b^G\in \mathbb{R}^{r}$ and $b^g\in \mathbb{R}^{p}$, as well as the matrices $A^G\in \mathbb{R}^{r\times n}$, $A^g\in \mathbb{R}^{p\times n}$, and $B^g\in \mathbb{R}^{p\times m}$. 
Our tSLR reduces in this case to the following classical-type quadratic optimization problem with linear constraints: 
\begin{equation}\label{eq:QP_Reform}
    \begin{array}{rl}
   \underset{x, y, z, u,  w}{\min} & F(x,y) - u^\top \left(A^gx + B^gz - b^g\right) - w (y-z)^\top c^f  \\[1ex]
     \mbox{s.t.} & Q^\top_{12}x + Q_{22}y  -  \left(B^g\right)^\top u  =  c_y, \;\, {x\in \mathbb{Z}^{n_I} \times \mathbb{R}^{n_C}},\\[1ex] 
                & A^Gx \leq b^G, \;\, A^gx + B^gz \leq b^g, \;\, u\geq 0, \;\, w\geq 0. %
\end{array}
\end{equation}
Without any additional assumption, the $x$-component of any global optimal solution of this problem is a global optimal solution of the corresponding version of problem \eqref{eq:Pessimistic-Problem}, and conversely, for any global optimal solution $x$ of the corresponding \eqref{eq:Pessimistic-Problem}, there are possibly many choices of the quadruple $(y, z, u, w)$ such that $(x, y, z, u, w)$ is globally optimal for problem \eqref{eq:QP_Reform}.  Obviously,  problem \eqref{eq:QP_Reform} can be solved by any {suitable} quadratic optimization solver. A reformulation like this (without any complementarity constraint or optimal value function) is not possible in optimistic bilevel optimization, and does not seem to have been obtained before for pessimistic bilevel optimization. \qed

In Section \ref{sec:Single-level-Reform}, we introduce a general version of our tSLR model and conduct a rigorous analysis establishing global and local relationships between it and the original problem \eqref{eq:Pessimistic-Problem}. 
\end{example}

For the remainder of the paper, note that in the next section, we conduct a detailed survey of the different existing reformulations and solution methods to tackle problem \eqref{eq:Pessimistic-Problem}. Subsequently, Section \ref{sec:Preliminaries} covers the basic mathematical tools that will be used for the analysis in Sections \ref{sec:Single-level-Reform}--\ref{sec:Algorithmic framework for numerical computations}. More specifically, in Section \ref{sec:Single-level-Reform}, we introduce the motivational background and construction process for our tSLR model, establish the global and local relationship with problem \eqref{eq:Pessimistic-Problem}, and address insights in terms of the behavior of this reformulation and some practical implications. Subsequently, Section \ref{sec:Optimality conditions} is dedicated to the analysis of the behavior of classical MFCQ and LICQ when applied to tSLR in the case of purely continuous upper level variables; in particular, we construct tractable frameworks ensuring the automatic fulfillment of these constraint qualifications. Subsequently, we use the MFCQ and LICQ to derive completely new \textit{first order} and \textit{second order} necessary optimality conditions for problem \eqref{eq:Pessimistic-Problem}, respectively, under conditions not affordable even in the context of the optimistic bilevel program.  In Section \ref{sec:Optimality conditions}, however, we also prove that the  classical second order sufficient condition systematically fails for the tSLR problem developed in this paper. To demonstrate the potential of our tSLR reformulation, Section \ref{sec:Algorithmic framework for numerical computations} provides some basic numerical illustrations. Our experiments on a selection of problems show that overall, our model can lead to better numerical performance, in comparison to the global approaches from the existing literature, which are presented in the next section. 

\section{Existing reformulations and methods} \label{sec:Existing reformulations and numerical algorithms}
We start here by recalling that based on the lower-level optimal value function from \eqref{eq:phi_L}, the lower-level solution set-valued mapping  can be written as 
\begin{align*}
    S_L(x)=\{z\in Y(x)\mid f(x,z)\leq\varphi_L(x)\}
\end{align*}
for each $x\in X$. With
the function $G_L :\mathbb{R}^n\times \mathbb{R}^m \rightarrow\mathbb{R}^{p+1}$ defined by
\[
G_L(x, y):=\left(\begin{array}{c}
         g(x, y)\\[1ex]
         f(x, y)-\varphi_L(x)
    \end{array}\right),
\]
we have 
\begin{align}\label{eq:Zx}
    S_L(x)=Z(x):=\{y\in\R^m\mid G_L(x,y)\leq0, \; h(x, y)=0\},
\end{align}
and may therefore write the  intermediate-level maximal value function from \eqref{eq:def_phi_p} as
\[
\varphi_p(x)=\max_{y\in Z(x)}F(x, y).
\]

Considering the pessimistic bilevel program \eqref{eq:Pessimistic-Problem}, 
we introduce the following assumption; see next section for the definition of the Guignard constraint qualification (GCQ).
\begin{assumption}\label{Assumption1} It holds that:
    \begin{itemize}
        \item[$(1)$] $\forall x\in X$, $F(x, \cdot)$ is a concave function.
        \item[$(2)$] $\forall x\in X$, $f(x, \cdot)$ is a convex function. %
        \item[$(3)$] $\forall x\in X$, $g_i(x, \cdot)$ is a convex function for $i=1, \ldots, p$.
         \item[$(4)$] $\forall x\in X$, $h(x, \cdot)$ is affine linear.
        \item[$(5)$] %
         $\forall x\in X$, \eqref{eq:GCQ} holds in $Z(x)$ at each $y\in S_p(x)$.
   \end{itemize}
\end{assumption}
We suppose throughout this section that Assumption \ref{Assumption1} holds and all the functions $F$, $f$, $g$, and $h$ are at least once continuously differentiable.

The most natural interpretation of problem \eqref{eq:Pessimistic-Problem} is to view it as a minmax problem, but with a very complex coupled and implicitly defined inner feasible set %
$S_L(x)$. Standard minmax programming algorithms already struggle with simple linear constraints (see survey in \cite{cipollasingle}). %
 \cite{zeng2020practical}  addresses the implicit nature of the coupled constraint in \eqref{eq:Pessimistic-Problem} with the new minmax model 
\begin{equation}\label{eq:minmax_Bo_Zeng}\tag{MM}
    \underset{x\in X,\, z\in Y(x)}{\min}~\underset{y\in \mathcal{S}_\mathcal{L}(x, z)}{\max}~F(x, y),
\end{equation}
where the new set-valued mapping $\mathcal{S}_\mathcal{L}$ that replaces $S_L$ is given by
\[
        \mathcal{S}_\mathcal{L}(x, z):=\left\{y\in Y(x)\left|~f(x,y)\leq f(x,z) \right.\right\}.
\]
Let $A$ denote the set of global optimal solutions of problem \eqref{eq:Pessimistic-Problem} in the sense of Definition \ref{def:local_optimal_sol} and similarly, let $B$ collect all the global optimal solutions of problem \eqref{eq:minmax_Bo_Zeng}. Then it can be shown (see \cite{zeng2020practical} or \cite[Theorem 2.5]{BeckEtAl2027}) that 
$
A=\text{proj}_x B,
$
where \textit{proj} stands for the parallel projection mapping. 

Observe that problem \eqref{eq:minmax_Bo_Zeng} is globally equivalent to the constrained optimization problem
\begin{equation}\label{eq:Standard_minmax}
     \underset{x, y, z}{\min}~F(x, y) \;\mbox{ s.t. }\; x\in X,\, z\in Y(x), \;\; y\in K(x, z):=\underset{\quad y\in \mathcal{S}_\mathcal{L}(x, z)}{\arg\max}~F(x, y),
\end{equation}
If Assumption \ref{Assumption1} holds, then problem \eqref{eq:Standard_minmax} is globally equivalent to the problem
\begin{equation}\label{eq:minmax_Bo_Zeng_SLR}\tag{MM-CC}
\begin{array}{rl}
   \underset{x, y, z, u, v, w}{\min} & F(x, y) \\[1ex]
     \mbox{s.t.} & x\in X,\; z\in Y(x),\\[1ex]
                &\nabla_y F(x, y) - \nabla_y g(x, y)^\top u - \nabla_y h(x, y)^\top v- w \nabla_y f(x, y) =0,\\[1ex]
                & w\geq 0,\;\, f(x, y)-f(x,z)\leq 0,\;\,w(f(x,y)-f(x,z))=0,\\[1ex]
                & u\geq 0, \;\, g(x, y)\leq 0, \;\, u^\top g(x, y)=0,\\[1ex]
                &h(x, y)=0,
\end{array}
\end{equation}
which is a special class of the mathematical program with equilibrium constraints (MPEC). It is important to note that due to the presence of the constraint $f(x,y)\leq f(x, z)$, which is a proxy of the classical lower-level value function constraint $f(x, y)\leq \varphi_L(x)$, stronger constraint qualifications such as the MFCQ will fail for problem \eqref{eq:Standard_minmax}. Hence, the much weaker Assumption \ref{Assumption1}(5) is more appropriate here. For some linear cases of problem \eqref{eq:Pessimistic-Problem}, transformations of the form \eqref{eq:minmax_Bo_Zeng_SLR} are used for the development of numerical methods in the articles \cite{liu2018new,zeng2020practical}. %

Obviously, the SLR \eqref{eq:minmax_Bo_Zeng_SLR}, as well as the standard pessimistic and closely related ones introduced in the next subsection, bring the pessimistic bilevel program in the realm of MPECs, and therefore techniques commonly used for the classical KKT reformulation of the standard optimistic bilevel optimization problem might be possible paths for numerical methods for the problem. However, there are some key differences between reformulation \eqref{eq:minmax_Bo_Zeng_SLR} and the KKT reformulation for optimistic bilevel optimization. First, the feasible set of the former involves the leader's objective function, while the follower's objective function is part of the complementarity constraints. This potentially makes problem \eqref{eq:minmax_Bo_Zeng_SLR} much more difficult to solve. Moreover, no rigorous analysis of the relationship between this problem and the original problem \eqref{eq:Pessimistic-Problem} has been conducted yet, especially w.r.t. local optimal solutions. Additionally, it is unclear whether the MPEC theory w.r.t. constraint qualifications, as well as necessary and sufficient optimality conditions (in terms of S-, M-, and C-type constraint qualifications and stationarity concepts) can work as it is the case for the KKT reformulation for the optimistic bilevel optimization problem; see, e.g., \cite{dempe2012karush,jane2005necessary,guo2013second}.

Another notable issue is that problem \eqref{eq:minmax_Bo_Zeng_SLR} involves two more \textit{implicit variable categories} (i.e., variables present in the feasible set but not in the objective function) than the KKT reformulation of the standard optimistic bilevel program; namely, the second presence of the lower-level variable via $z$ and the Lagrange multiplier $w$ associated to the constraint $f(x,y)\leq f(x, z)$. As it has been extensively studied in the literature (see, e.g., \cite{dempe2025duality} for analysis related to SLRs for optimistic bilevel programs), implicit variables are one of the main causes of challenges involved in numerically solving the KKT reformulation of the standard optimistic bilevel optimization problem. As it can be sensed from the introductory example in \eqref{eq:QP_Reform}, the true single-level reformulation introduced in this paper does not involve any implicit variable or complementarity constraint. 

\subsection{Standard pessimistic-type reformulations}
The following \textit{standard pessimistic} version of problem \eqref{eq:Pessimistic-Problem} was formally introduced in \cite{lampariello2019standard}:
\begin{equation}\label{eq:StandPess}\tag{SP}
    \underset{x, y}{\min}~F(x, y) \;\mbox{ s.t. }\; x\in X, \;\; y\in S_p(x).
\end{equation}
\eqref{eq:Pessimistic-Problem} is globally equivalent to \eqref{eq:StandPess} in the sense of \cite[Proposition~4.1]{lampariello2019standard}. In \cite{lampariello2019standard}, the ``bridge construction'' for simple bilevel programs from \cite{lampariello2017bridge} is applied to reformulate the ``vertical'' bilevel connection between the intermediate \eqref{eq:IL} and the lower-level \eqref{eq:LL} as the ``horizontal'' connection of a generalized Nash equilibrium problem. This reformulates \eqref{eq:StandPess} into the multi-follower game 
\begin{equation}\label{eq:MFG}\tag{MFG}
    \underset{x, y, z}{\min}~F(x, y) \;\mbox{ s.t. }\; x\in X, \;\; (y,z)\in E(x),
\end{equation}
where, for $x\in X$, $E(x)$ denotes the set of generalized Nash equilibria of the two player problem that can be written as
\begin{equation}\label{eq:GNEP}\tag{GNEP}
 \begin{array}{rl}
   \underset{y}{\min} & -F(x,y)\\[1ex]
     \mbox{s.t.} & y\in Y(x),\ f(x,y)\leq f(x,z)
\end{array} \qquad \qquad 
 \begin{array}{rl}
   \underset{z}{\min} & f(x,z)\\[1ex]
     \mbox{s.t.} & z\in Y(x).
\end{array}
\end{equation}

The right-hand side problem in \eqref{eq:GNEP} is convex under Assumption~\ref{Assumption1}, and since $z$ acts as a parameter, also the left-hand side problem is convex. Since in equilibrium points the left-hand side player's constraint $f(x,y)\leq f(x,z)$ becomes $f(x,y)\leq\varphi_L(x)$, the \eqref{eq:GCQ} is satisfied at all optimal points of the left-hand side problem under Assumption~\ref{Assumption1}(5). Under the mild additional assumption of \eqref{eq:GCQ} at each optimal point of the right-hand side problem, in the equilibrium points the respective KKT conditions are necessary and sufficient for optimality in the two player problems (note that, as opposed to LICQ and MFCQ, for nonpolyhedral convex sets neither the ACQ nor the GCQ are necessarily preserved under dropping of constraints). 

Replacing optimality in the two player problems by their respective KKT systems results in the true single-level mathematical program with complementarity constraints
\begin{equation}\label{eq:tSLRCC}\tag{SP-CC}
\begin{array}{rl}
   \underset{x, y, z, u, v, w,\lambda,\mu}{\min} & F(x, y) \\[1ex]
     \mbox{s.t.} & x\in X,\\[1ex]
                &\nabla_y F(x, y) - \nabla_y g(x, y)^\top u - \nabla_y h(x, y)^\top v- w \nabla_y f(x, y) =0,\\[1ex]
                & w\geq 0,\;\, f(x, y)-f(x,z)\leq 0,\;\,w(f(x,y)-f(x,z))=0,\\[1ex]
                & u\geq 0, \;\, g(x, y)\leq 0, \;\, u^\top g(x, y)=0,\\[1ex]
                &h(x, y)=0,\\[1ex]
                &\nabla_z f(x, z) + \nabla_z g(x, z)^\top\lambda + \nabla_z h(x, z)^\top\mu =0,\\[1ex]
                & \lambda\geq 0, \;\, g(x, z)\leq 0, \;\, \lambda^\top g(x, z)=0,\\[1ex]
                &h(x, z)=0.
\end{array}
\end{equation}
The global equivalence of \eqref{eq:tSLRCC} to \eqref{eq:MFG} in the sense of \cite[Proposition.~5.4]{lampariello2019standard}
yields also the global equivalence of \eqref{eq:tSLRCC} to \eqref{eq:Pessimistic-Problem}. %
The main difference between this SLR and \eqref{eq:minmax_Bo_Zeng_SLR} is the presence of the lower-level KKT conditions in the feasible set of \eqref{eq:tSLRCC}; therefore leading to an increase in the number of implicit variables with the additional presence of lower-level Lagrange multipliers. 
Clearly, the number of variables and constraints in  \eqref{eq:tSLRCC} is significantly larger, and an additional algorithmically challenging (lower-level) complementarity system appears. We also remark that, due to the latter issue, in \cite{lampariello2019standard},  \eqref{eq:tSLRCC} is also solved in a mixed-integer reformulation using additional binary variables to address the complementarity constraints. 
Although usually from completely different angles, reformulations of  \eqref{eq:Pessimistic-Problem} of a flavor similar to \eqref{eq:tSLRCC} have been derived in other papers \cite{dempe2018pessimistic,kis2021optimistic,calvete2025novel} and used to solve the problem, usually from the perspective of established techniques in optimistic bilevel optimization such as the the mixed-integer reformulation mentioned above. 

Recall that in problem \eqref{eq:tSLRCC}, the second complementarity system in the feasible set, counting from the top, is a kind of proxy expression representing the value function constraint if the lower-level value function reformulation is applied to the intermediate-level problem \eqref{eq:IL}. Instead of the \eqref{eq:MFG} reformulation, problem \eqref{eq:StandPess} is also globally equivalent to 
\begin{equation}\label{eq:GE}
\underset{x, y}{\min}~F(x, y) \;\mbox{ s.t. }\; x\in X, \;\; \nabla_y F(x, y) \in N_{S_L(x)} (y),
\end{equation}
under Assumption~\ref{Assumption1}(1)--(4). Here, $N_{S_L(x)} (y)$ represents the normal cone to $S_L(x)$ at $y$, in the sense of convex analysis. If additionally,  Assumption~\ref{Assumption1}(5) holds in $S_L(x)$ at $y$ for every $x\in X$, then this problem is globally equivalent, in a suitable sense, to the problem 
\begin{equation}\label{eq:StandPess-KKT}\tag{SP-LF-CC}
\begin{array}{rl}
   \underset{x, y, u, v, w}{\min} & F(x, y) \\[1ex]
     \mbox{s.t.} & x\in X, \;\, h(x, y)=0,\\[1ex]
                 & w\geq 0,\;\, f(x, y)-\varphi_L(x)\leq 0,\\[1ex]
                 & u\geq 0, \;\, g(x, y)\leq 0, \;\, u^\top g(x, y)=0,\\[1ex]
                 &\nabla_y F(x, y) - \nabla_y g(x, y)^\top u - \nabla_y h(x, y)^\top v- w \nabla_y f(x, y) =0.
\end{array}
\end{equation}
This reformulation is used in \cite{guan2025adaprox} to develop a proximal point-type algorithm for the special case of problem \eqref{eq:Pessimistic-Problem}, where the lower-level problem is unconstrained. It might also be useful to mention that the feasible set of  \eqref{eq:StandPess-KKT} is reminiscent of the so-called \textit{combined MPEC and the value function approach} introduced for the standard optimistic bilevel program introduced in \cite{ye2010new}. 

Instead of using the generalized equation reformulation of the intermediate-level problem in \eqref{eq:GE}, the two-level value function \eqref{eq:def_phi_p} could also be used, and this would lead to a double-value function reformulation for problem \eqref{eq:StandPess}, rather than the model in \eqref{eq:StandPess-KKT}. Such a transformation is the basis of the heuristic-type method developed in \cite{antoniou2024delta}. 

\subsection{Semi-infinite programming-based reformulation}
It is well-known (see, e.g., \cite{zemkoho2014simple} and references therein) that  problem \eqref{eq:Pessimistic-Problem} is globally equivalent to the generalized semi-infinite programming problem 
\begin{equation}\label{eq:GSIP}
\underset{x, t}{\min}~t \;\;\;\mbox{ s.t. }\;\;\; x\in X, \quad F(x, y) \leq t\;\;\; \forall y\in S_L(x).
\end{equation}
In the paper \cite{wiesemann2013pessimistic}, the standard semi-infinite programming approximation 
\begin{equation}\label{eq:SIP}\tag{SIP}
 \begin{array}{rl}
   \underset{x, z, t}{\min} & t \\[1ex]
     \mbox{s.t.} & x\in X, \;\, z\in Y, \lambda \,:Y \mapsto [0, 1],\\[1ex]
                 & \lambda(y)\left[f(x, z) - f(x, y) + \epsilon\right] + (1-\lambda (y)) \left[F(x, y) - t \right]\leq 0 \;\;\; \forall y\in Y
\end{array}   
\end{equation}
of problem \eqref{eq:GSIP}, where,  $\lambda \,:Y \mapsto [0, 1]$ defines a decision variable, is used to construct a discretization-type algorithm to compute approximate global optimal solutions for problem \eqref{eq:Pessimistic-Problem}.  Practical approaches to generate $\lambda$ can be found in the latter reference. 

Note that semi-infinite programming-based techniques have been used to developed algorithms to globally solve optimistic bilevel optimization problems in many papers; see, e.g., \cite{mitsos2008global,jungen2026libdips,kleniati2014branch}.

\subsection{Two-level value function-based reformulations}
{We assume in this subsection that $X\subset \mathbb{R}^n$}. Unlike in the context of the previous approaches, where methods are built based on transformations of problem \eqref{eq:Pessimistic-Problem} into constrained single-level reformulations of the problem, there is a stream of methods that rely preliminarily on directly minimizing the two-level value function $\varphi_p$ or constructing a modification of the function that serves as base for numerical algorithms. 
We start here by referring to the work in \cite{vcervinka2013computation}, where a derivative-free optimization (DFO) method based on off-the-shelf tools is applied to approximate the values of the two-level function $\varphi_p$ in order to estimate the derivatives of the function to design an iterative process to solve  \eqref{eq:Pessimistic-Problem}. 

In \cite{strekalovsky2025one}, the idea of Molodtsov \cite{MOLODTSOV197667,MOLODTSOV1973123} is used as preliminary step to design a method to approximate solutions for problem \eqref{eq:Pessimistic-Problem}. Note that the idea of Molodtsov can be viewed as approximating the maximization two-level value function $\varphi_p$ with the minimization one
\begin{equation}\label{eq:regularized-TLVF}
    \varphi^{\theta, \epsilon}_{o}(x):=\underset{y}{\min}\left\{F(x, y)\left|\quad y\in S^{\theta, \epsilon}_L(x)\right.\right\},
\end{equation}
where $\theta > 0$ is penalization parameter and $\varepsilon> 0$ is a relaxation parameter, such that $S^{\theta, \epsilon}_L$ represents a relaxation of the optimal solution set-valued mapping of the regularized lower-level problem obtained by replacing the function $f$ in \eqref{eq:LL} by $f-\theta F$; i.e., 
\[
S^{\theta, \epsilon}_L(x):=\left\{y\in Y(x)\left|\;\, f(x,y)-\theta F(x,y)\leq\varphi^{\theta}_L(x) + \epsilon\right.\right\},
\]
where  $\varphi^{\theta}_L$ is the corresponding optimal value function defined by
\[
\varphi^{\theta}_L(x):= \underset{y}{\min}\left\{f(x,y)-\theta F(x,y)\left|\quad y\in Y(x)\right.\right\}.
\]

\begin{figure}[htbp]
    \centering
    \begin{subfigure}[t]{0.49\textwidth}
        \centering
        \includegraphics[width=\textwidth]{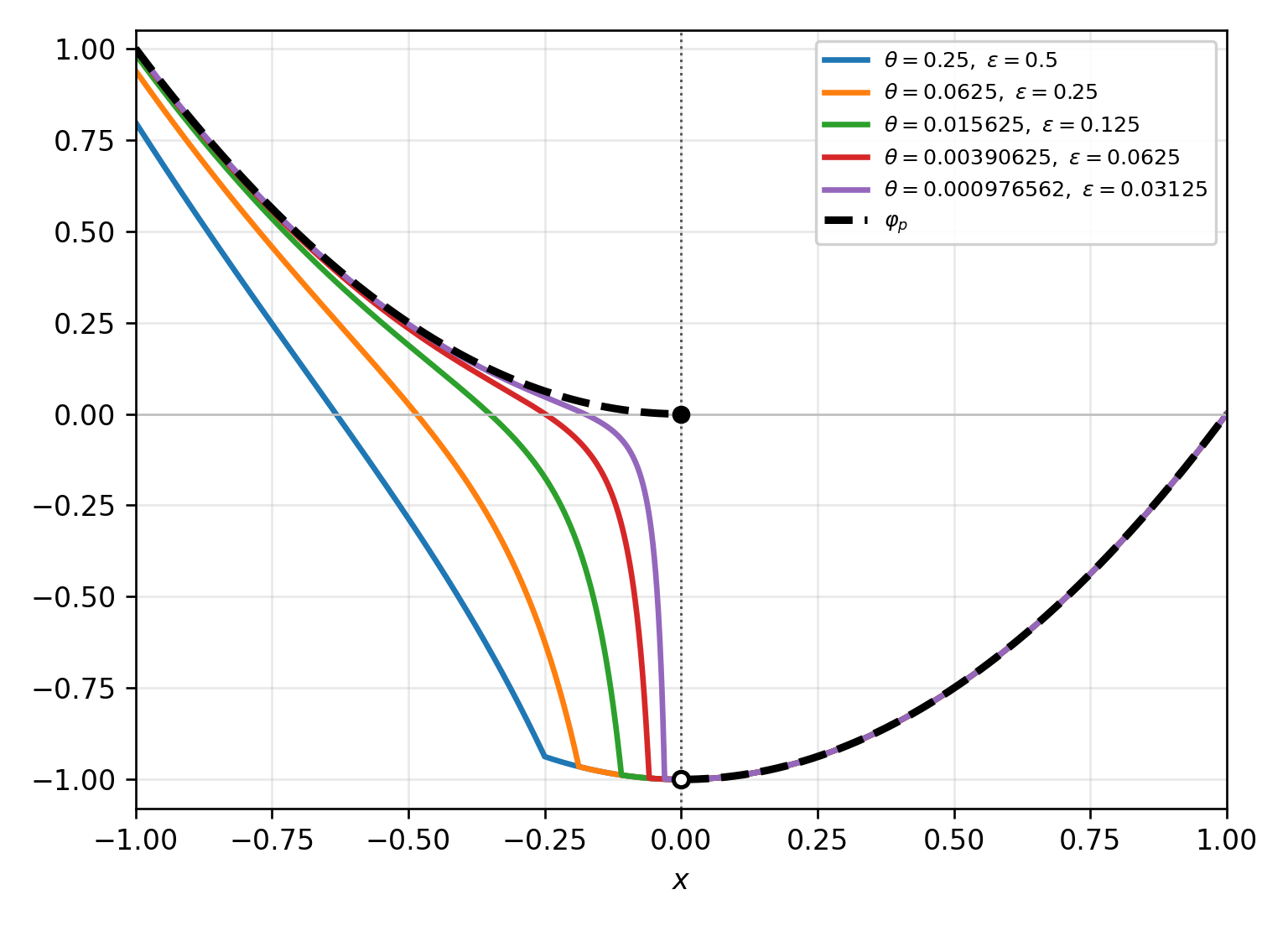}
        \caption{Molodtsov approximation $\varphi_o^{\theta,\epsilon}$ \eqref{eq:regularized-TLVF}}
        \label{fig:improved-molodtsov}
    \end{subfigure}
    \hfill
    \begin{subfigure}[t]{0.49\textwidth}
        \centering
        \includegraphics[width=\textwidth]{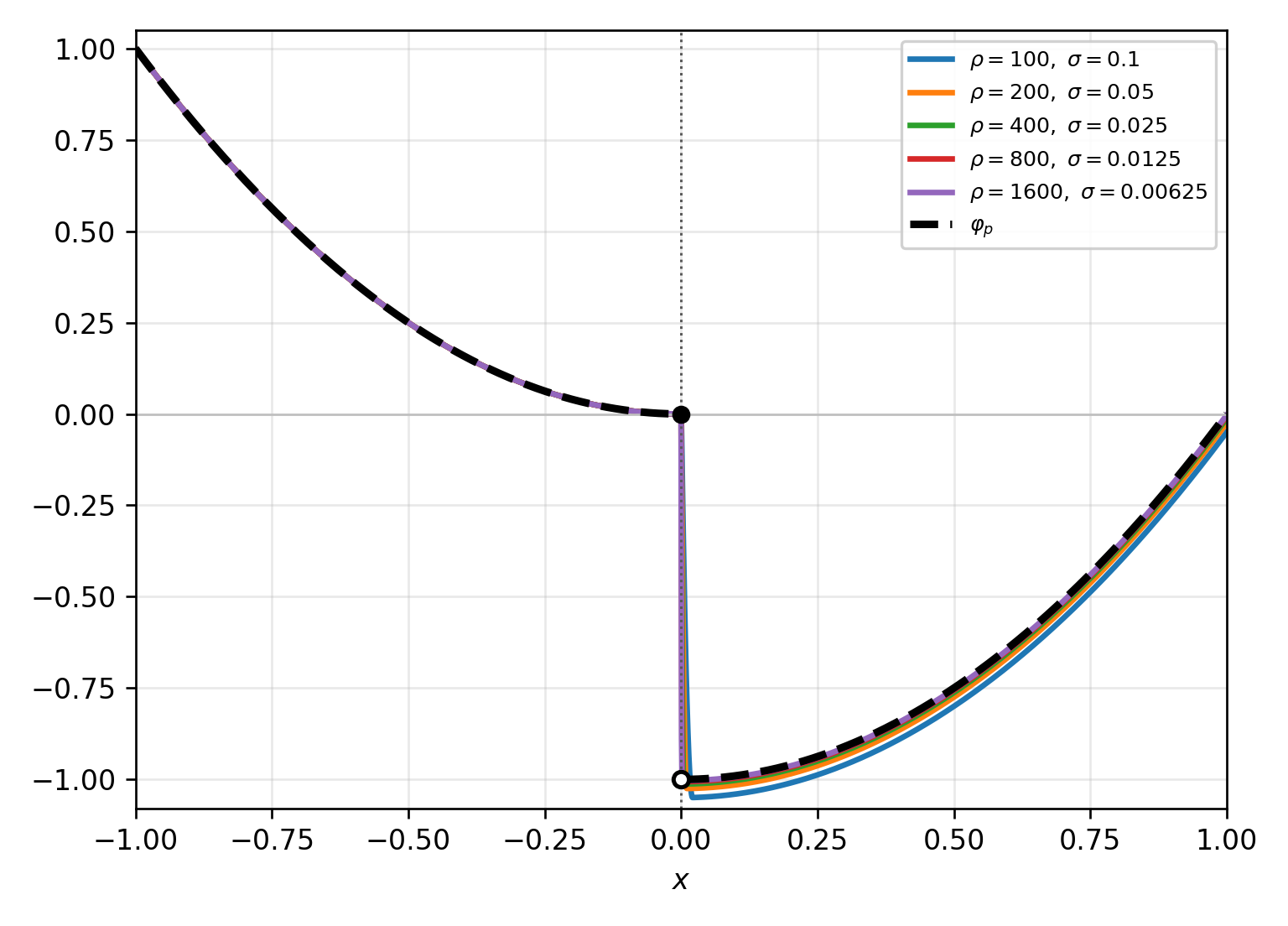}
        \caption{minmax approximation $\phi_{\rho,\sigma}$ \eqref{eq:minmax_approx_phi_p}}
        \label{fig:improved-czz}
    \end{subfigure}

    \vspace{0.5em}

    \begin{subfigure}[t]{0.49\textwidth}
        \centering
        \includegraphics[width=\textwidth]{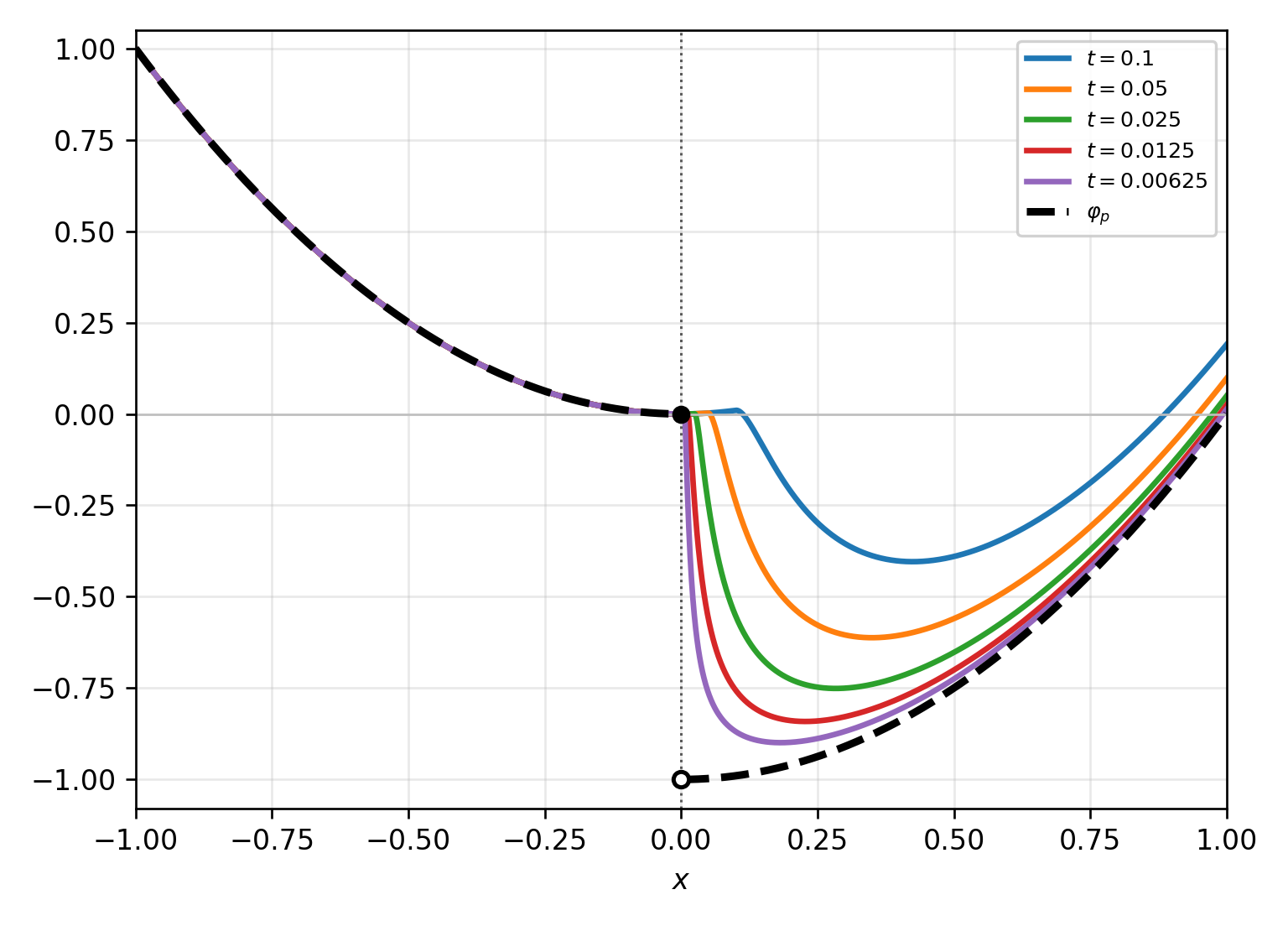}
        \caption{Scholtes relaxation $\psi_{\mathcal S}^{t}$ \eqref{eq:Relaxation_psi_p}}
        \label{fig:improved-scholtes}
    \end{subfigure}
    \hfill
    \begin{subfigure}[t]{0.49\textwidth}
        \centering
        \includegraphics[width=\textwidth]{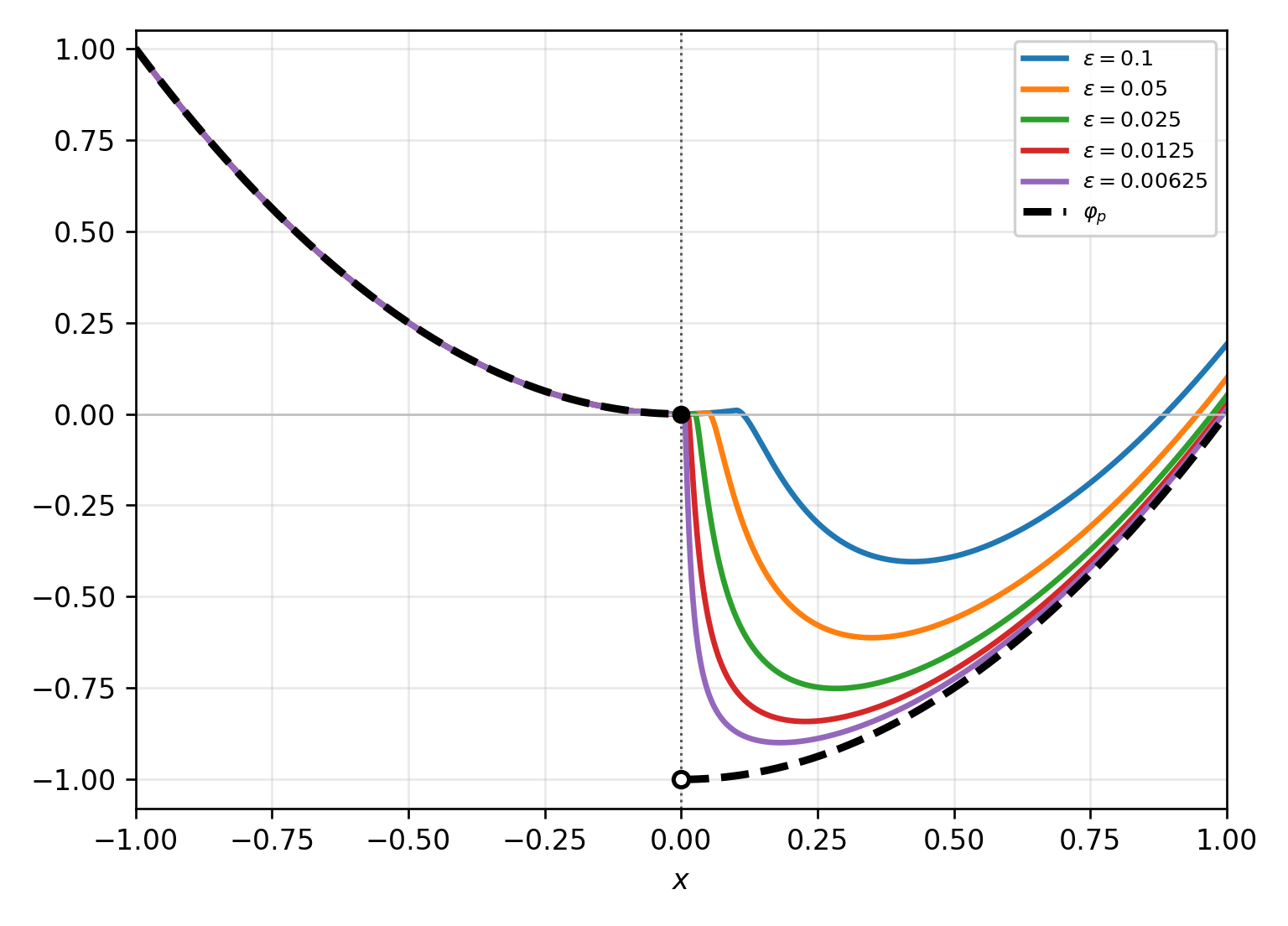}
        \caption{Lower-level value function relaxation $\varphi_p^{\varepsilon}$ \eqref{eq:LLVF-approx-phi_p}}
        \label{fig:improved-value-function}
    \end{subfigure}

    \caption{Approximations of the two-level value function $\varphi_p$ \eqref{eq:def_phi_p} for the pessimistic bilevel program \eqref{eq:Pessimistic-Problem} with $X=[-1,1]$, $Y=[0,1]$, $F(x,y)=x^2-y^2$, and $f(x,y)=-xy$.
    The reference value function is
    $\varphi_p(x)=x^2$ for $x\leq 0$ and $\varphi_p(x)=x^2-1$ for $x>0$,
    with a jump discontinuity at $x=0$.}
    \label{fig:improved-pessimistic-approximations}
\end{figure}

Let the optimal values of the pessimistic bilevel program \eqref{eq:Pessimistic-Problem} and the corresponding regularized problem based on \eqref{eq:regularized-TLVF} be denoted by 
\[
\bar{\varphi}_p := \underset{x\in X}{\min}~\varphi_p(x) \quad \mbox{ and } \quad \bar{\varphi}^{\theta,\varepsilon}_o:= \underset{x\in X}{\min}~\varphi^{\theta, \epsilon}_{o}(x),
\]
respectively. If we assume that the lower-level feasible set is fixed (i.e., $Y(x):=Y$ for all $x\in X$) with $X$ and 
$Y$ being compact metric spaces, while the functions $F$ and $f$ are even just continuous on $X\times Y$, then it is shown in \cite{MOLODTSOV197667} that 
\[
\bar{\varphi}_{p}
=
\lim_{n\to\infty}
\bar{\varphi}^{\theta_n,\varepsilon_n}_{o}
\quad
\mbox{ if }\quad 
\theta_n\to0^+,
\quad
\varepsilon_n\to0^+\;\mbox{ with } \; \frac{\theta_n}{\varepsilon_n} \to0^+.
\]
This result is established under weaker assumptions in \cite{loridan1996weak} and a similar framework for optimistic bilevel optimization is studied in the paper \cite{dempe1996algorithm}.

One fact that is clear from the Molodtsov framework is that the pessimistic bilevel optimization problem is approximated by an optimistic bilevel program; and in fact for $\theta=0$ and $\varepsilon=0$, the optimistic bilevel program is recovered. Hence, in \cite{strekalovsky2025one}, the core optimization problem solved is a standard optimistic-type transformation resulting from \eqref{eq:regularized-TLVF}. 

Under the assumption that $Y(x):=Y$ for all $x\in X$, the article \cite{cao2026single} suggests an approximation of the two-level value function $\varphi_p$ by the minmax value function 
\begin{equation}\label{eq:minmax_approx_phi_p}
    \phi_{\rho, \sigma}(x):=\underset{z\in Y}{\min}\,\underset{y\in Y}{\max}~F(x,y)-\rho \left(f(x, y)- f(x,z)\right) 
+ \frac{\sigma}{2}\|z\|^2 -\sigma y^\top z,
\end{equation}
(with the penalization and regularization parameters $\rho$ and $\sigma$ being positive) and a framework is constructed  to ensure that this function is differentiable and it holds that 
\[
\lim_{k \rightarrow \infty}\left[\underset{x\in X}{\inf}\,\phi_{\rho_k, \sigma_k}(x)\right] = \underset{x\in X}{\inf}~\varphi_p(x)
\]
if $X$ or $Y$ is a bounded set and the sequences $\left\{\rho_k\right\}$ and $\left\{\sigma_k\right\}$ are such that $\rho_k \rightarrow \infty$ and $\sigma_k \rightarrow 0$. 
A gradient descent-type iterative process based on the derivative of the function $\phi_{\rho, \sigma}$ (for $\rho>0$ and $\sigma>0$) is proposed and studied in \cite{cao2026single}.

Finally, in the series of papers \cite{benchouk2025scholtes,benchouk2026relaxation}, considering the KKT reformulation for the lower-level problem, under Assumption \ref{Assumption1}(2)--(4) and the fulfillment of the GCQ in $Y(x)$ at all $y\in S_L(x)$ and $x\in X$, the two-level value function $\varphi_p$ can be approximated by
\begin{equation}\label{eq:Relaxation_psi_p}
    \psi_{\mathcal{R}}^{t}(x):=\underset{(y,u)\in\mathcal{D}^{t}_{\mathcal{R}}(x)}{\max}F(x,y),
\end{equation}
where, for $t>0$ and $x\in X$, $\mathcal{D}^{t}_{\mathcal{R}}(x)$ can be any classical relaxation (labeled as $\mathcal{R}$) w.r.t. complementarity conditions of set of the KKT points of the lower-level problem \eqref{eq:LL}; i.e., 
\begin{equation}\label{Dt}
\mathcal{D}^{t}_{\mathcal{R}}(x):=\left\{(y,u)\in\mathbb{R}^{m+q}\left\vert\,
\nabla_y \ell(x,y,u)=0, \;\, \phi^t_{i, \mathcal{R}}(x, y, u)\leq 0, \;\; i=1, \ldots, q \right.\right\},
\end{equation}
where $\ell(x, y, u):= f(x,y) + u^\top g(x,y)$ denotes the lower-level Lagrangian function with the assumption that $Y(x):=\left\{y\in \mathbb{R}^m\left|\;\, g(x, y)\leq 0\right.\right\}$. In particular, for illustration, we assume here that $\mathcal{R}$ corresponds to the \textit{Scholtes} relaxation (denoted by $\mathcal{R}:=\mathcal{S}$); then for all $t>0$ and $i=1, \ldots, q$, the function $\phi^t_{i, \mathcal{S}}$ is defined for a triplet $(x, y, u)$ by
\[
\phi^t_{i, \mathcal{S}}(x, y, u):= \left(\begin{array}{cc}
            g_i(x,y)  \\
            -u_i\\
            -u_i g_i(x,y)-t
       \end{array} \right).
\]

Such a relaxation creates advantages; for instance, \eqref{eq:Relaxation_psi_p} provides an upper bound for $\varphi_p$ an  if the function  $t\mapsto\psi_{\mathcal{R}}^{t}(x)$ is  upper semicontinuous at $0^+$ for any $x\in \mathbb{R}^n$ and
	 $(t_{k})\downarrow0$, then for any $x\in\mathbb{R}^{n}$, we have
	$\psi_{\mathcal{R}}^{t_{k}}(x)\rightarrow\psi_{p}(x)$ as $k\rightarrow\infty$. Furthermore, let $(t_{k})\downarrow 0$ and $(x^{k})$ be a sequence such that the point $x^{k}$ is a global optimal solution of problem 
\begin{equation}\label{KKT-RG}
    \underset{x\in X}{\min}~\psi_{\mathcal{R}}^{t}(x)
\end{equation}
for $t:=t_k$. If $x^{k}\rightarrow\bar{x}$ as $k\rightarrow\infty$, then $\bar{x}$ is a global optimal solution of  \eqref{eq:Pessimistic-Problem} provided that the function $t\mapsto\psi_{\mathcal{R}}^{t}(x)$ is upper semicontinuous at $0^+$ for any $x\in \mathbb{R}^n$, and the function $x \mapsto \psi_{p}(x)$ is lower semicontinuous at $\bar{x}$.  
Considering this, a method to approximate stationary points of \eqref{eq:Pessimistic-Problem} via the computation of those of problem \eqref{KKT-RG} is introduced and studied in the articles \cite{benchouk2025scholtes,benchouk2026relaxation}.

To conclude this section, it is important to note that a common point between all the approximations of  $\varphi_p$ introduced here is that they lead to functions with much better behavior, as illustrated by the example in Fig. \ref{fig:improved-pessimistic-approximations}. In particular, as in this example, $\varphi_p$ is typically only upper semicontinuous. This highlights a possible drawback of the two-level value function-based numerical methods, as it is very likely that computed points are not optimal for \eqref{eq:Pessimistic-Problem} or at best would be optimistic optimal solutions (as it is very likely to be the case for the example in Fig. \ref{fig:improved-pessimistic-approximations}). 

Note that in Fig. \ref{fig:improved-pessimistic-approximations}, we also included the value function-based approximation of $\varphi_p$ defined by
\begin{equation}\label{eq:LLVF-approx-phi_p}
  \varphi^{\varepsilon}_p(x):= \underset{y}{\max}~\left\{F(x,y)\left|~y\in Y(x), \;\; f(x,y)-\varphi_L(x)\leq \varepsilon \right.\right\},  
\end{equation}
where the function $\varphi_L$ is defined in \eqref{eq:phi_L} and with the relaxation parameter $\varepsilon >0$, which is commonly used in the literature to tackle the bilevel optimization problem. 

\section{Preliminary mathematical background}\label{sec:Preliminaries}
In this section, we introduce some basic results and concepts that will be used throughout the paper.

\subsection{Constraint qualifications, optimality conditions, and duality}\label{sec:CQs} The material presented in this subsection can be found in any standard book on continuous optimization; see, e.g., \cite{fletcher2000practical,wright1999numerical}. The focus in this subsection will be the constrained optimization problem 
\begin{equation}\label{eq:min_standard}%
    \begin{array}{rl}
    \min &\mathfrak{f}(x)\\[1ex]
    \mbox{s.t.} & x\in \mathcal{C}:=\left\{\left. x\in \mathbb{R}^n\right|~\mathfrak{g}(x)\leq 0, \;\; \mathfrak{h}(x)=0\right\}
    \end{array}
\end{equation}
with continuously differentiable functions $\mathfrak{f} : \mathbb{R}^n \rightarrow \mathbb{R}$, $\mathfrak{g} : \mathbb{R}^n \rightarrow \mathbb{R}^p$, and $\mathfrak{h} : \mathbb{R}^n \rightarrow \mathbb{R}^q$. 
Let $\mathcal{I}(\bar x):=\left\{i:=1, \ldots, p|\; \mathfrak{g}_i(\bar x)=0 \right\}$. For ease of notation, in the sequel, we use $\mathcal{I}$ instead of $\mathcal{I}(\bar x)$. 
Let $\nabla g_\mathcal{I}(\bar x)$ be the submatrix of the Jacobian matrix $\nabla g(\bar x)$ made only of the rows with index $i\in \mathcal{I}$. The  linear independence constraint qualification (LICQ) will be said to hold at $\bar x\in \mathcal{C}$ if 
\begin{equation}\label{eq:LICQ}\tag{LICQ}
        \nabla\mathfrak{g}_\mathcal{I}(\bar x)^\top \alpha + \nabla\mathfrak{h}(\bar x)^\top \beta =0
\quad \Longrightarrow \quad \left[\alpha =0,\;\;\beta =0\right].
\end{equation}
The  Mangasarian-Fromovitz constraint qualification (MFCQ) will be said to hold at a point $\bar x\in \mathcal{C}$ if the following condition is satisfied:
\begin{equation}\label{eq:MFCQ}\tag{MFCQ}
    \left.\begin{array}{r}
        \nabla\mathfrak{g}(\bar x)^\top \alpha + \nabla\mathfrak{h}(\bar x)^\top \beta =0\\[1ex]
         \alpha\geq 0, \;\; \alpha^\top\mathfrak{g}(\bar x)=0
    \end{array}\right\} \Longrightarrow \quad \left[\alpha =0,\;\;\beta =0\right].
\end{equation}

Another constraint qualification that will be useful in the sequel is the Abadie constraint qualification (ACQ), which is weaker than the MFCQ. To introduce it, we first recall the tangent cone to $\mathcal{C}$ at one of its points $\bar x$: 
\[
T_{\mathcal{C}}(\bar x):=\left\{d\in \mathbb{R}^n\left|~\begin{array}{ll}
    \exists \{x^k\}_{k\in \mathbb{N}}\subset \mathcal{C}, \;\; \exists \{t_k\}_{k\in \mathbb{N}}\subset (0, \; \infty): \\[2ex]
     x^k \rightarrow \bar x, \;\; t_k\downarrow 0, \;\; \left(x^k -\bar x\right)/t_k \rightarrow d
\end{array} \right.\right\}.
\]
The ACQ will be said to hold at a point $\bar x \in \mathcal{C}$ if this tangent cone coincides with the linearized tangent cone to $\mathcal{C}$ at the same point $\bar x$; i.e., 
\begin{equation}\label{eq:ACQ}\tag{ACQ}
    T_{\mathcal{C}}(\bar x)=\left\{d\in \mathbb{R}^n\left|~\begin{array}{ll}
    \nabla\mathfrak{g}_i(\bar x)^\top d\leq 0\;\; \forall i:\;\, \mathfrak{g}_i(\bar x)=0\\[2ex]
     \nabla\mathfrak{h}_j(\bar x)^\top d = 0\;\; \forall j=1, \ldots, q
\end{array} \right.\right\} =: L_{\mathcal{C}}(\bar x).
\end{equation}
And finally, we introduce a constraint qualification weaker than the ACQ. To proceed, note that for a given cone $\mathcal{K}\subset \mathbb{R}^n$, its dual is the cone 
\begin{equation}\label{eq:dual_cone}
\mathcal{K}^*:=\left\{\left.v\in \mathbb{R}^n\right|\; v^\top d\geq 0 \mbox{ for all } d\in \mathcal{K}\right\}.
\end{equation}
The Guignard constraint qualification (GCQ) will be said to hold at $\bar x \in \mathcal{C}$ if the equality is preserved if the dual is applied on both sides of \eqref{eq:ACQ}; i.e., 
\begin{equation}\label{eq:GCQ}\tag{GCQ}
    \left(T_{\mathcal{C}}(\bar x)\right)^* =  \left(L_{\mathcal{C}}(\bar x)\right)^*.
\end{equation}
It is well-known that the GCQ is strictly weaker than the ACQ. Overall, in summary, at a given point $\bar x \in \mathcal{C}$, we have the following chain of implications:
\[
\eqref{eq:LICQ} \;\; \Longrightarrow \;\; \eqref{eq:MFCQ} \;\; \Longrightarrow \;\; \eqref{eq:ACQ} \;\; \Longrightarrow \;\; \eqref{eq:GCQ}.
\]

If $\bar x$ is a local optimal solution of problem \eqref{eq:min_standard} and the GCQ holds at $\bar x$, then we can find Lagrange multipliers $\alpha\in \mathbb{R}^p$ and $\beta\in \mathbb{R}^q$ such 
\begin{eqnarray}
  \nabla\mathfrak{f}(\bar x) +   \nabla\mathfrak{g}(\bar x)^\top \alpha + \nabla\mathfrak{h}(\bar x)^\top \beta =0, \label{OptCond1}\\[1ex]
         \alpha\geq 0, \;\; \mathfrak{g}(\bar x) \leq 0, \;\; \alpha^\top\mathfrak{g}(\bar x)=0,\label{OptCond2}\\[1ex]
         \mathfrak{h}(x)=0.\label{OptCond3}
\end{eqnarray}

To state a second order sufficient optimality condition for problem \eqref{eq:min_standard}, we introduce the Lagrangian function associated to the problem
\[
\ell(x, \alpha, \beta):= \mathfrak{f}(x) + \alpha^\top \mathfrak{g}(x) + \beta^\top \mathfrak{h}(x).
\] 
Now, let $x$ be such that we can find Lagrange multipliers $\alpha\in \mathbb{R}^p$ and $\beta\in \mathbb{R}^q$ such that the optimality conditions \eqref{OptCond1}--\eqref{OptCond3} are satisfied. If we assume that the second order sufficient condition (SOSC)
\begin{equation}\label{eq:SOSC}\tag{SOSC}
d^\top \nabla^2_{xx} \ell\left(x, \alpha, \beta\right)d > 0 \;\; \mbox{ for all }\;\; d\in \mathfrak{C}(x, \alpha)\setminus \{0\}
\end{equation}
holds, then $x$ is a strict local optimal solution of problem \eqref{eq:min_standard}; see, e.g., \cite{shetty1993nonlinear} for more details on \eqref{eq:SOSC}. Note that here, $\mathfrak{C}(x, \alpha)$ denotes the 
critical cone associated to  \eqref{eq:min_standard}, and which is defined by
\begin{equation}\label{eq:Critial_Cone}
   \mathfrak{C}(x, \alpha) :=\left\{d\in \mathbb{R}^n\left|~\begin{array}{ll}
    \nabla\mathfrak{g}_i(x)^\top d = 0\;\; \forall i:\;\, \mathfrak{g}_i(x)=0, \;\; \alpha_i >0\\[2ex]
    \nabla\mathfrak{g}_i(x)^\top d\leq 0\;\; \forall i:\;\, \mathfrak{g}_i(x)=0, \;\; \alpha_i =0\\[2ex]
     \nabla\mathfrak{h}_j(x)^\top d = 0\;\; \forall j=1, \ldots, q
\end{array} \right.\right\}.  
\end{equation}

To close this section, consider the following Wolfe dual of problem \eqref{eq:min_standard}:
\begin{equation}\label{eq:min_standard_dual}%
    \begin{array}{rl}
    \underset{x,\, \alpha,\, \beta}\max & \ell(x, \alpha, \beta)\\[2ex]
    \mbox{s.t.} &  \nabla_x \ell(x, \alpha, \beta)=0, \;\; \alpha\geq 0.
    \end{array}
\end{equation}
Then, we can state the following strong Wolfe duality result from \cite{wolfe1961duality,fletcher2000practical}:
\begin{lemma}\label{eq:Wolfe_Dual_P} For problem \eqref{eq:min_standard}, let the functions $\mathfrak{f}$ and $\mathfrak{g}_i$, for $i=1, \ldots, p$, be convex, while $\mathfrak{h}$ is an affine linear function. Furthermore, let the point $x$ be an optimal solution of problem \eqref{eq:min_standard} that satisfies \eqref{eq:GCQ}. Then, there exist Lagrange multipliers $\alpha\in \mathbb{R}^p$ and $\beta\in \mathbb{R}^q$ such that $\left(x, \alpha, \beta\right)$ is an optimal solution of problem \eqref{eq:min_standard_dual} and it holds that $\mathfrak{f}(x) = \ell\left(x, \alpha, \beta\right)$.
\end{lemma}

\subsection{Parametric and minmin optimization}
The focus of this subsection will be on the \textit{minmin} optimization problem 
\begin{equation}\label{eq:MINMIN}\tag{\text{$\mathcal{P}_{2m}$}}
      \underset{x\in \mathcal{X}}{\min}~\underset{y\in \mathcal{Y}(x)}{\min}\mathfrak{f}(x, y),
\end{equation}
which involves an \textit{outer} (resp. \textit{inner}) minimization w.r.t. to the outer (resp. inner) variable $x\in \mathbb{R}^n$ (resp.  $y\in \mathbb{R}^m$).   $\mathfrak{f} : \mathbb{R}^n \times \mathbb{R}^m \rightarrow \mathbb{R}$ represents the objective function of \eqref{eq:MINMIN}, while $\mathcal{X} \subseteq \mathbb{R}^n$ corresponds to the outer feasible set and the set-valued mapping $\mathcal{Y} :\mathbb{R}^n \rightrightarrows \mathbb{R}^m$ describes the inner feasible set.

 Considering the inner problem in \eqref{eq:MINMIN}, which is obviously a parametric optimization problem (in the outer variable), two objects will play an important role in our analysis. That is, we need the optimal solution set-valued mapping $\mathcal{S}: \mathbb{R}^n \rightrightarrows \mathbb{R}^m$:
 \begin{equation}\label{eq:mathcal-S}
  \mathcal{S}(x):=\underset{\qquad y\in \mathcal{Y}(x)}{\arg\min}\mathfrak{f}(x, y)   
 \end{equation}
and the corresponding optimal value function defined by 
\[
\phi(x):=\underset{y\in \mathcal{Y}(x)}{\min}\mathfrak{f}(x, y).
\]
Based on this concept, note that problem \eqref{eq:MINMIN} can be equivalently written as 
\[
\underset{x\in \mathcal{X}}{\min}~\phi(x).
\]
Hence, throughout this section we will use the following concepts of solution:
\begin{definition}\label{def:SolutionConcept_Pp} A point $\bar x\in \mathcal{X}$ will be said to be a local optimal solution of problem \eqref{eq:MINMIN} if there exists a neighborhood $U$ of $\bar x$ such that 
\[
\phi(\bar x) \leq \phi(x) \;\; \mbox{ for all }\; x\in X\cap U. 
\]
As usual, if $U=\mathbb{R}^n$, then the point is a global optimal solution.
\end{definition}

Next, we introduce the single-min operator problem associated to problem \eqref{eq:MINMIN}:
\begin{equation}\label{eq:MIN}\tag{\text{$\mathcal{P}_{1m}$}}
    \begin{array}{rl}
    \underset{x, y}{\min}&\mathfrak{f}(x, y)\\[1ex]
    \mbox{s.t.} & x\in \mathcal{X}, \;\; y\in \mathcal{Y}(x).
    \end{array}
\end{equation}
Denote by $\Omega:=\left\{(x, y)\in \mathbb{R}^n\times \mathbb{R}^m\left|\;x\in \mathcal{X},\; y\in \mathcal{Y}(x)\right.\right\}$; $(\bar x, \bar y)$ will be said to be a local optimal solution of \eqref{eq:MIN} if there exists a neighborhood $W$ of $(\bar x, \bar y)$ such that 
\[
\mathfrak{f}(\bar x, \bar y) \leq \mathfrak{f}(x, y) \;\; \mbox{ for all }\; (x, y)\in \Omega \cap W.
\]
Similarly, if $W=\mathbb{R}^n\times \mathbb{R}^m$, then $(\bar x, \bar y)$ is a global optimal solution of problem \eqref{eq:MIN}. Since $\phi$ is the global minimal value function of the inner problem, we will only be interested in local optimal solutions $(\bar x,\bar y)$ of problem \eqref{eq:MIN} with $\bar y\in \mathcal{S}(\bar x)$ when we compare local optimal solutions of problem \eqref{eq:MIN} with those of problem \eqref{eq:MINMIN} (cf.  Lemma~\ref{thm:LocalRelationship}).

Next, we state the global relationship between problems \eqref{eq:MINMIN} and \eqref{eq:MIN}.
\begin{lemma}\label{minmin_min_rel}The following statements hold true:
\begin{itemize}
    \item[(a)] Let $\bar x$ be a global optimal solution of problem \eqref{eq:MINMIN}. Then, for all $\bar y\in \mathcal{S}(\bar x)$, the point $(\bar x, \bar y)$ is a global optimal solution of problem \eqref{eq:MIN}. 
    \item[(b)] Let $(\bar x, \bar y)$ be globally optimal for problem \eqref{eq:MIN}. %
    Then $\bar x$ is a global optimal solution of  \eqref{eq:MINMIN}. 
\end{itemize}
\end{lemma}
\begin{proof}
(a) For any $\bar y \in \mathcal{S}(\bar x)$ and any couple $(x, y)$ such that $x\in \mathcal{X}$ and $y\in \mathcal{Y}(x)$, 
\[
\mathfrak{f}(\bar x, \bar y)=\phi(\bar x) \leq \phi(x)\leq \mathfrak{f}(x, y). 
\]
Hence, $(\bar x, \bar y)$ is a global optimal solution of problem \eqref{eq:MIN}. 

As for (b), first note that $(\bar x, \bar y)$ being a global optimal solution of  problem \eqref{eq:MIN}, we automatically have $\bar y\in \mathcal{S}(\bar x)$. Otherwise, we can find $\tilde{y}\in \mathcal{Y}(\bar x)$ such that 
\[
\mathfrak{f}(\bar x, \bar y) > \mathfrak{f}(\bar x, \tilde{y}).
\]
Note that $(\bar x, \tilde{y})$ is a feasible point to problem \eqref{eq:MIN}. Hence, we have a contradiction.
It therefore follows that for any $x\in \mathcal{X}$ and $y\in \mathcal{Y}(x)$, we have 
\begin{equation}\label{eq:IneqIm}
    \phi(\bar x)= \mathfrak{f}(\bar x, \bar y) \leq \mathfrak{f}(x, y),
\end{equation}
considering the fact that $\bar y\in \mathcal{S}(\bar x)$. Then given that 
\[
\phi(x)=\left\{\begin{array}{ll}
    +\infty &\mbox{if}\quad  \mathcal{S}(x)=\emptyset,\\[2ex]
     f(x, y^*) \;\mbox{(for some } y^*\in \mathcal{S}(x) \subset \mathcal{Y}(x))\, & \mbox{otherwise,}
\end{array}\right.
\]
it follows from \eqref{eq:IneqIm} and the arbitrary choice of $y\in \mathcal{Y}(x)$ that 
\[
\phi(\bar x)= \mathfrak{f}(\bar x, \bar y) \leq \phi(x).
\]
Therefore, $\bar x$ is a global optimal solution of problem \eqref{eq:MINMIN}.  \qed
\end{proof}

To establish the local relationship between the two problems, we need the inner semicontinuity of the inner optimal solution set-valued mapping $\mathcal{S}$. So, $\mathcal{S}$ will be said to be inner semicontinuous at a point $(\bar x, \bar y)\in \text{gph}\,\mathcal{S}$ if for every sequence $x^k \rightarrow \bar x$, there exists a sequence $y^k \in \mathcal{S}(x^k)$ such that $y^k \rightarrow \bar y$.  Note that for any given set-valued mapping $\Psi$, $(x, y)\in \text{gph}\,\Psi$ iff $y\in \Psi(x)$. It is worth to mention that the concept of inner semicontinuity holds if the corresponding set-valued mapping is lower semicontinuous in the usual sense; see, e.g., \cite{dempe2007new} for relevant discussion, some references, and some sufficient conditions that ensure the fulfillment of lower semicontinuity of set-valued mappings that are relevant to  $\mathcal{S}$, as described by \eqref{eq:mathcal-S}. 

\begin{lemma}\label{thm:LocalRelationship}The following statements hold true:
\begin{itemize}
    \item[(a)] Let $\bar x$ be a local optimal solution of problem \eqref{eq:MINMIN}. Then, for all $\bar y\in \mathcal{S}(\bar x)$, the point $(\bar x, \bar y)$ is a local optimal solution of problem \eqref{eq:MIN}. 
    \item[(b)] Let the point  $(\bar x, \bar y)\in \text{gph}\,\mathcal{S}$, where $\mathcal{S}$ is inner semicontinuous, be a local optimal solution of   \eqref{eq:MIN}. Then $\bar x$ is a local optimal solution of problem \eqref{eq:MINMIN}. 
\end{itemize}
\end{lemma}
\begin{proof}(a) Assume that there is some $\bar y\in \mathcal{S}(\bar x)$ such that $(\bar x, \bar y)$ is not a local optimal solution of \eqref{eq:MIN}. Then we can find a sequence $(x^k, y^k)$ from $\Omega$ with $x^k \rightarrow \bar x$ and $y^k \rightarrow \bar y$ such that
\[
\mathfrak{f}(x^k, y^k) < \mathfrak{f}(\bar x, \bar y) = \phi(\bar x) \;\mbox{ for all }\, k.
\]
Then considering the definition of $\phi$, it follows that 
\[
\phi(x^k) \leq \mathfrak{f}(x^k, y^k) < \mathfrak{f}(\bar x, \bar y) = \phi(\bar x) \;\mbox{ for all }\, k.
\]
Therefore, contradicting the fact that the point $\bar x$ is a local optimal solution of problem \eqref{eq:MINMIN}, given that  we have $x^k\in \mathcal{X}$ for all $k$.

(b) If $\bar x$ is not a local optimal solution of problem \eqref{eq:MINMIN}, then we can find a feasible sequence $x^k\rightarrow \bar x$ such that $\phi(\bar x) > \phi(x^k)$ for all $k$. As the set-valued mapping $\mathcal{S}$ is inner semicontinuous at $(\bar x, \bar y)$, we can find a sequence $y^k\in \mathcal{S}(x^k)$ that converges to $\bar y$. It follows by the construction that 
\[
\mathfrak{f}(\bar x, \bar y) = \phi(\bar x) > \phi(x^k) = \mathfrak{f}\left(x^k, y^k\right)
\]
with $x^k\in X$, $y^k\in \mathcal{Y}(x^k)$ for all $k$. This implies that  $(\bar x, \bar y)$ is not  locally optimal for \eqref{eq:MIN}.  \qed
\end{proof}
The proof of Lemma \ref{thm:LocalRelationship} follows along the lines of \cite[Theorem 6.9]{dempe2012sensitivity}, in the context of the link between the original and standard optimistic bilevel programs. But we include it here for completeness. Also consider the following example based on Example 6.10 from the latter reference: 
\[
\mathfrak{f}(x, y):=x, \;\; \mathcal{X}:=[-1, \; 1],\; \mbox{ and }\; \mathcal{Y}(x):=\left\{\begin{array}{lll}
  [0,\, 1] & \mbox{ if }  & x=0,  \\
   \{0\} & \mbox{ if }  & x > 0,  \\
   \{1\} & \mbox{ if }  & x < 0.  \\
\end{array}\right.
\]
We can easily check that the point $(0, 0)$ is a local optimal solution of the corresponding problem \eqref{eq:MIN}, while $0$ is not a local optimal solution of problem \eqref{eq:MINMIN}.  Furthermore, as $\mathcal{S}(x)=\mathcal{Y}(x)$ for all $x\in \mathcal{X}$, $\mathcal{S}$ is not inner semicontinuous at $(0, 0)$. This confirms the importance of the inner semicontinuity assumption in part (b) of Lemma \ref{thm:LocalRelationship}.

\section{Duality and new single-level reformulations}\label{sec:Single-level-Reform} %
In this section, we will introduce a true single-level reformulation for \eqref{eq:Pessimistic-Problem} and establish suitable global and local relationships. As mentioned in Section~\ref{sec:Intro}, by \textit{true} we mean that the reformulation is a standard nonlinear optimization problem involving neither complementarity constraints nor value functions, and such that classical constraint qualifications hold in solution points. Throughout the remainder of this paper, in addition to Assumption~\ref{ass:Sp_domain},  we will also consider Assumption~\ref{Assumption1} a blanket assumption.

\subsection{A dual reformulation of the intermediate problem} 
For the subsequent analysis it is crucial to note that the pessimistic bilevel problem \eqref{eq:Pessimistic-Problem} actually possesses a three-level structure. In the lower-level problem \eqref{eq:LL} the function $f(x,\cdot)$ is minimized over $Y(x)$, in the intermediate-level problem \eqref{eq:IL} the function $F(x,\cdot)$ is maximized over $S_L(x)$, while in the upper-level the function $\varphi_p$ from \eqref{eq:def_phi_p}
is minimized over $X$. The main idea leading to single-level reformulations will, under appropriate convexity assumptions, be based on a dual description of the maximal value $\varphi_p(x)$ of \eqref{eq:IL} for each fixed $x\in X$ and, hence, of the upper-level problem's objective function $\varphi_p$. To this end, we write the intermediate problem in its equivalent optimal-value function formulation with $Z(x)$ given in \eqref{eq:Zx},
\begin{equation}\label{eq:ILv}\tag{\text{$IL^v(x)$}}
    \max_y\,F(x,y)\ \text{ s.t. }\ y\in Z(x)=\left\{y\in Y(x)\left| \; f(x,y)-\varphi_L(x)\leq0\right.\right\}.
\end{equation}

Since for fixed $x\in X$ also $\varphi_L(x)$ is a constant, the above optimal-value function formulation is rather an ``optimal value formulation'', which reflects the structure of a so-called  \textit{simple bilevel program}; see, e.g., \cite{shehu2021inertial}, for an overview on the subject. This shall promote a beneficial structure of the single-level reformulation to be introduced below. We emphasize, however, that for each $x\in X$ the MFCQ is violated everywhere in $Z(x)$. Indeed, in view of $Z(x)=S_L(x)$, each $y\in Z(x)$ is a minimal point of the lower-level problem and, thus, satisfies the necessary optimality condition of Fritz-John. The latter prevents the MFCQ from holding at $y$.

Since Lemma~\ref{eq:Wolfe_Dual_P} does not need the MFCQ but holds under the weaker GCQ, it still makes sense to consider the Wolfe dual \eqref{eq:min_standard_dual} to \eqref{eq:ILv} for fixed $x\in X$. Indeed, this Wolfe dual consists in the minimization of the Lagrangian
\begin{equation}\label{eq:LLpD}
    \mathcal{L}^L_{p_D}(x, y, u, v, w) := F(x, y) - u^\top g(x, y) - v^\top h(x, y)- w(f(x, y) - \varphi_L(x)),
\end{equation}
defined from $\mathbb{R}^n\times \mathbb{R}^m\times \mathbb{R}^p\times \mathbb{R}^q\times \mathbb{R}$ to $\mathbb{R}$, over the set
\[
\Lambda_{p_D}(x):=  \left\{(y, u, v, w)\in \mathbb{R}^m\times \mathbb{R}^p \times \mathbb{R}^q\times \mathbb{R}\,\left|\begin{array}{l}
u\geq 0, \;\; w\geq 0\\[2ex]
         \nabla_y \mathcal{L}^L_{p_D} (x, y, u, v, w)=0
    \end{array} \right. \right\}.
\]
By
\[
\varphi_{p_D}(x):=\underset{(y, u, v, w)\in \Lambda_{p_D}(x)}{\min}~\mathcal{L}^L_{p_D}(x, y, u, v, w),
\]
we denote the minimal value function of the Wolfe dual, defined on $X$. Corresponding to the problem \eqref{eq:Pessimistic-Problem}, where $\varphi_p$ is minimized over $X$, we introduce the problem 
\begin{equation}\label{eq:P_pD}\tag{\text{$P_{p_D}$}}
    \underset{x\in X}{\min}~\varphi_{p_D}(x).
\end{equation}
For $x\in X$,
we denote the set of minimal points associated to  $\varphi_{p_D}(x)$
by
\begin{align}\label{eq:SLpD}
S^L_{p_D}(x) := \underset{\quad\;\,(y, u, v, w)\in \Lambda_{p_D}(x)}{\arg\min}~\mathcal{L}^L_{p_D}(x, y, u, v, w).
\end{align}

Observe that the Lagrangian $\mathcal{L}^L_{p_D}$ from \eqref{eq:LLpD} relies on the lower-level value function $\varphi_L$, which is implicitly defined. Considering potential challenges that could arise in its algorithmic treatment, and to arrive at a true single-level reformulation, let us also introduce the Lagrangian-type real-valued function $\mathcal{L}_{p_D}$ defined by 
\[
\begin{array}{l}
   \forall (x, y, z, u, v, w)\in \mathbb{R}^n\times \mathbb{R}^m\times \mathbb{R}^m\times \mathbb{R}^p\times \mathbb{R}^q\times \mathbb{R}: \\[2ex]
     \mathcal{L}_{p_D}(x, y, z, u, v, w) := F(x, y) - u^\top g(x, y) - v^\top h(x, y) - w\left(f(x, y) - f(x, z)\right).
\end{array}
\]
We can easily observe 
that for any quintuple  $(x, y, u, v, w)\in \mathbb{R}^n\times \mathbb{R}^m\times \mathbb{R}^p \times \mathbb{R}^q\times \mathbb{R}$ with $w\geq 0$, the Lagrangian-type function $\mathcal{L}^L_{p_D}$ can be rewritten as
\begin{align}\label{eq:LminL}
\mathcal{L}^L_{p_D}(x, y, u, v, w) = \underset{z\in Y(x)}{\min} \mathcal{L}_{p_D}(x, y, z, u, v, w),
\end{align}
and that
\begin{align}
 \nabla_y\mathcal{L}^L_{p_D}(x, y, u, v, w) &=  \nabla_y\mathcal{L}_{p_D}(x, y, z, u, v, w)\label{eq:LminLgrad}\\
 &= \nabla_y F(x, y) - \nabla_y g(x, y)^\top u - \nabla_y h(x, y)^\top v- w \nabla_y f(x, y)\nonumber
\end{align}
holds for any $(x, y, z, u, v, w)\in \mathbb{R}^n\times \mathbb{R}^m\times \mathbb{R}^m\times \mathbb{R}^p\times \mathbb{R}^q\times \mathbb{R}$.

The following lemma will play a crucial role in building the relationships between the problems \eqref{eq:Pessimistic-Problem}, \eqref{eq:P_pD}, and the single-level reformulations \eqref{eq:pSLR}, and \eqref{eq:tSLR} to be introduced in Section~\ref{sec:single-level-reformulations}. 
\begin{lemma}\label{lem:phi_p=phi_pD} %
For all $x\in X$, it holds that %
\begin{equation}\label{eq:phi_pD}
  \varphi_p(x)=\varphi_{p_D}(x) = \underset{(y, u, v, w)\in \Lambda_{p_D}(x)}{\min}~\underset{z\in Y(x)}{\min}\mathcal{L}_{p_D}(x, y, z, u, v, w).
\end{equation}
\end{lemma}
\begin{proof}
    Under Assumptions \ref{ass:Sp_domain} and \ref{Assumption1}, for each $x\in X$, \eqref{eq:ILv} possesses an optimal point $y\in S_p(x)$ which is also a KKT point with corresponding multipliers $(u,v,w)$. By the strong Wolfe duality result from Lemma \ref{eq:Wolfe_Dual_P}, this yields the minimality of $(y,u,v,w)$ for the Wolfe dual and, thus, we have $\varphi_p(x)=F(x,y)= \mathcal{L}^L_{p_D}(x, y, u, v, w) =\varphi_{p_D}(x)$.
The result therefore holds in view of \eqref{eq:LminL}.   \qed
\end{proof} 
\begin{remark}\label{rem:dual_y_not_optimal}
    The proof of Lemma~\ref{lem:phi_p=phi_pD} uses the fact that each KKT point $y$ of  \eqref{eq:ILv} with corresponding multipliers $(u,v,w)$ generates a dually optimal point $(y,u,v,w)\in S^L_{p_D}(x)$. We emphasize that \emph{not all} elements of $S^L_{p_D}(x)$ need to correspond to KKT points of  \eqref{eq:ILv}, since the primal feasibility condition $y\in Z(x)$ is not part of the definition of the Wolfe dual. In particular, for an optimal point $(y,u,v,w)$ of the Wolfe dual, the point $y$ need not lie in $S_p(x)$ unless $y$ is primally feasible. Nor do $u$, $v$, $w$ need to satisfy complementary slackness conditions with the corresponding constraint functions.
\end{remark}
Under the assumptions of Lemma \ref{lem:phi_p=phi_pD} the problems \eqref{eq:Pessimistic-Problem} and \eqref{eq:P_pD} are globally and locally equivalent, with the understanding that solution concepts for \eqref{eq:P_pD} are analogous to those of \eqref{eq:Pessimistic-Problem} (see Definition \ref{def:local_optimal_sol} or Definition \ref{def:SolutionConcept_Pp} for a general framework for such an optimal solution notion). 
While, in particular, for each $x\in X$, the set $\Lambda_{p_D}(x)$ is nonempty, we emphasize that it is also unbounded. Indeed, since the MFCQ is violated everywhere in $Z(x)$, by \cite{gauvin1977necessary} the set of Lagrange multipliers corresponding to $y\in S_p(x)$ is empty or unbounded, where the first alternative is ruled out by the assumption of the GCQ at $y$. Since the set of $y\in S_p(x)$ with corresponding Lagrange multipliers forms a subset of $\Lambda_{p_D}(x)$, also the latter is unbounded. This unboundedness is inherited by the feasible sets of the single-level problems \eqref{eq:pSLR} and \eqref{eq:tSLR} introduced in the subsequent subsection (which, of course, does not entail that also the objective functions of these problems are unbounded on the respective feasible sets).

\subsection{Single-level reformulations}\label{sec:single-level-reformulations}
Next, we introduce two single-level optimization problems associated to the pessimistic bilevel optimization problem \eqref{eq:Pessimistic-Problem}. We start with the preliminary single-level reformulation defined by 
\begin{equation}\label{eq:pSLR}\tag{pSLR}
\begin{array}{rl}
   \underset{x, y, u, v, w}{\min} & \mathcal{L}^L_{p_D}(x, y, u, v, w) \\[1ex]
     \mbox{s.t.} & x\in X, \;\, u\geq 0, \;\, w\geq 0,\\[1ex]
                 &\nabla_y F(x, y) - \nabla_y g(x, y)^\top u - \nabla_y h(x, y)^\top v- w \nabla_y f(x, y) =0
\end{array}
\end{equation}
in whose objective function the implicitly defined and algorithmically potentially challenging function $\varphi_L$ appears.
For the latter reason, we also introduce the true single-level reformulation 
\begin{equation}\label{eq:tSLR}\tag{tSLR}
\begin{array}{rl}
   \underset{x, y, z, u, v, w}{\min} & \mathcal{L}_{p_D}(x, y, z, u, v, w) \\[1ex]
     \mbox{s.t.} & x\in X, \;\, g(x, z)\leq 0, \;\, h(x, z)=0, \;\, u\geq 0, \;\, w\geq 0,\\[1ex]
                 &\nabla_y F(x, y) - \nabla_y g(x, y)^\top u - \nabla_y h(x, y)^\top v- w \nabla_y f(x, y) =0.
\end{array}
\end{equation}

We remark that, while the last equality constraint in problem \eqref{eq:tSLR} originates from the condition $\nabla_y \mathcal{L}^L_{p_D} (x, y, u, v,w)=0$, in view of \eqref{eq:LminLgrad} it may as well be written as $\nabla_y \mathcal{L}_{p_D} (x, y, z, u, v,w)=0$. 
Hence, for each fixed $(x,z,u,v,w)$, $y$ is a critical point of $\mathcal{L}_{p_D}(x,\cdot,z,u,v,w)$. Since this function is concave, $y$ is even a global maximal point of $\mathcal{L}_{p_D}(x,\cdot,z,u,v,w)$ over $\R^m$. Therefore, as in the original Wolfe duality argument, the objective function of problem \eqref{eq:tSLR} may be replaced by $\max_{y\in\R^m}\mathcal{L}_{p_D}(x,y,z,u,v,w)$, and the last equality constraint of problem \eqref{eq:tSLR} may instead be dropped. Under appropriate additional linearity assumptions on the defining functions, the variable $y$ can indeed be eliminated completely from problem \eqref{eq:tSLR}; see Section~\ref{sec:LP}.

Next, we provide links between problems \eqref{eq:P_pD} and \eqref{eq:pSLR}.

\begin{theorem}\label{them:LocalRelationship}
It holds that:
\begin{itemize}
    \item[(a)] If $\bar x$ is globally optimal for problem \eqref{eq:P_pD}, then for all $(\bar y, \bar u, \bar v, \bar w)\in S^L_{p_D}(\bar x)$, the point $(\bar x, \bar y, \bar u, \bar v, \bar w)$ is globally optimal for problem \eqref{eq:pSLR}. Conversely, if $(\bar x, \bar y, \bar u, \bar v, \bar w)$ is a global optimal solution of \eqref{eq:pSLR}, then $\bar x$ is a globally optimal for problem \eqref{eq:P_pD}.
    \item[(b)] If $\bar x$ is a local optimal solution of  \eqref{eq:P_pD}, then for all $(\bar y, \bar u, \bar v, \bar w)\in S^L_{p_D}(\bar x)$, the point  $(\bar x, \bar y, \bar u, \bar v, \bar w)$ is a local optimal solution of \eqref{eq:pSLR}. Conversely, let $(\bar x, \bar y, \bar u, \bar v, \bar w)$, where $S^L_{p_D}$ from \eqref{eq:SLpD} is inner semicontinuous, be a local optimal solution of  \eqref{eq:pSLR}, then $\bar x$ is a local optimal solution of \eqref{eq:P_pD}.
\end{itemize}
\end{theorem}
\begin{proof}
    (a) follows from Lemma \ref{minmin_min_rel}, while (b) results from Lemma \ref{thm:LocalRelationship}.    \qed
\end{proof}
Subsequently, we have the following link between problems \eqref{eq:pSLR} and \eqref{eq:tSLR}. To proceed, we introduce the set-valued mapping $S^*_L : \mathbb{R}^n\times \mathbb{R} \rightrightarrows \mathbb{R}^m$ defined by
\begin{equation}\label{eq:wS_L(x)}
    S^*_L(x, w) :=\underset{\qquad z\in Y(x)}{\arg\min}  w f(x, z).
\end{equation}
Obviously, $S^*_L(x, w) = S_L(x)$ if $w>0$ and $S^*_L(x, w) = Y(x)$ if $w=0$.
\begin{theorem}
It holds that:
\begin{itemize}
    \item[(a)] If $(\bar x, \bar y, \bar u, \bar v, \bar w)$ is globally optimal for \eqref{eq:pSLR}, then for all $\bar z\in S^*_L(\bar x, \bar w)$, the point $(\bar x, \bar y, \bar z, \bar u, \bar v, \bar w)$ is globally optimal for  \eqref{eq:tSLR}. Conversely, if  $(\bar x, \bar y, \bar z, \bar u, \bar v, \bar w)$ is a global optimal solution of problem \eqref{eq:tSLR}, then $(\bar x, \bar y, \bar u, \bar v, \bar w)$ is a global optimal solution of  \eqref{eq:pSLR}.
    \item[(b)] If $(\bar x, \bar y, \bar u, \bar v, \bar w)$ is locally optimal for \eqref{eq:pSLR}, then for all $\bar z\in S^*_L(\bar x, \bar w)$, $(\bar x, \bar y, \bar z, \bar u, \bar v, \bar w)$ is locally optimal for \eqref{eq:tSLR}. Conversely, let  $(\bar x, \bar y, \bar z, \bar u, \bar v, \bar w)$, with  $S^*_L$  inner semicontinuous at $(\bar x, \bar w, \bar z)$, be locally optimal for \eqref{eq:tSLR}, then $(\bar x, \bar y, \bar u, \bar v, \bar w)$ is locally optimal for \eqref{eq:pSLR}.
\end{itemize}
\end{theorem}
\begin{proof}
Problem \eqref{eq:pSLR} is globally and locally equivalent to the problem
\begin{equation}\label{eq:pSLR2}\tag{pSLR2}
 \underset{(x, y, u, v)\in \Omega^L}{\min}~\underset{z\in Y(x)}{\min}\mathcal{L}_{p_D}(x, y, z, u, v, w),
\end{equation}
where the concept of optimal solution is understood in the same sense as in Definition \ref{def:SolutionConcept_Pp}, 
given that we have $\mathcal{L}^L_{p_D}(x, y, u, v, w) = \underset{z\in Y(x)}{\min}\mathcal{L}_{p_D}(x, y, z, u, v, w)$. Note that the outer feasible set $\Omega^L$ in problem \eqref{eq:pSLR2} is given by 
\[
\Omega^L:=\left\{(x, y, u, v, w)\left|~x\in X, \;   (y, u, v, w)\in \Lambda_{p_D}(x)\right.\right\}.
\]
Subsequently, considering the fact that 
\begin{equation}\label{eq:SL=SLpD}
    S^*_L(x, w):=\underset{z\in Y(x)}{\arg\min}~\mathcal{L}_{p_D}(x, y, z, u, v, w),
\end{equation}
(a) and (b) follow from Lemma \ref{minmin_min_rel} and Lemma \ref{thm:LocalRelationship}, respectively. \qed
\end{proof}

\begin{corollary}\label{cor:summary_relationships}
It holds that:
\begin{itemize}
    \item[(a)] If the point $\bar x$ is globally optimal for problem \eqref{eq:Pessimistic-Problem}, then for all $(\bar y, \bar u, \bar v, \bar w)\in S^L_{p_D}(\bar x)$ and $\bar z\in S^*_L(\bar x, \bar w)$, the point  $(\bar x, \bar y, \bar z, \bar u, \bar v, \bar w)$ is globally optimal for \eqref{eq:tSLR}. Conversely, if the point $(\bar x, \bar y, \bar z, \bar u, \bar v, \bar w)$ is globally optimal for \eqref{eq:tSLR}, then $\bar x$ is globally optimal  \eqref{eq:Pessimistic-Problem}.
    \item[(b)] If $\bar x$ is locally optimal for \eqref{eq:Pessimistic-Problem}, then for all $(\bar y, \bar u, \bar v, \bar w)\in S^L_{p_D}(\bar x)$ and $\bar z\in S^*_L(\bar x, \bar w)$, the point  $(\bar x, \bar y, \bar z, \bar u, \bar v, \bar w)$ is locally optimal for \eqref{eq:tSLR}. Conversely, let  $(\bar x, \bar y, \bar z, \bar u, \bar v, \bar w)$, which is such that $S^L_{p_D}$ (resp. $S^*_L$) is inner semicontinuous at $(\bar x, \bar y, \bar u, \bar v, \bar w)$ (resp. $(\bar x, \bar w, \bar z)$), be locally optimal for \eqref{eq:tSLR}, then $\bar x$ is a local optimal solution of  problem \eqref{eq:Pessimistic-Problem}.
\end{itemize}
\end{corollary}

\begin{figure}[htp]%
    \centering
\begin{center}
\scalebox{0.75}{ %
\begin{tikzpicture}[
  node distance=35mm, %
  font=\large,
  >=Stealth,
  mynode/.style={draw,circle,minimum size=14mm,inner sep=2pt}
]

\node[mynode] (A) {$(P_p)$};
\node[mynode] (B) [right=of A] {$(P_{p_D})$};
\node[mynode] (C) [right=of B] {$(\mathrm{pSLR})$};
\node[mynode] (D) [right=of C] {$(\mathrm{tSLR})$};

\coordinate (ABmid) at ($(A)!0.5!(B)$);
\coordinate (BCmid) at ($(B)!0.5!(C)$);
\coordinate (CDmid) at ($(C)!0.5!(D)$);

\draw[<->,thick]
  (A.east) -- (B.west)
  node[pos=0.5, above] {$(a_0)$};

\draw[->,thick]
  (B.east) .. controls ($(BCmid)+(0,1.6)$) .. (C.west)
  node[pos=0.5, above] {$(a_1)$};

\draw[->,thick]
  (C.west) .. controls ($(BCmid)+(0,-1.6)$) .. (B.east)
  node[pos=0.5, below] {$(a_3)$};

\draw[->,thick]
  (C.east) .. controls ($(CDmid)+(0,1.6)$) .. (D.west)
  node[pos=0.5, above] {$(a_2)$};

\draw[->,thick]
  (D.west) .. controls ($(CDmid)+(0,-1.6)$) .. (C.east)
  node[pos=0.5, below] {$(a_4)$};

\end{tikzpicture}
}
\end{center}
\caption{Links between problems \eqref{eq:Pessimistic-Problem}, \eqref{eq:P_pD}, \eqref{eq:pSLR}, and \eqref{eq:tSLR} with assumptions $(a_1)$ and $(a_2)$ needed for both the corresponding global and local implications, while $(a_3)$ and $(a_4)$ are the required only  for the corresponding local relationships.}\label{fig:Diagram}
\end{figure}
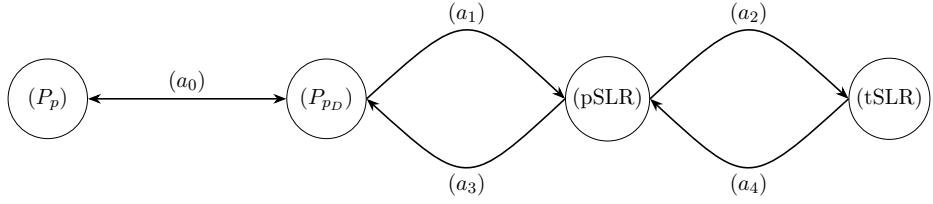

The relationships studied here are summarized in Fig. \ref{fig:Diagram}, where assumptions $(a_0)$, $(a_1)$, $(a_2)$, $(a_3)$, and $(a_4)$ are respectively defined as follows: 
\begin{itemize}
\item[$(a_0)$] Assumption \ref{Assumption1};
    \item[$(a_1)$]$(\bar x, \bar y, \bar u, \bar v, \bar w)$ is such that $(\bar y, \bar u, \bar v, \bar w)\in S^L_{p_D}(\bar x)$;
     \item[$(a_2)$] $(\bar x, \bar y, \bar z, \bar u, \bar v, \bar w)$ is such that $\bar z\in S^*_L(\bar x, \bar w)$;
    \item[$(a_3)$] $S^L_{p_D}$ is inner semicontinuous at $(\bar x, \bar y, \bar u, \bar v, \bar w)$;
    \item[$(a_4)$] $S^*_L$ is inner semicontinuous at $(\bar x, \bar w, \bar z)$.
\end{itemize}

Note that Assumption \ref{ass:Sp_domain} implies that $S^L_{p_D}(\bar x)\neq \emptyset$. So, the fulfillment of $(a_1)$ is not a problem, but the requirement rather emphasizes that for the first implications in Theorem \ref{them:LocalRelationship}(a) and (b), the point $(\bar x, \bar y, \bar u, \bar v, \bar w)$ has to be chosen in a specific way; i.e., such that $(\bar y, \bar u, \bar v, \bar w)\in S^L_{p_D}(\bar x)$. An analogous observation can be made for $(a_2)$, while considering the fact that for all $x\in X$, $S^*_L(x, w) = S_L(x)$ if $w>0$ and $S^*_L(x, w) = Y(x)$ if $w=0$, while $S_p(x) \subset S_L(x) \subset Y(x)$. As for $(a_3)$ and $(a_4)$, they are only needed for the corresponding local relationships. 

\begin{remark}\label{rem:z-in-SL}
    Due to the observation from Remark~\ref{rem:dual_y_not_optimal}, for a solution $(\bar x, \bar y, \bar z, \bar u, \bar v, \bar w)$ of \eqref{eq:tSLR} Corollary~\ref{cor:summary_relationships} neither implies optimality of $\bar y$ for the intermediate problem \eqref{eq:IL} nor can optimality of $\bar z$ for the lower-level problem \eqref{eq:LL} be expected. However, once $\bar x$ has been computed from a solution of \eqref{eq:tSLR}, a corresponding optimal point $\bar z$ can be generated by solving the convex lower-level problem $(LL(\bar x))$, and the solution of the simple bilevel problem $(IL(\bar x))$ yields a corresponding optimal point $\bar y$ (cf. \cite{shehu2021inertial} and the references therein for algorithmic approaches). 
\end{remark}

\subsection{The linear case with respect to the lower-level variable}\label{sec:LP}
This section illustrates how the above constructions can be streamlined under appropriate linearity assumptions. Indeed, we shall employ the following assumption, where the occurring functions $c$, $\gamma$, $A$, $b$ and $d$ map from $X$ to spaces of appropriate dimensions, respectively.

\begin{assumption}\label{ass:polyhedral} It holds that:
    \begin{itemize}
        \item[$(1)$] $\forall x\in X,\ y\in\R^m:\ F(x,y)=c(x)^\top y+\gamma(x)$.
        \item[$(2)$] $\forall x\in X,\ y\in\R^m:\ g(x,y)=A(x)y-b(x)$.
        \item[$(3)$] $\forall x\in X,\ y\in\R^m:\ f(x,y)=d(x)^\top y$.
    \end{itemize}
\end{assumption}
Possibly present equality constraints in the description of $Y(x)$ are also assumed to be affine-linear with respect to $y$ like in Assumption~\ref{Assumption1}(4), but here they may be subsumed in the system $A(x)y\leq b$. Observe that, due the the linearity assumptions, for each $x\in X$ \eqref{eq:ACQ} holds at each $y\in Z(x)$ so that, altogether, Assumption~\ref{ass:polyhedral} implies Assumption~\ref{Assumption1}. Hence, with
\begin{align*}
\mathcal{L}_{p_D}(x, y, z, u, w) = c(x)^\top y+\gamma(x)-u^\top(A(x)y-b(x))-w\,d(x)^\top(y-z)
\end{align*}
and
\begin{align*}
\nabla_y\mathcal{L}_{p_D}(x, y, z, u, w) = c(x)-A(x)^\top u-w\,d(x)
\end{align*}
the true single-level reformulation \eqref{eq:tSLR} reads
\begin{equation*}
   \begin{array}{rl}
   \underset{x, y, z, u, w}{\min} & \mathcal{L}_{p_D}(x, y, z, u, w) \\[1ex]
     \mbox{s.t.} & x\in X, \;\, g(x, z)\leq 0, \;\, u\geq 0, \;\, w\geq 0,\\[1ex]
                 & \nabla_y\mathcal{L}_{p_D}(x, y, z, u, w)=0.
\end{array}
\end{equation*}
In analogy to the fact that Wolfe duality collapses to linear programming duality in the polyhedral case, the equality constraint of this problem can be used to simplify the objective function. This finally results in the problem
\begin{equation}\label{eq:tSLRLP}\tag{tSLR-LP}
   \begin{array}{rl}
   \underset{x, z, u, w}{\min} & \gamma(x)+u^\top b(x)+w\,d(x)^\top z\\[1ex]
     \mbox{s.t.} & x\in X, \;\, A(x)z\leq b(x), \;\, u\geq 0, \;\, w\geq 0,\\[1ex]
                 & c(x)-A(x)^\top u-w\,d(x)=0,
\end{array}
\end{equation}
in which the dependence on $y$ has been eliminated. Replacing the term $w\,d(x)$ in the objective function by another application of the equality constraint yields the alternative expression 
\[
(c(x)^\top z+\gamma(x))-u^\top (A(x)z-b(x))
\]
for the objective function of problem \eqref{eq:tSLRLP}. 
\begin{remark}
A special case of  Assumption~\ref{ass:polyhedral} is complete linearity, i.e., the affine-linearity of $F$, $g$ and $f$ in $(x,y)$. This happens for constant functions $c$, $A$, and $d$, and affine-linear $\gamma$ and $b$. In this case, and if the continuous relaxation of $X$ is a polyhedral set, the relaxed feasible set of \eqref{eq:tSLRLP}  is polyhedral as well. Since the objective functions of  \eqref{eq:tSLRLP} is a sum of linear and bilinear terms, the problem is nonconvex quadratic. This means that every linear pessimistic bilevel program may be rewritten as a nonconvex quadratic program.
Furthermore, observe choosing $X$ as in \eqref{eq:X-and-Y(x)-Linear}, as well as 
\[
A(x):=A,\;\; b(x):=-Bx + b, \;\; c(x):=x^\top Q_{12} - c^\top_2, \;\; \gamma(x):=\frac{1}{2}x^\top Q_{11}x -c^\top_1 x,\;\; d(x):=d,
\]
where $A$, $B$, $b$, $Q_{12}$, $c_2$, $Q_{11}$, $c_1$, and $d$ are constant matrices/vectors, then the problem described in Assumption~\ref{ass:polyhedral} corresponds to the special case of the problem in %
Example~\ref{ex:quadratic} with $Q_{22}=0$. 
\end{remark}

\subsection{Problems with discontinuities}\label{sec:jumps}
One of the main challenges in solving a pessimistic bilevel optimization problem \eqref{eq:Pessimistic-Problem} is the fact that $\varphi_p$ may only be upper semicontinuous (cf. Fig.~\ref{fig:improved-pessimistic-approximations}). This raises the question whether also the single-level reformulations may possess some hidden, but algorithmically unfavorable properties. We prepare our answer by considering the following example.

\begin{example}\label{ex:sawtooth}
Inspired by \cite[Example 2.1]{benchouk2025scholtes}, we consider the problem \eqref{eq:Pessimistic-Problem} with $n=m=1$, 
\[
F(x,y)=x+y,\;\, f(x,y)=xy,\;\, X=[\xi,1],\;\, Y=[0,1],
\]
and a parameter $\xi<0$. One can easily check that
\begin{align*}
    S_L(x)=\begin{cases}\{1\},& x\in[\xi,0),\\ [0,1],& x=0,\\\{0\},& x\in(0,1],\end{cases}\qquad\text{and}\qquad\varphi_L(x)=\begin{cases}x,& x\in[\xi,0],\\0, & x\in(0,1]    
    \end{cases}
\end{align*}
as well as
\begin{align*}
    S_p(x)=\begin{cases}\{1\},& x\in[\xi,0],\\ \{0\},& x\in(0,1]  \end{cases}\qquad\text{and}\qquad\varphi_p(x)=\begin{cases}x+1,& x\in[\xi,0],\\ x, & x\in(0,1],\end{cases}
\end{align*}
so that \eqref{eq:Pessimistic-Problem} in particular satisfies Assumption~\ref{ass:Sp_domain} for all $\xi<0$. 
With the functional description $Y=\{y\in\R\mid 0\leq y\leq 1\}$, also Assumption~\ref{ass:polyhedral} is satisfied. Recall that \eqref{eq:ACQ} holds everywhere in $Z(x)$ since
$Z(x)=\{y\in\R\mid 0\leq y\leq 1,\ xy\leq\varphi_L(x)\}$ 
is polyhedral for all $x\in X$.

For all $\xi\le -1$, \eqref{eq:Pessimistic-Problem} possesses a global optimal solution at $\bar x=\xi$, while for all $\xi\in(-1,0)$ it is not solvable since its infimum zero is not attained. In the latter case the point $\bar x=\xi$ is, however, a local optimal solution.
The case $\xi\in (-1,0)$ shows, in particular, that Assumptions~\ref{ass:Sp_domain} and~\ref{ass:polyhedral} (let alone the more general Assumption~\ref{Assumption1}) are not sufficient for solvability of \eqref{eq:Pessimistic-Problem}.

With $\mathcal{L}_{p_D}(x, y, z, u, w)=x+(1+u_1-u_2-wx)y+u_2+wxz$, the corresponding \eqref{eq:tSLRLP}  is
\begin{align*}
   \begin{array}{rl}
   \underset{x, z, u, w}{\min} & x+u_2+wxz \\[1ex]
     \mbox{s.t.} & \xi\leq x\leq 1, \;\, 0\leq z\leq 1, \;\, u_1,u_2,w\geq 0, \;\, 1+u_1-u_2-wx=0.
\end{array}
\end{align*}
Although at first glance this problem may look well-behaved, by Corollary~\ref{cor:summary_relationships}(a), it is solvable only for $\xi\leq -1$, while it cannot possess a global optimal solution for $\xi\in(-1,0)$. In the latter case, like \eqref{eq:Pessimistic-Problem}, the corresponding problem \eqref{eq:tSLRLP} possesses the non-attained infimum zero. In fact, the feasible points $(x^k,y^k,z^k,u_1^k,u_2^k,w^k)=(1/k,0,0,0,0,k)$ yield the objective values $1/k$ for all $k\in\mathbb{N}$, and the following case distinction shows that all feasible points possess a positive objective value: In effect, for all $x\in[\xi,0]$, we have  
\begin{align*}
    x+u_2+wxz&=x+(1+u_1-wx)+wxz=x+1 + u_1 + wx(z-1)\\
    &\geq x+1\geq\xi+1>0,
\end{align*}
and $x\in(0,1]$ entails $x+u_2+wxz\geq x>0$.
\qed 
\end{example}

We point out that the Weierstrass theorem guarantees solvability of \eqref{eq:Pessimistic-Problem} if $X$ is nonempty and compact, and if $\varphi_p$ is lower semicontinuous on $X$. Given the continuity of $F$, a standard result from parametric optimization yields the lower semicontinuity of $\varphi_p$ if the set-valued mapping $S_L$ is lower semicontinuous on $X$. In Example~\ref{ex:sawtooth}, $S_L$ is not lower semicontinuous at $x=0$ and, in fact, $\varphi_p$ is not lower semicontinuous at $x=0$. For $\xi\in(-1,0)$ the problem \eqref{eq:Pessimistic-Problem} is actually not solvable.

Likewise, after the dualization of $\varphi_p$ to $\varphi_{p_D}$, the lower semicontinuity of $\varphi_p=\varphi_{p_D}$ would follow from the lower semicontinuity of $\varphi_L$ as well the as outer semicontinuity and local boundedness of the set-valued mapping $\Lambda_{p_D}$ on $X$. Sufficient conditions for the lower semicontinuity of $\varphi_L$ are the continuity of $f$ together with the outer semicontinuity and local boundedness of the set-valued mapping $Y$, which may all be considered mild assumptions. Moreover, it is not hard to see that $\Lambda_{p_D}$ is outer semicontinuous on $X$. However, since $\Lambda_{p_D}(x)$ is unbounded (even for all $x\in X$), the local boundedness assumption for $\Lambda_{p_D}$ fails. 
This shows that typical unsolvability issues in the pessimistic bilevel optimization problem \eqref{eq:Pessimistic-Problem} are inherited by the single-level problem \eqref{eq:tSLR}. They can be ruled out by additional assumptions like the lower semicontinuity of $S_L$ on $X$.

\section{Optimality conditions}\label{sec:Optimality conditions}
{We assume throughout in this section that $X\subset \mathbb{R}^n$}. In this section, we derive necessary and sufficient optimality conditions for problem \eqref{eq:tSLR} and show how to leverage on them to obtain necessary and sufficient optimality conditions for the pessimistic bilevel program \eqref{eq:Pessimistic-Problem}. Throughout this section,  we assume here that in problem \eqref{eq:tSLR}, the upper-level feasible set is described as 
\begin{equation}\label{eq:X-def}
    X:=\left\{x\in \mathbb{R}^n\left|\;\, G(x)\leq 0, \;\; H(x)=0\right.\right\}
\end{equation}
with the functions $G : \mathbb{R}^n \rightarrow \mathbb{R}^r$ and $H : \mathbb{R}^n \rightarrow \mathbb{R}^s$ being continuously differentiable. Furthermore, the functions $F$, $f$,  $g$, and $h$ are assumed to be twice continuously differentiable.  

We start in the next subsection with the construction of tractable sufficient conditions to ensure the fulfillment of the constraint qualifications \eqref{eq:MFCQ} and \eqref{eq:LICQ}  for problem \eqref{eq:tSLR}. Subsequently, in Subsection \ref{sec:Necessary and sufficient optimality conditions}, these sufficient conditions are used to  derive necessary optimality conditions for  \eqref{eq:Pessimistic-Problem}, %
before we study the standard second order sufficient condition in Subsection~\ref{sec:Second order sufficient conditions}.

\subsection{Constraint qualifications} To proceed here, note that the \textit{upper-level regularity} will be said to hold at $x$ is the MFCQ, as defined in \eqref{eq:MFCQ}, holds at this point for the constraint system defining the set  $X$. Similarly, the \textit{lower-level regularity} will be said to be satisfied at $(x, z)$ if the MFCQ is satisfied at this point for the constraint system describing the set  $Y(x)$. From now on, as necessary, we will use the notation $\zeta:=(x, y, z, u, v, w)$. 
The next result provides sufficient conditions for the fulfillment of the MFCQ for problem \eqref{eq:tSLR}.

\begin{theorem}\label{lem:MFCQ-tSLR} The MFCQ is satisfied at a feasible point $\zeta:=(x, y, z, u,  v, w)$ of  \eqref{eq:tSLR} if the upper-level (resp. lower-level) regularity holds at $x$ (resp. $(x, z)$) and the matrix $\nabla^2_{yy}\mathcal{L}_{p_D}(\zeta)$ is full rank. 
\end{theorem}
\begin{proof}
We start with the notation 
\begin{equation}\label{eq:alpha_beta_etc}
    \begin{array}{c}
\alpha:=\left[\begin{array}{l}
\alpha_G\\
\alpha_g\\
\alpha_u\\
\alpha_w
\end{array}
\right], \;
\beta:=\left[\begin{array}{l}
\beta_H\\
\beta_h\\
\beta_{\mathcal{L}_{p_D}}
\end{array}
\right],\;\;
\tilde{G}(\zeta):=\left[\begin{array}{l}
G(x)\\
g(x, z)\\
-u\\
-w
\end{array}
\right], \;\;
\tilde{H}(\zeta):=\left[\begin{array}{l}
H(x)\\
h(x, z)\\
\nabla_y\mathcal{L}_{p_D}(\zeta)
\end{array}
\right].
\end{array}
\end{equation}
Then, we can easily check that the condition $\nabla_\zeta \tilde{G}(\zeta)^\top \alpha + \nabla_\zeta \tilde{H}(\zeta)^\top \beta =0$ is equivalent to the following system of equations:
\begin{eqnarray}
  \nabla G(x)^\top \alpha_G + \nabla H(x)^\top \beta_H + \nabla_x g(x, z)^\top \alpha_g + \nabla_x h(x, z)^\top \beta_h  \qquad \quad \nonumber\\[2ex]
  +\;\, \nabla^2_{xy}\mathcal{L}_{p_D}(\zeta)^\top \beta_{\mathcal{L}_{p_D}}=0,\label{MFCQ-1}\\[2ex]
  \nabla_z g(x, z)^\top \alpha_g + \nabla_z h(x, z)^\top \beta_h =0,\label{MFCQ-2}\\[2ex]
  \nabla^2_{yy}\mathcal{L}_{p_D}(\zeta)^\top \beta_{\mathcal{L}_{p_D}}=0, \qquad \nabla^2_{vy}\mathcal{L}_{p_D}(\zeta)^\top \beta_{\mathcal{L}_{p_D}}=0,\label{MFCQ-3}\\[2ex]
  \nabla^2_{uy}\mathcal{L}_{p_D}(\zeta)^\top \beta_{\mathcal{L}_{p_D}}=\alpha_u, \qquad \nabla^2_{wy}\mathcal{L}_{p_D}(\zeta)^\top \beta_{\mathcal{L}_{p_D}}=\alpha_w.\label{MFCQ-4}
\end{eqnarray}
With this system, if we formally write the dual form of the MFCQ from Section~\ref{sec:CQs} for  \eqref{eq:tSLR} at  $\zeta:=(x, y, z,  u,  v,  w)$, we can easily check that from the lower-level regularity at $(x, z)$, condition \eqref{MFCQ-2} will imply that $\alpha_g =0$ and $\beta_h =0$, while the first equation in \eqref{MFCQ-3} will lead to $\beta_{\mathcal{L}_{p_D}}=0$ under the full rank condition imposed in the statement of the lemma. The latter will also imply, considering \eqref{MFCQ-4}, that we have $\alpha_u=0$ and $\alpha_w=0$. Finally, with $\alpha_g =0$, $\beta_h =0$, and $\beta_{\mathcal{L}_{p_D}}=0$, it will follow from \eqref{MFCQ-1} that $\alpha_G=0$ and $\alpha_H=0$ under the fulfillment of the upper-level regularity at $x$.    \qed
\end{proof}

Similarly, we next provide sufficient conditions for the fulfillment of the LICQ for \eqref{eq:tSLR}. The\textit{ upper-level (resp. lower-level) LICQ} will be said to hold at $x$ (resp. $(x, z)$) if LICQ holds at this point for the constraint system defining the upper-level (resp. lower-level) feasible  $X$ (resp.  $Y(x)$).
\begin{theorem}\label{lem:LICQ-tSLR} 
The LICQ is satisfied at a feasible point $\zeta:=(x, y,  z,  u,  v,  w)$ of   \eqref{eq:tSLR} if the upper-level (resp. lower-level) LICQ holds at $x$ (resp. $(x, z)$) and $\nabla^2_{yy}\mathcal{L}_{p_D}(\zeta)$ is full rank. 
\end{theorem}
\begin{proof}
Follows along the same line as in the proof of Theorem \ref{lem:MFCQ-tSLR}. \qed
\end{proof}

Next, we provide a framework for the fulfillment of the full rank condition in the result above. 
\begin{proposition}\label{Prop:sufficient_condition_full_rank} %
If $\zeta:=(x, y,   u,  v,  w)$ is feasible for \eqref{eq:tSLR}, then $\nabla^2_{yy}\mathcal{L}_{p_D}(\zeta)$ is full rank, provided that one of the following assumptions is satisfied:
\begin{itemize}
    \item[(a)] $\nabla^2_{yy}{F}(x, y) \prec 0$;
    \item[(b)] $\nabla^2_{yy}{f}(x, y) \succ 0$ and $w>0$;
    \item[(c)] $\nabla^2_{yy}{g}_i(x, y) \succ 0$ and $u_i>0$ for some $i=1, \ldots, p$.
\end{itemize}
\end{proposition}
\begin{proof}
    Start by observing that based on Assumption \ref{Assumption1}(1)--(4), it holds that 
    \begin{equation}\label{eq:full_rank_fulfillment}
       \nabla^2_{yy}\mathcal{L}_{p_D}(\zeta) = -\left(-\nabla^2_{yy}{F}(x, y) + w \nabla^2_{yy}{f}(x, y) +\sum^{p}_{i=1} u_i \nabla^2_{yy}{g}_i(x, y)  \right)\preceq 0
    \end{equation}
given that $w\geq 0$ and $u_i\geq 0$ for $i=1, \ldots, p$ (thanks to the feasibility of $\zeta:=(x, y,   u,  v,  w)$) and the fact that the functions $F$, $f$, and $g$ are twice continuously differentiable w.r.t. $y$. Therefore, if assumption (a), (b), or (c) of the statement holds, then we have from \eqref{eq:full_rank_fulfillment} that 
$
\nabla^2_{yy}\mathcal{L}_{p_D}(\zeta) \prec 0,
$
ensuring that the latter matrix is full rank. \qed
\end{proof}
As for the upper- and lower-level regularity and LICQ, we can easily construct examples of functions describing the set $X$ \eqref{eq:X-def} and the set-valued mapping $Y$ \eqref{eq:Y(x)}, under the framework of Proposition \ref{Prop:sufficient_condition_full_rank}, such that they are satisfied.
 This is one of the strengths of reformulation \eqref{eq:tSLR}, as none of the SLRs of \eqref{eq:Pessimistic-Problem} introduced in Section \ref{sec:Existing reformulations and numerical algorithms} can permit the fulfillment of the MFCQ or LICQ. This is neither possible for the optimistic bilevel program, as  widely documented in the literature \cite{dempe2002foundations,dempe2020bilevel}. %

\subsection{Necessary  optimality conditions}\label{sec:Necessary and sufficient optimality conditions}
We start here by establishing the first order necessary optimality conditions of problem \eqref{eq:tSLR}. 
\begin{theorem}\label{the:tSLR_KKT}
    Let $(x, y, z,  u, v,  w)$ be a local optimal solution of problem \eqref{eq:tSLR}, where all the assumptions of Theorem \ref{lem:MFCQ-tSLR} are satisfied. Then, there exist Lagrange multipliers $\alpha_G$, $\alpha_g$, 
    $\beta_H$, $\beta_h$, and $\beta_{\mathcal{L}_{p_D}}$ such that the following conditions are satisfied: 
    \begin{eqnarray}
  \nabla G(x)^\top \alpha_G + \nabla H(x)^\top \beta_H + \nabla_x g(x, z)^\top \alpha_g + \nabla_x h(x, z)^\top \beta_h\; \qquad \quad \nonumber\\[2ex]
  +\;\,\nabla_{x}\mathcal{L}_{p_D}(\zeta) \;+\; \nabla^2_{xy}\mathcal{L}_{p_D}(\zeta)^\top \beta_{\mathcal{L}_{p_D}}=0,\label{KKT-U-1}\\[2ex]
   \alpha_G \geq 0, \quad G(x)\leq 0, \quad \alpha^\top_G G(x)=0,\label{KKT-U-2}\\[2ex]
   H(x)=0,\label{KKT-U-3}\\[2ex]
  {\nabla^2_{yy}\mathcal{L}_{p_D}(\zeta)^\top \beta_{\mathcal{L}_{p_D}}=0}, \;\;
  {\nabla_{y}\mathcal{L}_{p_D}(\zeta)=0},\;\;
   {h(x, y) + \nabla_y h(x,y)\beta_{\mathcal{L}_{p_D}}=0},\label{KKT-I-1}\\[2ex]
   {u \geq 0, \;\; g(x, y) + \nabla_y g(x,y)\beta_{\mathcal{L}_{p_D}}\leq 0, \;\; u^\top \left(g(x, y) + \nabla_y g(x,y)\beta_{\mathcal{L}_{p_D}}\right)=0},\label{KKT-I-2}\\[2ex]
    {w \geq 0, \;\; f(x, y) - f(x,z) + \nabla_y f(x, y)^\top\beta_{\mathcal{L}_{p_D}}\leq 0},\; \qquad\qquad\qquad\qquad\quad\quad \;\;\,\nonumber\\[2ex]
     {w \left(f(x, y) - f(x,z) + \nabla_y f(x, y)^\top\beta_{\mathcal{L}_{p_D}}\right)=0},\label{KKT-I-3}\\[2ex]
 {w\nabla_z f(x, z) + \nabla_z g(x, z)^\top \alpha_g + \nabla_z h(x, z)^\top \beta_h =0},\label{KKT-L-1}\\[2ex]
    {\alpha_g \geq 0, \quad g(x, z)\leq 0, \quad \alpha^\top_g g(x, z)=0},\label{KKT-L-2}\\[2ex]
    {h(x, z)=0}. \label{KKT-L-3}%
\end{eqnarray}
\end{theorem}
\begin{proof}
It follows straightforwardly from the application of the classical Lagrange multiplier rule to problem  \eqref{eq:tSLR}, while taking into account the fact that under the assumptions of Theorem \ref{lem:MFCQ-tSLR}, the MFCQ holds at  $(\bar x, \bar y, \bar z, \bar u, \bar v, \bar w)$, as a feasible point of  \eqref{eq:tSLR}, and also observing that 
\begin{equation}\label{eq:alpha_u_alpha_w}
\alpha_u:=-\left(g(x, y) + \nabla_y g(x,y)\beta_{\mathcal{L}_{p_D}}\right) \mbox{ and } \alpha_w:= -\left(f(x, y) - f(x,z) + \nabla_y f(x, y)^\top\beta_{\mathcal{L}_{p_D}}\right),
\end{equation}
respectively, based on the corresponding definitions in \eqref{eq:alpha_beta_etc}. \qed
\end{proof}

Note that the system \eqref{KKT-U-1}--\eqref{KKT-L-3} corresponds to the KKT conditions of problem \eqref{eq:tSLR}. 
\begin{remark}\label{rem:Explanability_KKT_conditions}
   Based on Assumption \ref{Assumption1}(2)-(4), with $w\geq 0$, the existence of the Lagrange multipliers $\alpha_g$ and $\beta_h$ such that the block \eqref{KKT-L-1}--\eqref{KKT-L-3} of this system holds is equivalent to the inclusion $z\in S^*_L(x,  w)$. Therefore, in some sense, this represents the lower-level problem in the KKT conditions \eqref{KKT-U-1}--\eqref{KKT-L-3}. As for the block \eqref{KKT-I-1}--\eqref{KKT-I-3}, it corresponds to necessary conditions for $(y, u, v, w)\in S^L_{p_D}(x)$; hence, meaning that this part of the KKT conditions of problem \eqref{eq:tSLR} represents the Wolfe dual of the intermediate problem \eqref{eq:IL}. In the same vein, the block \eqref{KKT-U-1}--\eqref{KKT-U-3} can be viewed as the part of these optimality conditions representing the upper-level problem described in  \eqref{eq:Pessimistic-Problem} or \eqref{eq:P_pD}. 
\end{remark}

\begin{corollary}\label{cor:KKT_conditions_Pp}
Let $x$ be a local optimal solution of problem \eqref{eq:Pessimistic-Problem}, and assume that there exist points $(y, u, v, w)\in S^L_{p_D}(x)$ and $z\in S^*_L(x,  w)$ such that the upper-level regularity (resp. lower-level regularity) holds at $x$ (resp. $(x, z)$) and $\nabla^2_{yy}\mathcal{L}_{p_D}(x, y, z, u, v,  w)$ is full rank. Then, there exist Lagrange multipliers $\alpha_G$, $\alpha_g$, %
$\beta_H$, $\beta_h$, and $\beta_{\mathcal{L}_{p_D}}$ such that the KKT conditions \eqref{KKT-U-1}--\eqref{KKT-L-3} are satisfied. 
\end{corollary}
\begin{proof}
First note that considering the fulfillment of Assumption \ref{ass:Sp_domain}, it holds that $S^L_{p_D}(x) \neq \emptyset$, while accounting for the fulfillment of the Wolfe duality result, which in turn holds thanks to Assumption  \ref{Assumption1} (for reference, see Lemma \ref{eq:Wolfe_Dual_P} and Lemma \ref{lem:phi_p=phi_pD}). Then observe that also due to Assumption \ref{ass:Sp_domain}, it holds that $S^*_L(x, w)\neq \emptyset$. Subsequently, the overall conclusion of the result follows from a combination of the first part of Corollary \ref {cor:summary_relationships}(b) and Theorem \ref{the:tSLR_KKT}. \qed
\end{proof}
It follows from Remark \ref{rem:Explanability_KKT_conditions} that the inclusions $(y, u, v, w)\in S^L_{p_D}(x)$ and $z\in S^*_L(x,  w)$ are already represented in the KKT conditions \eqref{KKT-U-1}--\eqref{KKT-L-3}, with $z\in S^*_L(x,  w)$ equivalently via \eqref{KKT-L-1}--\eqref{KKT-L-3} and $(y, u, v, w)\in S^L_{p_D}(x)$ necessarily with the presence of \eqref{KKT-I-1}--\eqref{KKT-I-3}. Therefore, they are somewhat redundant, and do not necessarily need to be accounted for while referring to the necessary optimality conditions of problem \eqref{eq:Pessimistic-Problem} obtained here via \eqref{eq:tSLR}. 

The result in Corollary \ref{cor:KKT_conditions_Pp} represents a fundamental paradigm shift in terms of the construction of necessary optimality conditions for the pessimistic bilevel optimization problem \eqref{eq:Pessimistic-Problem}, and two main observations could made to compare it with existing ones from the literature:\\[1ex]
\textbf{(i)} The existing approaches to derive necessary optimality conditions for problem \eqref{eq:Pessimistic-Problem} are based on calculations of upper estimates for the subdifferential of $\varphi_p$ \cite{dempe2014necessary,dempe2012sensitivity,dempe2019two}. Hence, the required qualification conditions involve assumptions to ensure that this function is Lipschitz continuous near the point of interest. In particular, it is usually required that the set-valued mapping $S_p$ \eqref{eq:Two-level-Function-S_p} or its suitable transformation, depending on the context, satisfies some continuity properties such as the inner semicontinuity, which is not only a strong requirement, but also a very difficult condition to verify in practice. On the other hand, Proposition \ref{Prop:sufficient_condition_full_rank} provides a base for a large class of problems for which all the requirements of Corollary \ref{cor:KKT_conditions_Pp} are automatically satisfied. \\[1ex]
\textbf{(ii)} The existing optimality conditions for  \eqref{eq:Pessimistic-Problem} usually involve combinatorial structures such as the S-, M-, and C-type necessary optimality conditions, which are difficult to check or compute in practice. Additionally, most of the existing optimality conditions can give rise to quite large systems, given that they involve convex combinations due to the convex hull structure that intervenes by virtue of the process to compute elements from the Clarke subdifferential of $\varphi_p$. On the contrary, our necessary optimality conditions \eqref{KKT-U-1}--\eqref{KKT-L-3} are of the usual KKT-type, as they involve only complementarity conditions. Furthermore, we can easily check that \eqref{KKT-U-1}--\eqref{KKT-L-3} can be written as a 
$(n+3m+2p+2q+r+s+1)\times (n+3m+2p+2q+r+s+1)$
square system of equations. %

In preparation of the following subsection, we complement the above first order necessary optimality conditions by second order conditions. Considering the nature of the feasible set of problem \eqref{eq:tSLR}, we need the following assumption.

\begin{assumption}\label{ass:C3}
The functions $G : \mathbb{R}^n \rightarrow \mathbb{R}^r$ and $H : \mathbb{R}^n \rightarrow \mathbb{R}^s$ are twice continuously differentiable, while $F$, $f$,  $g$, and $h$ are thrice continuously differentiable.
\end{assumption}
In the subsequent results, we shall use the combined multiplier vectors $\alpha$, $\beta$ and the combined constraint functions $\tilde{G}$, $\tilde{H}$ from \eqref{eq:alpha_beta_etc}, where $\alpha_u$ and $\alpha_w$ are defined as in \eqref{eq:alpha_u_alpha_w}. The Lagrangian function of problem \eqref{eq:tSLR} thus is
\begin{equation}\label{eq:Lagrangian_tSLR}
\begin{array}{rll}
 \mathcal{L}(\zeta, \alpha, \beta) & := & \mathcal{L}_{p_D}(\zeta) + \alpha_G^\top G(x) + \alpha_g^\top g(x, z) -\alpha_u^\top u -\alpha_w w\\[2ex]
                         &    & \qquad \quad\; + \,\beta_H^\top H(x) + \beta_h^\top h(x, z) + \sum^{m}_{j=1}\beta_{\mathcal{L}_{p_D},j} \nabla_{y_j} \mathcal{L}_{p_D}(\zeta),
\end{array}
\end{equation}
and the critical cone from \eqref{eq:Critial_Cone} has the form
\begin{equation}\label{eq:Critial_Cone_tSLR}
   \mathfrak{C}(\zeta, \alpha) :=\left\{d\in \mathbb{R}^{n+2m+p+q+1}\left|~\begin{array}{ll}
    \nabla\tilde G_i(\zeta)^\top d = 0\;\; \forall i:\;\, \tilde{G}_i(\zeta)=0, \;\; \alpha_i >0\\[2ex]
    \nabla\tilde{G}_i(\zeta)^\top d\leq 0\;\; \forall i:\;\, \tilde{G}_i(\zeta)=0, \;\; \alpha_i =0\\[2ex]
     \nabla\tilde H_j(\zeta)^\top d = 0\;\; \forall j=1, \ldots, q
\end{array} \right.\right\}.  
\end{equation}

The following theorem results from an application of the standard second order necessary optimality condition under LICQ to the problem \eqref{eq:tSLR}.
\begin{theorem}\label{the:tSLR_SONC}
    Let Assumption~\ref{ass:C3} hold and let $\zeta=(x, y, z,  u, v,  w)$ be a local optimal solution of problem \eqref{eq:tSLR}, where all the assumptions of Theorem \ref{lem:LICQ-tSLR} are satisfied. Then, there exist unique multipliers $\alpha$ and $\beta$ such that the KKT conditions \eqref{KKT-U-1}--\eqref{KKT-L-3} are satisfied and such that we have 
    \begin{equation}\label{eq:SONC}
d^\top \nabla^2_{\zeta\zeta} \mathcal{L}\left(\zeta, \alpha, \beta\right)d \geq 0 \;\; \mbox{ for all }\;\; d\in \mathfrak{C}(\zeta, \alpha).
\end{equation}
\end{theorem}
Linking a local optimization solution of problem  \eqref{eq:Pessimistic-Problem} to one of \eqref{eq:tSLR} like in the proof of Corollary~\ref{cor:KKT_conditions_Pp}, we obtain the following result.
\begin{corollary}\label{cor:SONC_Pp}
Let $x$ be a local optimal solution of problem \eqref{eq:Pessimistic-Problem}, and assume that there exist points $(y, u, v, w)\in S^L_{p_D}(x)$ and $z\in S^*_L(x,  w)$ such that the upper-level LICQ (resp. lower-level LICQ) holds at $x$ (resp. $(x, z)$) and $\nabla^2_{yy}\mathcal{L}_{p_D}(x, y, z, u, v,  w)$ is full rank. Then, there exist unique Lagrange multipliers $\alpha$ and $\beta$ such that the KKT conditions \eqref{KKT-U-1}--\eqref{KKT-L-3} and the second order condition \eqref{eq:SONC} are satisfied. 
\end{corollary}

\subsection{Sufficient optimality conditions}\label{sec:Second order sufficient conditions}
In continuous optimization, it is often important to characterize strict local optimal solutions, as done with the second order sufficient condition \eqref{eq:SOSC} for  problem \eqref{eq:min_standard}. A key question in this subsection is to know whether such a framework is applicable to our true single-level reformulation model  \eqref{eq:tSLR}, and how this could potentially help to build sufficient conditions for strict locality for problem \eqref{eq:Pessimistic-Problem}. We require Assumption~\ref{ass:C3} to hold throughout this subsection.

We start by establishing a relationship between the strict local optimality of \eqref{eq:tSLR} and \eqref{eq:Pessimistic-Problem}. 
\begin{theorem}\label{the:SOSC}%
 If $\bar\zeta:=(\bar x, \bar y, \bar z, \bar u, \bar v, \bar w)$ is a strict local optimal solution of \eqref{eq:tSLR}, then $\bar x$ is a strict local optimal solution of problem \eqref{eq:Pessimistic-Problem}, provided that the following assumptions are satisfied:
\begin{itemize}
    \item[(a)] $S^L_{p_D}$  is inner semicontinuous at $(\bar x, \bar y, \bar u, \bar v, \bar w)$;
    \item[(b)] $S^*_L$ is inner semicontinuous at  $(\bar x, \bar w, \bar z)$.
\end{itemize}
\end{theorem}
\begin{proof}
Note that based on the hypothesis, the local optimality of $\bar x$ for problem \eqref{eq:Pessimistic-Problem} follows directly by virtue of the of the second part of Corollary \ref{cor:summary_relationships}(b), thanks to assumptions (a) and (b) of this theorem. %
On the strictness, suppose, by contradiction, that $\bar x$ is not a strict local optimal solution of problem \eqref{eq:Pessimistic-Problem}. %
Then, there exists a sequence $\{x^k\}\subset X\setminus\{\bar x\}$ satisfying
$
x^k\rightarrow\bar x$ such that 
$\varphi_p(x^k)\leq\varphi_p(\bar x)$ 
for all $k$.
Subsequently, by the inner semicontinuity of $S^L_{p_D}$ at
$(\bar x,\bar y,\bar u,\bar v,\bar w)$, there exists a sequence
$
(y^k,u^k,v^k,w^k)\in S^L_{p_D}(x^k)
$
such that
$
(y^k,u^k,v^k,w^k)\rightarrow
(\bar y,\bar u,\bar v,\bar w).
$
Similarly, by the inner semicontinuity of $S^*_L$ at
$(\bar x,\bar w,\bar z)$, there exists a sequence
$
z^k\in S^*_L(x^k,w^k)
$
such that we have 
$
z^k\rightarrow\bar z. 
$
It follows from equations \eqref{eq:phi_pD} and \eqref{eq:SL=SLpD}, together with \eqref{eq:LminL}, that 
\[
\mathcal{L}^L_{p_D}(x^k,y^k,u^k,v^k,w^k)
=
\varphi_p(x^k) \;\, \mbox{ and }\;\, \mathcal{L}_{p_D}(x^k,y^k,z^k,u^k,v^k,w^k)
=
\mathcal{L}^L_{p_D}(x^k,y^k,u^k,v^k,w^k),
\]
respectively. Therefore,
$
\mathcal{L}_{p_D}(x^k,y^k,z^k,u^k,v^k,w^k)
=
\varphi_p(x^k)
\leq
\varphi_p(\bar x).
$

On the other hand, we also implicitly have from assumptions (a) and (b) of the theorem that 
$
(\bar y,\bar u,\bar v,\bar w)\in S^L_{p_D}(\bar x)
$
and
$
\bar z\in S^*_L(\bar x,\bar w), 
$ respectively, are satisfied. Hence, it follows once more from equations \eqref{eq:phi_pD} and \eqref{eq:SL=SLpD}, together with \eqref{eq:LminL}, that 
\[
\mathcal{L}_{p_D}(\bar x,\bar y,\bar z,\bar u,\bar v,\bar w)
=
\mathcal{L}^L_{p_D}(\bar x,\bar y,\bar u,\bar v,\bar w)
=
\varphi_p(\bar x).
\]
Subsequently, we get that for every $k$, 
\[
\mathcal{L}_{p_D}(x^k,y^k,z^k,u^k,v^k,w^k)
\leq
\mathcal{L}_{p_D}(\bar x,\bar y,\bar z,\bar u,\bar v,\bar w).
\]
Since
$
(x^k,y^k,z^k,u^k,v^k,w^k)
\rightarrow
(\bar x,\bar y,\bar z,\bar u,\bar v,\bar w),
$
this contradicts the strict local optimality of the point 
$(\bar x,\bar y,\bar z,\bar u,\bar v,\bar w)$ for  problem \eqref{eq:tSLR}. 
Therefore, $\bar x$ is a strict local optimal solution of  \eqref{eq:Pessimistic-Problem}. 
\qed
\end{proof}

It is important to note that assumptions (a) and (b) are not there just by default to ensure the local optimality the $x$-component of a local optimal solution of problem \eqref{eq:tSLR} for problem \eqref{eq:Pessimistic-Problem}; these assumptions are also crucial to establish the strictness of such a local optimal solution. 

We emphasize that the assumptions of Theorem~\ref{the:SOSC} cover the ones which guarantee local optimality of $\bar x$ for \eqref{eq:Pessimistic-Problem} in the second part of Corollary~\ref{cor:summary_relationships}(b). The only additional assumption is the strictness of local optimality in the corresponding problem \eqref{eq:tSLR}. In the following example, all the assumptions of Theorem \ref{the:SOSC} hold, except the strict local optimality of problem \eqref{eq:tSLR}. Nevertheless, the local optimality of problem \eqref{eq:Pessimistic-Problem} is still strict.

\begin{example}\label{exam:Strict_Fails_V2}
For the problem \eqref{eq:Pessimistic-Problem} with  %
\begin{align*}%
    F(x,y):=y, \;\; f(x,y):=x, \;\; X:=\mathbb{R}, \;\;\mbox{ and } \;\; Y(x):=[-x^2,x^2],
\end{align*}
one obtains $\varphi_p(x)=x^2$, so that $\bar x=0$ is a strict local optimal solution of problem \eqref{eq:Pessimistic-Problem}. %

Assumption \ref{ass:Sp_domain} holds as $S_p(x)=\{x^2\}$ for all $x\in X$. Additionally, Assumption~\ref{Assumption1} also holds, as we can put $g_1(x,y)=-x^2-y$, $g_2(x,y)=-x^2+y$, while accounting for the fact that  $\varphi_L(x)=x$. %
Hence, the corresponding problem \eqref{eq:tSLR} here is
\begin{align}
    \min_{x,y,z,u_1,u_2,w}\,& \mathcal{L}_{p_D}(x, y, z, u, w):= y+u_1(x^2+y)+u_2(x^2-y)\nonumber\\
    \text{s.t. }\ & -x^2-z\leq0,\ -x^2+z\leq0,\  u,w\geq0,\label{eq:example58}\\
    & 1+u_1-u_2=0.\nonumber 
\end{align}
The point $\bar\zeta=(\bar x,\bar y,\bar z,\bar u,\bar w)$, with $\bar x=\bar y=\bar z=0$, $\bar u=(0,1)^\top$, and $\bar w=1$, is a global, and therefore a local, optimal solution of this problem. In fact, for any feasible point $\zeta=(x, y, z, u, w)$ of \eqref{eq:example58}, 
\[
\mathcal{L}_{p_D}(\zeta)= (1+u_1 -u_2)y + (u_1 + u_2)x^2 = (u_1 + u_2)x^2 \geq 0= \mathcal{L}_{p_D}(\bar\zeta).
\]
However, the family of points 
$\zeta(t):=(0,t,0,0,1,1)$ for $t\in \mathbb{R}$, which is feasible for problem \eqref{eq:example58} and satisfies $\zeta(t) \rightarrow \bar\zeta$ as $t \rightarrow 0$, is such that for any $t\neq 0$, $\zeta(t) \neq \bar\zeta$ and $\mathcal{L}_{p_D}(\zeta(t)) = \mathcal{L}_{p_D}(\bar\zeta)$. Thus, $\bar\zeta$ is not a strict local optimal solution of  problem \eqref{eq:example58}.

With respect to the inner semicontinuity of $S^L_{p_D}$ at $(\bar x,\bar y,\bar u,\bar w)$, observe that 
\begin{align*}
   \Lambda_{p_D}(x)=\{(y,u,w)\mid u\geq0,\ w\geq0,\ 1+u_1-u_2=0\}
\end{align*}
and $\varphi_{p_D}(x)=\varphi_p(x)=x^2$ yield
\[
S^L_{p_D}(x)
=
\begin{cases}
\left\{
(y,0,1,w)
\;\middle|\;
y\in\mathbb R,\ w\geq0
\right\},
& \text{if }x\neq0,\\[2ex]
\left\{
(y,u_1,1+u_1,w)
\;\middle|\;
y\in\mathbb R,\ u_1\geq0,\ w\geq0
\right\},
& \text{if }x=0.
\end{cases}
\]
From $(\bar y,\bar u,\bar w)\in S^L_{p_D}(x)$ for all $x\in\R$ we obtain the required inner semicontinuity assumption.

Finally, $S^*_L$ is inner semicontinuous at $(\bar x,\bar w,\bar z)=(0,1,0)$ in view of $S^*_L(x)=Y(x)$ for all $x\in\R$ and the inner semicontinuity of $Y$ at $(0,0)$.
\qed 
\end{example}

An interesting feature of this example is the fact that $S_L(\bar x)$ is a singleton, while $S_L(x)$ is a proper interval for any $x\in X\setminus \{\bar x\}$. Moreover, as it can be seen in the graphs of $\varphi_o$ and $\varphi_p$ in Fig. \ref{fig:data_exp}, it seems quite interesting that $\bar x$ is the global optimal solution for the corresponding problem \eqref{eq:Pessimistic-Problem}, while being the worst point (from the minimization perspective) for the optimistic problem \eqref{eq:Optimistic-Problem}. %
\begin{figure}[htp]
\centering
\includegraphics[width=0.65\linewidth]{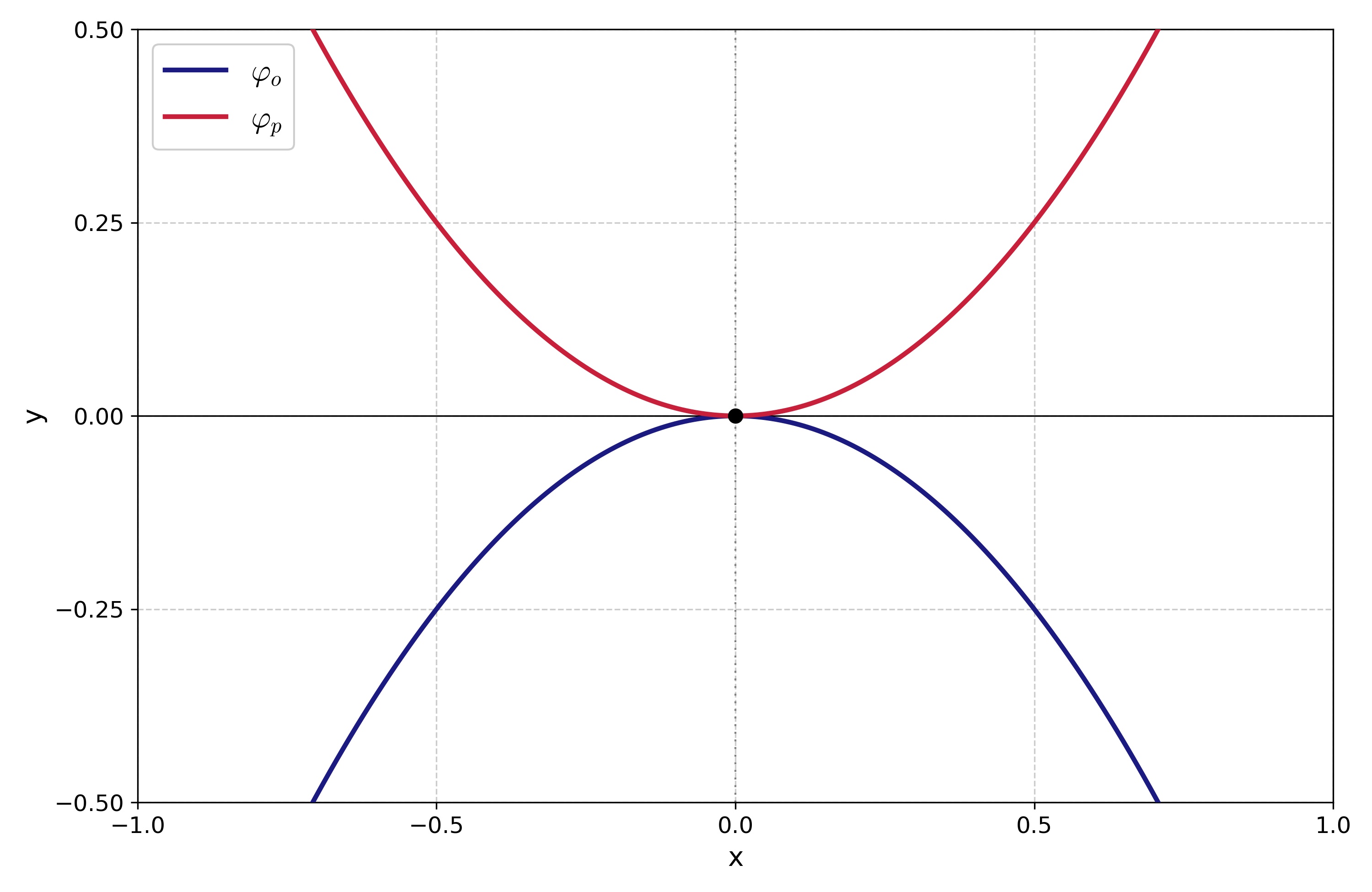}
\caption{Graphs of $\varphi_o$ from \eqref{eq:def_phi_o} and $\varphi_p$ from \eqref{eq:def_phi_p} for the problem in  Example~\ref{exam:Strict_Fails_V2}.}\label{fig:data_exp}
\end{figure}

Example \ref{exam:Strict_Fails_V2} also illustrates that the converse of the implication in Theorem \ref{the:SOSC} is not necessarily true. 
In practice, one would expect to use the SOSC of problem \eqref{eq:tSLR} as key tool to guarantee that a point is a  strict local optimal solution, in order to enable one to leverage on Theorem \ref{the:SOSC} to derive a sufficient condition for strict local optimality for problem \eqref{eq:Pessimistic-Problem}. Unfortunately, based on the discussion in Subsection \ref{sec:CQs}, the point $\bar\zeta=(\bar x,\bar y,\bar z,\bar u,\bar w)$ with $\bar x=\bar y=\bar z=0$, $\bar u=(0,1)^\top$, $\bar w=1$, which can easily be shown to be a KKT point of problem \eqref{eq:example58} with multipliers $\bar\alpha_g=(1/4, 1/4)^\top$, $\bar\alpha_u=(0,0)^\top$, $\bar\alpha_w=0$, and $\bar\beta:=\bar\beta_{\mathcal{L}_{p_D}}=0$ does not satisfy the SOSC (as $\bar\zeta$ is not a strict local optimal solution of problem \eqref{eq:example58}, according to Example \ref{exam:Strict_Fails_V2}). 
In fact, the critical cone can be obtained as 
\begin{align*}
    \mathfrak{C}(\bar\zeta,\bar\alpha)=\{(d_x,d_y,0,d_{u_1},d_{u_1}, d_w)^\top\mid d_x,d_y, d_w\in\R,\,d_{u_1}\geq0\},
\end{align*}
and the Hessian of the Lagrangian function $\mathcal{L}$ of problem \eqref{eq:example58} is 
\begin{equation*}
\nabla_{\zeta\zeta}^{2}
\mathcal{L}(\bar\zeta,\bar\alpha,\bar\beta)
=
\begin{pmatrix}
1&0&0&0&0&0\\
0&0&0&1&-1&0\\
0&0&0&0&0&0\\
0&1&0&0&0&0\\
0&-1&0&0&0&0\\
0&0&0&0&0&0
\end{pmatrix}.
\end{equation*}
This yields $d^\top \nabla^2_\zeta \mathcal{L}(\bar\zeta,\bar\alpha, \bar\beta)d=d_x^2$ for all $d\in\mathfrak{C}(\bar\zeta,\bar\alpha)\setminus\{0\}$. Now take, for example,  $d^\top_*=(0,1,0,0,0,0)$. We have $d_* \in \mathfrak{C}(\bar\zeta,\bar\alpha)\setminus\{0\}$, and obviously, $d^\top_* \nabla^2_\zeta \mathcal{L}(\bar\zeta,\bar\alpha, \bar\beta)d_*=0$. Thus, the SOSC fails at $(\bar\zeta,\bar\alpha, \bar\beta)$.

A key observation that could be made for this example is that $\bar y, \bar z\in S_L(\bar x)$ and $\Lambda_L(\bar x, \bar y) \neq \emptyset$, where $\Lambda_L(\bar x, \bar y)$ denotes the set of lower-level Lagrange multipliers, which is defined as %
\[
\Lambda_L(x, y):=\left\{(u, v)\in \mathbb{R}^p\times \mathbb{R}^q\left|
\begin{array}{l}
   \nabla_y f(x,y) + \nabla_y g(x,y)^\top u + \nabla_y h(x,y)^\top v =0\\[1ex]
     u\geq 0, \;\; g(x,y)\leq 0, \;\; u^\top g(x,y)=0\\[1ex]
     h(x,y)=0
\end{array}  \right.\right\}
\]
for any pair $(x, y)$ such that $y$ is feasible for \eqref{eq:LL}. The next result will enable us to show in Remark \ref{rem:Failure_Strict_Optimality} that \eqref{eq:SOSC} fails for the problem from Example \ref{exam:Strict_Fails_V2} because  $\bar y, \bar z\in S_L(\bar x)$ and $\Lambda_L(\bar x, \bar y) \neq \emptyset$. %
\begin{proposition}\label{prop:flat_ray}
 Let $\bar\zeta:=(\bar x, \bar y, \bar z, \bar u, \bar v, \bar w)$ be a feasible point for problem \eqref{eq:tSLR} with $\bar y, \bar z \in S_L(\bar x)$ and $\Lambda_L(\bar x, \bar y)\neq \emptyset$. Then, for every $(u, v)\in \Lambda_L(\bar x, \bar y)$, the path
 \[
 \zeta(t):=(\bar x, \bar y, \bar z, \bar u +tu, \bar v +tv, \bar w +t) \;\; \mbox{ for every }\; t\geq 0
 \]
 is feasible for problem \eqref{eq:tSLR} and satisfies $\mathcal{L}_{p_D}(\zeta(t))=\mathcal{L}_{p_D}(\bar\zeta)$ for every $t\geq 0$.
\end{proposition}
\begin{proof} For any $(u, v)\in \Lambda_L(\bar x, \bar y)$, considering the fact that $\bar\zeta:=(\bar x, \bar y, \bar z, \bar u, \bar v, \bar w)$ is feasible for \eqref{eq:tSLR}, 
\begin{align*}
&\nabla_yF(\bar x,\bar y)
-\nabla_yg(\bar x,\bar y)^\top u(t)
-\nabla_yh(\bar x,\bar y)^\top v(t)
-w(t)\nabla_yf(\bar x,\bar y)\\
& \qquad \qquad \qquad \qquad =
\nabla_yF(\bar x,\bar y)
-\nabla_yg(\bar x,\bar y)^\top\bar u
-\nabla_yh(\bar x,\bar y)^\top\bar v
-\bar w\nabla_yf(\bar x,\bar y)\\
&\qquad \qquad \qquad \qquad \qquad \qquad \qquad \qquad \qquad
-t\left[
\nabla_yf(\bar x,\bar y)
+\nabla_yg(\bar x,\bar y)^\top u
+\nabla_yh(\bar x,\bar y)^\top v
\right] =0
\end{align*}
for any $t\geq 0$. Hence, $\zeta(t)$ is a feasible point for problem \eqref{eq:tSLR} for any $t\geq 0$. Moreover, 
\begin{align*}
\mathcal{L}_{p_D}(\zeta(t))-\mathcal{L}_{p_D}(\bar\zeta)
= -tu^\top g(\bar x,\bar y) - tv^\top h(\bar x,\bar y) -t\left(f(\bar x,\bar y)-f(\bar x,\bar z)\right).
\end{align*}
Since $(u, v)\in \Lambda_L(\bar x, \bar y)$, it holds that $u^\top g(\bar x,\bar y)=0$ and $h(\bar x,\bar y)=0$. Additionally, given that we have $\bar y, \bar z \in S_L(\bar x)$, it holds that $f(\bar x, \bar y)=\varphi_L(\bar x) = f(\bar x, \bar z)$. Hence, $\mathcal{L}_{p_D}(\zeta(t))-\mathcal{L}_{p_D}(\bar\zeta) =0$ for all $t\geq 0$. \qed
\end{proof}

\begin{remark}\label{rem:Failure_Strict_Optimality}
Another interpretation of this result is that for any feasible point $\zeta:=(\bar x, \bar y, \bar z, \bar u, \bar v, \bar w)$ of problem \eqref{eq:tSLR} with $\bar y, \bar z \in S_L(\bar x)$ and $\Lambda_L(\bar x, \bar y)\neq \emptyset$, there exists a nonconstant feasible path
$\zeta(\cdot):[0,\infty)\to\mathcal F_{\mathrm{tSLR}}$ with origin at $0$ (i.e., with $\zeta(0)=\bar\zeta$), where the objective function of problem \eqref{eq:tSLR} is constant; i.e., 
$
\mathcal L_{p_D}(\zeta(t))=\mathcal L_{p_D}(\bar\zeta)
$
for every $t\geq 0$. 
Note that $\mathcal F_{\mathrm{tSLR}}$ denotes the feasible set of problem \eqref{eq:tSLR}. As a consequence, no feasible point $\zeta:=(\bar x, \bar y, \bar z, \bar u, \bar v, \bar w)$ of problem \eqref{eq:tSLR} can be a strict local optimal solution if $\bar y, \bar z \in S_L(\bar x)$ and $\Lambda_L(\bar x, \bar y)\neq \emptyset$. 
\end{remark}

For the final discussion of this section, recall that the classical SOSC for  \eqref{eq:tSLR} is a sufficient condition for the fulfillment of strict local optimality for the problem, as described in \eqref{eq:SOSC}. Therefore, to potentially leverage on the implication in Theorem \ref{the:SOSC} to construct a second order sufficient optimality condition for problem \eqref{eq:Pessimistic-Problem}, it is crucial to know whether the second order sufficient condition \eqref{eq:SOSC} can hold for \eqref{eq:tSLR}. In fact, we show that \eqref{eq:SOSC} systematically fails for any KKT point of problem \eqref{eq:tSLR}. The following lemmas will be crucial for this proof.

\begin{lemma}\label{lem:SOSC-implies-negative-definiteness}
    Let $\bar\zeta:=(\bar x, \bar y, \bar z, \bar u, \bar v, \bar w)$ satisfy the KKT conditions \eqref{KKT-U-1}--\eqref{KKT-L-3} with the Lagrange multipliers $\bar\alpha$ and $\bar\beta$. If the SOSC for \eqref{eq:tSLR} holds at $(\bar\zeta,\bar\alpha,\bar\beta)$, then $\nabla_{yy}^{2}\mathcal{L}_{p_D}(\bar\zeta)$ is negative definite. 
\end{lemma}
\begin{proof}Based on Assumption \ref{Assumption1}(1)--(4), it follows along the lines of the proof of Proposition \ref{Prop:sufficient_condition_full_rank}, cf.\eqref{eq:full_rank_fulfillment}, that for any $\zeta:=(x, y, z, u, v, w)$ with $u\geq$ and $w\geq 0$, we have $\nabla_{yy}^{2}\mathcal{L}_{p_D}(\zeta)\preceq 0$.

Suppose that $\nabla_{yy}^{2}\mathcal{L}_{p_D}(\bar\zeta)$ is singular. Then, there exists $d_y\neq0$ such that $\nabla_{yy}^{2}\mathcal{L}_{p_D}(\bar\zeta)d_y = 0$. Hence, we can easily check that 
\[
 d:=(0, d_y, 0, 0, 0, 0) \in    \mathfrak{C}(\bar\zeta,\bar\alpha),
\]
where $\mathfrak{C}(\bar\zeta,\bar\alpha)$ is the critical cone to the feasible set of problem \eqref{eq:tSLR} from  \eqref{eq:Critial_Cone_tSLR}. 

With the Lagrangian $\mathcal L(\zeta,\alpha,\beta)$ of \eqref{eq:tSLR} from \eqref{eq:Lagrangian_tSLR}, along the direction $d$ above, it holds that
\begin{align}\label{eq:SOSCcontradiction}
d^\top\nabla_{\zeta\zeta}^{2}
\mathcal L(\bar\zeta,\bar\alpha,\bar\beta)d
=
d_y^\top \underbrace{\nabla_{yy}^{2}\mathcal{L}_{p_D}(\bar\zeta) d_y}_{=\,0}
+
\sum_{j=1}^{m}\bar\beta_{\mathcal{L}_{p_D},j}\,
d_y^\top
\nabla_{yy}^{2}
\left(\nabla_{y_j}\mathcal{L}_{p_D}\right)(\bar\zeta)d_y.
\end{align}
We introduce the function 
\[
\psi(y):=
d_y^\top
\nabla_{yy}^{2}\mathcal{L}_{p_D}
(\bar x,y,\bar z,\bar u,\bar v,\bar w)d_y.
\]
Observe that based on Assumption \ref{Assumption1}(1)--(4), we have 
\[
\psi(y) \leq 0 = \psi(\bar y) \; \mbox{ for all }\; y \;\mbox{ near } \;\bar y
\]
given that $\nabla_{yy}^{2}\mathcal{L}_{p_D}(\bar\zeta)d_y = 0$. 
Thus, \(\bar y\) is a local maximizer of the differentiable function
\(\psi\), and consequently, we have 
$
\nabla\psi(\bar y)=0.
$
Thus, 
\[
\sum_{j=1}^{m}\bar\beta_{\mathcal{L}_{p_D},j}\,d_y^\top \nabla_{yy}^{2} \bigl(\nabla_{y_j}\mathcal{L}_{p_D}\bigr)(\bar\zeta)d_y =
\nabla\psi(\bar y)^\top\bar\beta_{\mathcal{L}_{p_D}}
=0.
\]
Therefore, \eqref{eq:SOSCcontradiction} implies $d^\top\nabla_{\zeta\zeta}^{2}
\mathcal L(\bar\zeta,\bar\alpha,\bar\beta)d
=0$. This contradicts the fulfillment of the SOSC given that 
$
d\in \mathfrak{C}(\bar\zeta,\bar\alpha)\setminus\{0\}.
$
Hence, $\nabla_{yy}^{2}\mathcal{L}_{p_D}(\bar\zeta)$ must be nonsingular. Since this matrix is negative
semidefinite, it follows that 
$\nabla_{yy}^{2}\mathcal{L}_{p_D}(\bar\zeta) \prec 0$.
\qed
\end{proof}

\begin{lemma}\label{lem:Non_Emptyness}
If $S_L(\bar x)=\{\bar y\}$ and \eqref{eq:GCQ} holds in $Z(\bar x)$ at $\bar y$, then $\Lambda_L(\bar x, \bar y) \neq \emptyset$. 
\end{lemma}
\begin{proof}
Since $Z(\bar x)=S_L(\bar x)=\{\bar y\}$, it holds that $T_{Z(\bar x)}(\bar y)=\{0\}$. Hence, thanks to the GCQ,
\[
\left(L_{Z(\bar x)}(\bar y)\right)^* = \left(T_{Z(\bar x)}(\bar y)\right)^* =\mathbb{R}^m.
\]
Since $L_{Z(\bar x)}(\bar y)$ is a closed convex cone, 
$
L_{Z(\bar x)}(\bar y)
=
\left(
\left(L_{Z(\bar x)}(\bar y)\right)^*
\right)^*
=
\{0\}.
$ Moreover, 
\begin{equation}\label{eq:LZ-cone}
    L_{Z(\bar x)}(\bar y)
=
\left\{
d_y\in L_{Y(\bar x)}(\bar y)
\ \middle|\
\nabla_y f(\bar x,\bar y)^\top d_y\leq0
\right\}
\end{equation}
given that the constraint
$
f(\bar x,y)-\varphi_L(\bar x)\leq0
$
is active at $\bar y$. Combining \eqref{eq:LZ-cone} with the fact that $L_{Z(\bar x)}(\bar y)=\{0\}$, it holds that $\nabla_y f(\bar x,\bar y)^\top d_y>0$ if $d_y\in L_{Y(\bar x)}(\bar y)\setminus\{0\}$. Hence, 
\[
\nabla_y f(\bar x,\bar y)^\top d_y\geq 0 \;\mbox{ for all }\; d_y\in L_{Y(\bar x)}(\bar y).
\]
Thus, based on the dual cone definition \eqref{eq:dual_cone}, we clearly have that 
\begin{equation}\label{eq:Dual_Incl}
    \nabla_y f(\bar x,\bar y)
\in
\left(L_{Y(\bar x)}(\bar y)\right)^*.
\end{equation}
Furthermore, note that the linearized cone to $Y(\bar x)$ at $\bar y$ is
\[
L_{Y(\bar x)}(\bar y)
=
\left\{
d_y\in\mathbb{R}^m
\ \middle|\
\begin{aligned}
\nabla_y g_i(\bar x,\bar y)^\top d_y&\leq0,
&&i\in I_g(\bar x,\bar y)\\
\nabla_y h_j(\bar x,\bar y)^\top d_y&=0,
&&j=1,\ldots,q
\end{aligned}
\right\},
\]
where $I_g(\bar x,\bar y)
:=
\left\{
i\in\{1,\ldots,p\}
\ \middle|\
g_i(\bar x,\bar y)=0
\right\}$, and its dual cone can be written as 
\[
\begin{array}{rll}
 \left(L_{Y(\bar x)}(\bar y)\right)^* & = & \Big\{a\in \mathbb{R}^m|\; 
 \exists \lambda_i\geq0 \mbox{ for } i\in I_g(\bar x,\bar y) \mbox{ and } \exists \mu\in\mathbb{R}^q:\\[0.5ex]
                                      &   & \qquad \qquad \left. a= -\sum_{i\in I_g(\bar x,\bar y)}
\lambda_i\nabla_y g_i(\bar x,\bar y) - \sum_{j=1}^{q} \mu_j\nabla_y h_j(\bar x,\bar y)\right\}.
\end{array}
\]
Combining this with \eqref{eq:Dual_Incl}, we clearly have that $\Lambda_L(\bar x, \bar y) \neq \emptyset$. \qed
\end{proof}

We are now ready to show that the  SOSC for problem \eqref{eq:tSLR} fails at any of its KKT points. 
\begin{proposition}\label{prop:SOSCfail}
     Let $\bar\zeta:=(\bar x, \bar y, \bar z, \bar u, \bar v, \bar w)$ satisfy the KKT conditions \eqref{KKT-U-1}--\eqref{KKT-L-3} with the Lagrange multipliers $\bar\alpha$ and $\bar\beta$. Then, the SOSC for problem \eqref{eq:tSLR} fails at $(\bar\zeta,\bar\alpha,\bar\beta)$. %
\end{proposition}
\begin{proof}
By contradiction, we assume that the SOSC for \eqref{eq:tSLR} holds at $(\bar\zeta,\bar\alpha,\bar\beta)$. Then, it follows from 
Lemma~\ref{lem:SOSC-implies-negative-definiteness} that 
$
\nabla_{yy}^{2}\mathcal{L}_{p_D}(\bar\zeta)\prec0.
$
Hence, the first block of \eqref{KKT-I-1} implies that 
$
\bar\beta_{\mathcal{L}_{p_D}}=0.
$

\textbf{Case 1:} Suppose that $\bar w>0$. Then, combining  this with $\bar\beta_{L_{pD}}=0$, it follows from  \eqref{KKT-I-1}--\eqref{KKT-I-3} that we have 
\begin{equation}\label{eq:result_equations}
    g(\bar x, \bar y)\leq 0, \;\, \bar u\geq 0, \;\, \bar{u}^\top g(\bar x, \bar y)=0, \;\,  h(\bar x, \bar y)=0, \;\, f(\bar x, \bar y)-f(\bar x, \bar z)=0. 
\end{equation}
Additionally, considering $\bar w >0$ again, we have from \eqref{KKT-L-1}--\eqref{KKT-L-3} that $(\alpha'_g,\beta'_h):=(\bar\alpha_g,\bar\beta_h)/\bar w\in\Lambda_L(\bar x, \bar z)$. Hence, by the convexity of the lower-level problem in Assumption \ref{Assumption1}(2)--(4), it follows that $\bar z \in S_L(\bar x)$. Thus, combining this with the last equation in \eqref{eq:result_equations}, 
\begin{equation}\label{eq:result_equations_2}
f(\bar x, \bar y) = f(\bar x, \bar z)\leq \varphi_L(\bar x) \leq  f(\bar x, \bar y).    
\end{equation}
Furthermore, with $\bar y\in Y(\bar x)$ from \eqref{eq:result_equations}, we have from \eqref{eq:result_equations_2} that $\bar y \in S_L(\bar x)$.
Moreover, it is not hard to see that the vectors $d_z\in\R^m$ with $(0,0,d_z,0,0,0)\in\mathfrak{C}(\bar\zeta, \bar\alpha)$ form the critical cone of \eqref{eq:LL} for $x:=\bar x$ at $\bar z$ with respect to the multiplier vector $\alpha_g'$. Therefore, the SOSC with the Lagrangian $\mathcal{L}(\zeta, \alpha, \beta)$
of \eqref{eq:tSLR} from \eqref{eq:Lagrangian_tSLR} implies $d_z^\top \nabla^2_{zz} \mathcal{L}(\bar\zeta, \bar\alpha, \bar\beta)d_z>0$ for all $d_z\in\mathfrak{C}(\bar z,\alpha_g')\setminus\{0\}$. In view of
\begin{align*}
    \nabla^2_{zz} \mathcal{L}(\bar\zeta, \bar\alpha, \bar\beta)&= \bar w\nabla^2_{zz}f(\bar x, \bar z)+\nabla^2_{zz}g(\bar x, \bar z)\bar\alpha_g+\nabla^2_{zz} h(\bar x, \bar z)\bar\beta_h\\
    & =  \bar w\left(\nabla^2_{zz}f(\bar x, \bar z)+\nabla^2_{zz}g(\bar x, \bar z)\alpha_g'+\nabla^2_{zz} h(\bar x, \bar z)\beta_h'\right),
\end{align*}
one obtains the SOSC for $\bar z$ to be a unique local minimal point of problem \eqref{eq:LL} for $x:=\bar x$. In view of $\bar z\in S_L(\bar x)$, the latter set is thus a singleton. Therefore, it holds that $\bar y = \bar z$. Subsequently, considering the system \eqref{KKT-L-1}--\eqref{KKT-L-3} once again,  it holds that $\Lambda_L(\bar x, \bar y) \neq \emptyset$. 
Hence, applying Proposition \ref{prop:flat_ray}, it follows that $\bar \zeta$ is not a strict local optimality solution of  \eqref{eq:tSLR}. This contradicts the fulfillment of the SOSC for problem \eqref{eq:tSLR}  at $(\bar\zeta,\bar\alpha,\bar\beta)$.

\textbf{Case 2:} Suppose that $\bar w = 0$. Then, proceeding in a way similar to the previous case, it follows that with $\bar\beta_{\mathcal{L}_{p_D}}=0$, we get from equations \eqref{KKT-I-1}--\eqref{KKT-I-3} that 
\begin{equation}\label{eq:result_equations_V2}
    g(\bar x, \bar y)\leq 0, \;\, \bar u\geq 0, \;\, \bar{u}^\top g(\bar x, \bar y)=0, \;\,  h(\bar x, \bar y)=0, \;\, f(\bar x, \bar y)-f(\bar x, \bar z) \leq 0. 
\end{equation}
Moreover, with $\bar w =0$, it follows from the system \eqref{KKT-L-1}--\eqref{KKT-L-3} that $\bar z$  is a KKT point of the lower-level feasibility problem
\begin{equation}\label{eq:LL0}\tag{\text{$LL^0(x)$}}
    \min_z\,0\ \text{ s.t. }\ z\in Y(\bar x)
\end{equation}
with multipliers $\bar\alpha_g$ and $\bar\beta_h$. 
Since $Y(\bar x)$ is convex, $\bar z$  is also a global minimal point. Moreover, given that the vectors $d_z\in\R^m$ with $(0,0,d_z,0,0,0)\in\mathfrak{C}(\bar\zeta,\alpha)$ form the critical cone of \eqref{eq:LL0} at $\bar z$, we obtain from the fulfillment of the SOSC at $\bar z$ that this point is a unique local minimal point of \eqref{eq:LL0}. 
This is only possible if $Y(\bar x)$ is a singleton, implying that also its subset $S_L(\bar x)$ is a singleton. Hence, $S_p(\bar x) = S_L(\bar x) = Y(\bar x)=\{\bar z\}$. Subsequently, $f(\bar x, \bar z) \leq \varphi_L(\bar x)$. Combining this with \eqref{eq:result_equations_V2}, %
\[
 g(\bar x, \bar y)\leq 0, \;\,  h(\bar x, \bar y)=0, \;\, \varphi_L(\bar x) \leq f(\bar x, \bar y) \leq f(\bar x, \bar z) \leq \varphi_L(\bar x). 
\]
This implies that $\bar y\in Y(\bar x)$ and $f(\bar x, \bar y)=\varphi_L(\bar x)$. Thus, $\bar y\in S_L(\bar x)$. This implies that $\bar y =\bar z$. 
Hence, $S_L(\bar x)=\{\bar y\}$, and we have from a combination of Assumption \ref{Assumption1}(5) (while accounting for the fact that $S_p(\bar x) \subset S_L(\bar x)=\{\bar y\}$) and Lemma \ref{lem:Non_Emptyness} that $\Lambda_L(\bar x, \bar y) \neq \emptyset$. Similarly to the previous case (Case~1), we get a  contradiction  by applying Proposition \ref{prop:flat_ray}. %
\qed
\end{proof}

Considering Remark \ref{rem:dual_y_not_optimal} and Remark \ref{rem:z-in-SL}, it is not necessarily expected that for an optimal solution $\bar\zeta:=(\bar x, \bar y, \bar z, \bar u, \bar v, \bar w)$ of problem \eqref{eq:tSLR}, one would have $\bar y\in S_L(\bar x)$, $\bar z\in S_L(\bar x)$, or $(\bar u, \bar v)\in \Lambda_L(\bar x, \bar y)$. The proof of Proposition \ref{prop:SOSCfail} shows, however, that all three conditions would be enforced by the SOSC in the optimal solution, leading to a contradiction in view of %
Proposition \ref{prop:flat_ray}. 
This illustrates that having these three conditions satisfied simultaneously is actually not a positive situation for problem \eqref{eq:tSLR}. In particular, the systematic failure of the SOSC must be expected to interfere with the convergence of certain types of numerical methods for \eqref{eq:tSLR} as a consequence of the next result, which is a refinement of Theorem~\ref{the:tSLR_SONC} (implied by Proposition~\ref{prop:SOSCfail}).

\begin{corollary}\label{cor:failure_second_order_method}
Under the assumptions of Theorem~\ref{the:tSLR_SONC}, its assertions as well as the following two statements are satisfied: 
\begin{itemize}
\item[(a)] There exists some $d\in \mathfrak{C}(\zeta, \alpha)\setminus\{0\}$ with $d^\top\nabla^2_{\zeta\zeta} \mathcal{L}\left(\zeta, \alpha, \beta\right)d=0$.
\item[(b)] Let $\alpha_i>0$ for all $i$ with $\tilde{G}_i(\zeta)=0$ (i.e., strict complementarity holds) and let the rows of the matrix $B$ collect the derivatives $\nabla\tilde{G}_i(\zeta)^\top$ with $\tilde{G_i}(\zeta)=0$ and $\nabla\tilde{H_j}(\zeta)^\top$, $j=1\ldots,q$. Then the Jacobian $\begin{pmatrix}
            \nabla^2_{\zeta\zeta} \mathcal{L}\left(\zeta, \alpha, \beta\right) & B^\top\\ B & 0
        \end{pmatrix}$ of the equations in the KKT system \eqref{KKT-U-1}--\eqref{KKT-L-3} is singular.
\end{itemize}
\end{corollary}
\begin{proof}
The assertion of part (a) immediately follows from Theorem~\ref{the:tSLR_SONC} and Proposition~\ref{prop:SOSCfail}. To see part (b), note that the strict complementarity assumption implies that the cone $\mathfrak{C}(\zeta, \alpha)$ is a linear space (the tangent space to the feasible set). Therefore, the assertion of part (a) can be restated as the singularity of the restriction of $\nabla^2_{\zeta\zeta} \mathcal{L}\left(\zeta, \alpha, \beta\right)$ to this tangent space. Since under the assumptions of Theorem~\ref{the:tSLR_SONC}, the LICQ holds at $\zeta$, the matrix $B$ possesses full row rank. The assertion now follows from \cite[Lemma 3.4]{MR787745} (see also \cite{MR907394}). \qed
\end{proof}

A take away of this singularity result is that to leverage on Theorem \ref{the:SOSC} for practical sufficient conditions for problem \eqref{eq:Pessimistic-Problem}, different types of characterizations for sufficient optimality for problem \eqref{eq:tSLR} would need to be explored, including possibly first order-type conditions. In addition, one may not expect second order methods to work for the algorithmic solution of \eqref{eq:tSLR}, but without further modifications of the reformulation one should rather resort to first order methods.

\section{Numerical illustrations}\label{sec:Algorithmic framework for numerical computations}
Considering the relationships  established in Section \ref{sec:Single-level-Reform} between \eqref{eq:Pessimistic-Problem} and \eqref{eq:tSLR}, we use six well-known examples from the literature to illustrate this connection between the two problems. We only consider the global relationship here; cf. Corollary \ref{cor:summary_relationships}(a). 
To proceed, we use the DIRECT algorithm introduced in \cite{jones1993lipschitzian}. However, since this algorithm can only handle box constraints, it is combined with a (quadratic) penalization (to get an unconstrained problem). 
More precisely, assuming that we are solving the  constrained problem \eqref{eq:min_standard}, it is approximated by the
bound-constrained penalized problem
\begin{equation}\label{eq:penalized_standard}
\min_{x\in[\ell,u]}
\Psi_{\rho}(x)
:=
\mathfrak{f}(x)+\rho\mathcal{V}(x),
\end{equation}
where, $\rho>0$ is the penalty parameter and $[\ell,u]$ is a finite search box.
For variables that are originally unbounded or one-sided bounded, sufficiently
large artificial bounds, $-B\leq x_i\leq B$ and $0\leq x_i\leq B$,  are introduced, 
as appropriate. The value of $B$ is increased whenever a candidate solution
lies close to an artificial boundary. Note that in \eqref{eq:penalized_standard}, $\mathcal{V}$ denotes the quadratic constraint-violation function 
\begin{equation}\label{eq:violation_measure}
\mathcal{V}(x)
:=
\sum_{i=1}^{p}
    \bigl[\max\{0,\mathfrak{g}_i(x)\}\bigr]^2
+
\sum_{j=1}^{q}
    \mathfrak{h}_j(x)^2.
\end{equation}

Observe that $\mathcal{V}(x)=0$ if and only if $x$ is feasible for
\eqref{eq:min_standard}. Hence, the global-search procedure consists of two phases. First, DIRECT is applied to
\[
\min_{x\in[\ell,u]}\mathcal{V}(x)
\]
to identify a point with small constraint violation. Second, DIRECT is applied
to the penalized objective $\Psi_\rho$ \eqref{eq:penalized_standard}. The first phase prevents an improvement in $\mathfrak f$ from concealing a large infeasibility during the penalized
search. DIRECT partitions the normalized box into hyperrectangles and samples their
centers. Hyperrectangles that are potentially optimal, based on both their
sampled objective values and their sizes, are subdivided. The method therefore
balances global exploration of large unexamined regions with local refinement
around promising points. This local refinement is done by applying the well-known sequential least squares programming (SLSQP) method \cite{kraft1988software} to \eqref{eq:min_standard}. Considering the combination of the \textit{D}IRECT method with the \textit{P}enalization and \textit{S}LSQP algorithm, we label the computational tool used for the experiments presented here as \textit{DPS}. %

Let $\bar{x}$ denote the candidate returned by the two DIRECT phases. It is
accepted as feasible when
$\sqrt{\mathcal{V}(\bar{x})}
\leq \varepsilon$, where $\varepsilon>0$ represents the feasibility tolerance. %
Overall, DPS is just a deterministic global-search heuristic
for the constrained problem \eqref{eq:min_standard}. DIRECT has global convergence properties
for the bound-constrained penalized problem under its standard assumptions \cite{jones1993lipschitzian}.
However, a quadratic penalty with finite $\rho$ is not generally exact.
Consequently, global minimization of \eqref{eq:penalized_standard} alone does
not provide an unconditional certificate of global optimality for
\eqref{eq:min_standard}. %
But, as it will be shown below, the performance of DPS confirms the finding of Corollary \ref{cor:summary_relationships}(a) that problem \eqref{eq:Pessimistic-Problem} can effectively be solved via the tSLR model introduced in this paper. 

The reason for using the DPS scheme rather than established global solvers for constrained optimization is the fact that it can be implemented with fully open-source tools available within the  \texttt{scipy.optimize} Python environment. %
We apply this framework to six well-known toy examples from the literature, and for which global optimal solutions have been reported. We label these problems as P1 (see, e.g., \cite[Example 2.1]{benchouk2025scholtes}), P2 (see, e.g., \cite[Example 3.1]{benchouk2026relaxation}), P3 (see, e.g., \cite[Principal Agent problem in Table 5]{antoniou2024delta}), P4 (see, e.g., \cite[Example 4.1]{dempe2018pessimistic}), P5 (see, e.g., \cite[Example 1]{zeng2020practical}), and P6 (see, e.g., \cite[Example 2]{zeng2020practical}). All the six examples satisfy  Assumption \ref{ass:polyhedral}; hence, we also apply DPS to the corresponding versions of the model \eqref{eq:tSLRLP}. Additionally, we apply DPS to the SLRs \eqref{eq:minmax_Bo_Zeng_SLR} and \eqref{eq:tSLRCC} to enable a comparison of their performance with that of our tSLRs. Out of all the existing methods in the literature presented in Section \ref{sec:Existing reformulations and numerical algorithms}, the SIP-PBDA \cite{wiesemann2013pessimistic} and TLVF-GSS1 \cite{strekalovsky2025one} algorithms seem to be the only ones that can directly approximate global optimal solutions. Hence, we also apply them to applicable problems from our six examples (i.e., P1 and P2 for SIP-PBDA and TLVF-GSS1 for the rest). Note that here, we also use the DPS scheme as global solver for the inner problem for SIP-PBDA (Algorithm 1 in \cite{wiesemann2013pessimistic}).

\begin{table}[htbp]%
\centering
\begin{tabular}{lcccccc}
\toprule
Method & P1 & P2 & P3 & P4 & P5 & P6 \\
\midrule
\textsc{MM-CC}    & \checkmark & \checkmark & \checkmark & \checkmark & \ding{55} & \ding{55} \\
\textsc{SP-CC}    & \ding{55} & \checkmark & \ding{55}$^\dagger$ & \checkmark & \checkmark & \ding{55} \\
\textsc{tSLR}     & \checkmark & \checkmark & \checkmark & \checkmark & \checkmark & \checkmark \\
\textsc{tSLR-LP}  & \checkmark & \checkmark & \checkmark & \checkmark & \checkmark & \checkmark \\
\textsc{SIP-PBDA} & \checkmark & \checkmark & -- & -- & -- & -- \\
\textsc{TLVF-GSS1} & -- & -- & \checkmark & \checkmark & \checkmark & \ding{55} \\
\bottomrule
\end{tabular}
\caption{Success (\checkmark, $x_{\mathrm{err}}\le10^{-3}$ and feasible) or
failure (\ding{55}) on each of of the problems P1, \ldots, P6; ``--'' marks a method outside
its applicability scope for that problem. %
Exact $x_{\mathrm{err}}$ values are given in
Fig.~\ref{fig:p1-6-heatmap}, not repeated here.
($^\dagger$: \textsc{SP-CC} on P3 has small $x_{\mathrm{err}}$
($9.6\times10^{-14}$) but is marked as failure because the returned point
does not satisfy the $10^{-4}$ feasibility tolerance.)}\label{tab:p1-6-xerr}
\end{table}

\paragraph{DPS initialization.} Each DPS solve for the corresponding versions of MM-CC, SP-CC, tSLR, and tSLR-LP  uses $12$ deterministic
starting points for the local SLSQP polish stage: one from a global
DIRECT search on the exterior-penalized objective (penalty weight
$\rho=10^{5}$, evaluation budget $N_{\max}=\min \{3000\hat{n}, \,80000\}$ with $\hat{n}$ being the number of variables of the corresponding problems), plus $11$ further points obtained by scaling every coordinate to the same fraction
$\tfrac{k}{12}$, where $k=1,\dots,11$, of its bound interval--a fixed,
reproducible design with no random seed. The best feasible SLSQP result
across all $12$ starts is retained. %
\begin{table}[htbp]
\centering
\begin{tabular}{lccccc}
\toprule
Method & Scope & Success rate & Mean $x_{\mathrm{err}}$ & Mean time (s) \\
\midrule
tSLR     & 6 & $\mathbf{100.0\%}$ & $2.4\times10^{-11}$ & $1.06$ \\
tSLR-LP   & 6 & $\mathbf{100.0\%}$ & $4.5\times10^{-12}$ & $0.53$ \\
SIP-PBDA  & 2 & $\mathbf{100.0\%}$ & $1.8\times10^{-15}$ & $0.11$ \\
TLVF-GSS1  & 4 & $75.0\%$           & $2.50$              & $1.00$ \\
MM-CC    & 6 & $66.7\%$           & $0.37$              & $1.09$ \\
SP-CC    & 6 & $50.0\%$           & $0.91$              & $7.78$ \\
\bottomrule
\end{tabular}
\caption{Aggregate performance over each method's applicable scope within
problems P1, \ldots, P6. Scope refers to the number of problems that the corresponding method can solve.}\label{tab:p1-6-agg}
\end{table}

\begin{figure}[htbp]
\centering
\includegraphics[width=0.75\textwidth]{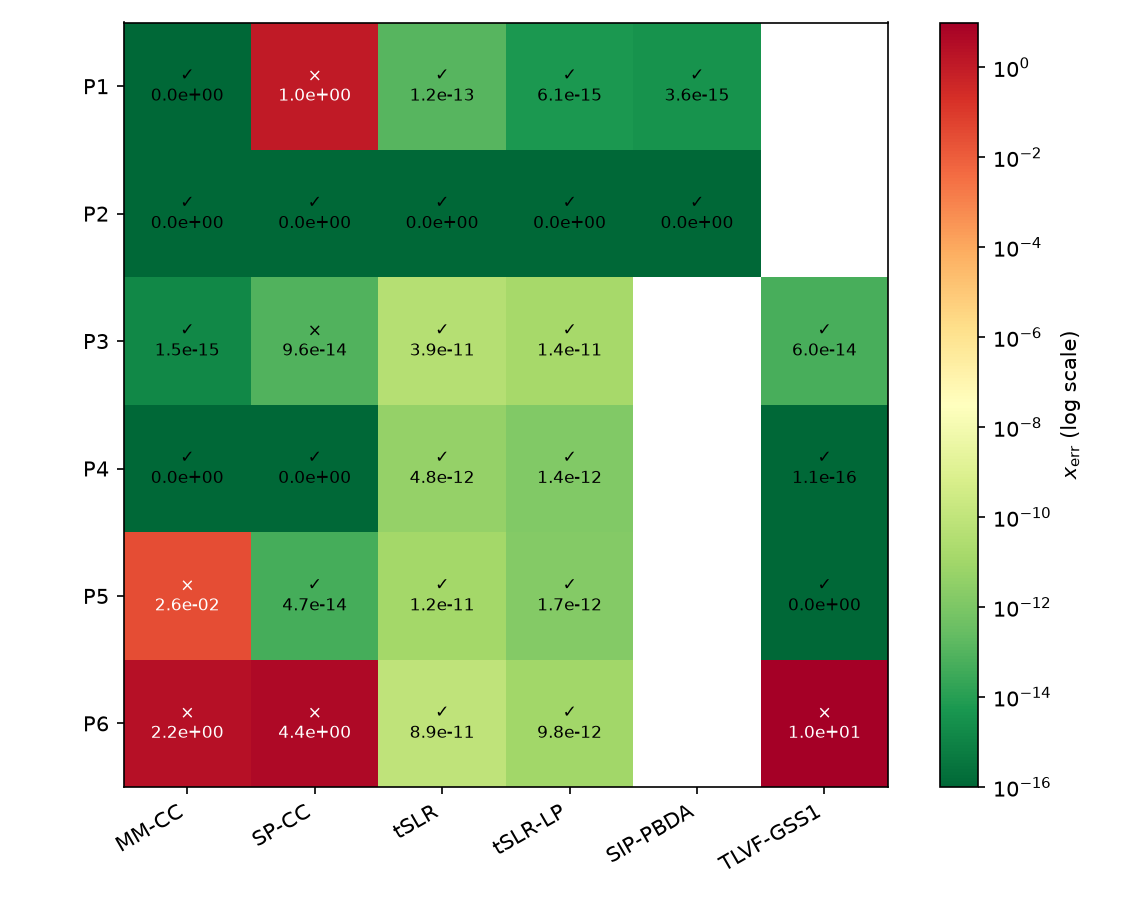}
\caption{Leader's solution error $x_{\mathrm{err}}$ across all problems and
all six methods (log color scale; green = accurate, red = inaccurate).
Empty cells represent cases outside a method's applicability scope.}
\label{fig:p1-6-heatmap}
\end{figure}

\paragraph{Performance measures.} Considering the focus of Corollary \ref{cor:summary_relationships}(a) on computing optimal solutions of problem \eqref{eq:Pessimistic-Problem} from the tSLR, the primary measure is the leader
solution error
\[
x_{\mathrm{err}} := \|\tilde x - x^{\ast}\|_\infty,
\]
the distance between the computed leader decision $\tilde x$ and the known
true optimizer $x^\ast$. %
A run of the DPS scheme will be said to be a \emph{success} if
$x_{\mathrm{err}} \le 10^{-3}$ and $\tilde x$ is feasible. Table \ref{tab:p1-6-xerr} presents the behavior of the methods w.r.t. their success or not when applied to the different problems (P1, \ldots, P6), while Table \ref{tab:p1-6-agg} provides the average error and time performances. 

\begin{figure}[htbp]
\centering
\includegraphics[width=0.85\textwidth]{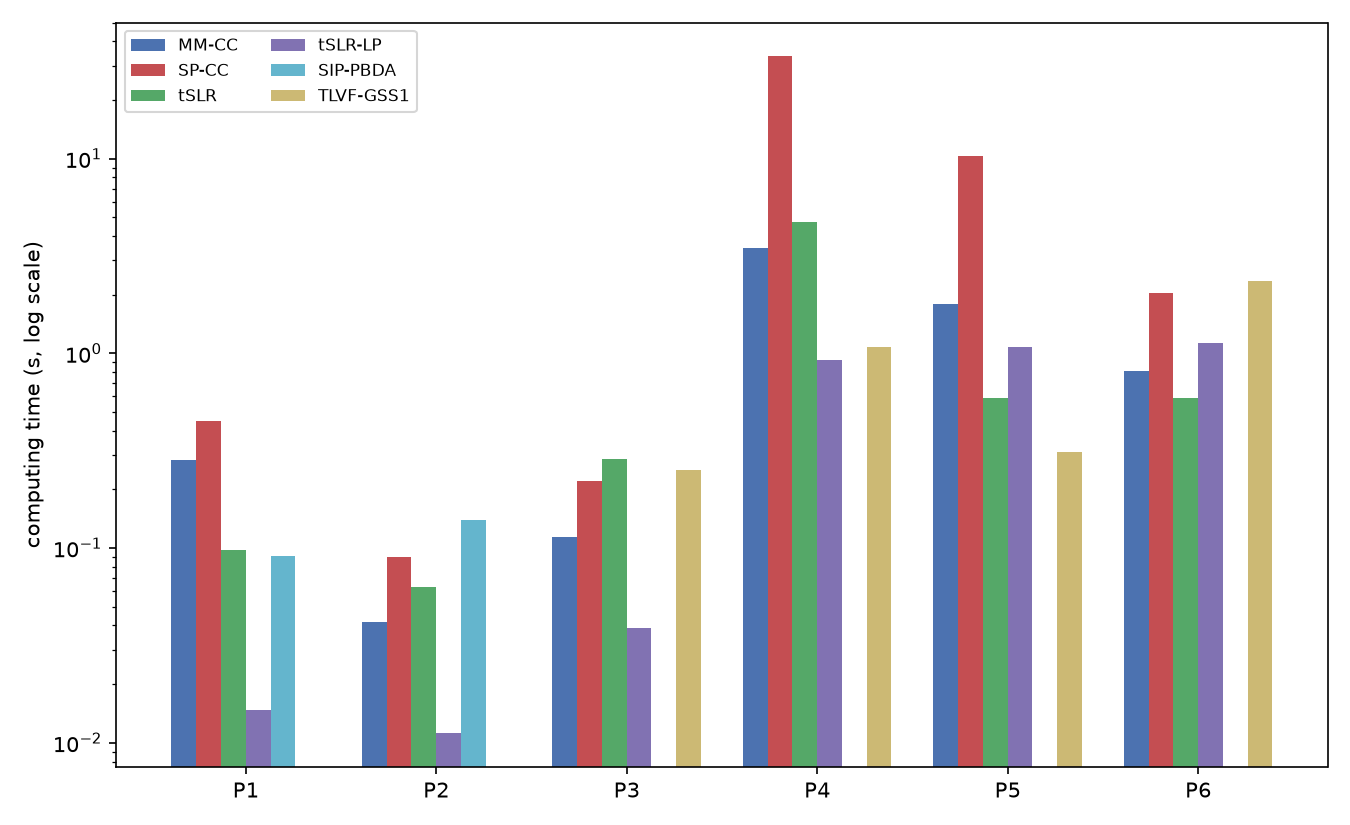}
\caption{%
A missing bar means the method is outside its applicability scope for that
problem. SP-CC's cost is the clear outlier, most visibly on
P4 ($33.5$s, an order of magnitude above every other method on that
problem), driven by its larger multiplier-variable count inflating each
SLSQP polish; tSLR-LP is consistently among the cheapest methods
run.}
\label{fig:p1-6-time}
\end{figure}

\paragraph{Discussion.}  tSLR and
tSLR-LP are the clear winners: both reach $100\%$ success with
error at or below machine precision on every problem, matching the
scope-limited bespoke methods \textsc{SIP-PBDA} and \textsc{TLVF-GSS1}
without their applicability restrictions. \textsc{TLVF-GSS1} is close
behind at $75\%$, missing only P6--consistent with a known
local-trap limitation of its global optimality conditions (GOC)--escape mechanism on that problem's
degenerate lower-level reported \cite{strekalovsky2025one}. \textsc{MM-CC} and \textsc{SP-CC}, the two generic
complementarity-penalty reformulations, are markedly less reliable
($66.7\%$ and $50.0\%$): both fail on problem P6, and each additionally fails
one of P1, P3, or P5 by landing on a spurious stationary point of the
MPCC reformulation rather than the true bilevel optimum--a known hazard
of penalty-based reformulations that exact gradients and multi-start
restarts substantially mitigate but do not eliminate on this problem set.
\textsc{SP-CC} is additionally the slowest method by a wide margin (mean
$7.78$s vs. \textit{less than} $1.1$s for every other method), owing to its larger multiplier-variable count
inflating the cost of each SLSQP polish -- most visibly on P4,
where it takes $33.5$s against at most a few seconds for every other
method. tSLR-LP is both the most reliable method and consistently
one of the cheapest, making it the strongest method overall on this
problem range by every measure reported here. The individual computing time for each method, where applicable, is provided in Fig. \ref{fig:p1-6-time}; and it is clear that except from P5, tSLR or tSLR-LP has the best computing time on each problem. 

\section{Concluding comments} 
In this paper, we introduce a single-level reformulation (SLR) for the pessimistic bilevel optimization problem, which we call a \textit{true single-level reformulation} (tSLR) because, unlike conventional SLRs for bilevel optimization, it involves neither complementarity conditions nor value functions. We establish rigorous global and local relationships between the tSLR and the original problem \eqref{eq:Pessimistic-Problem}. The proposed reformulation offers several important advantages. In particular, it contains no implicit variables and can satisfy the linear independence constraint qualification, whereas even the weaker Mangasarian--Fromovitz constraint qualification systematically fails for all known SLRs of optimistic and pessimistic bilevel programs. Moreover, \eqref{eq:tSLR} enables the derivation of new necessary optimality conditions for pessimistic bilevel programs. On the numerical side, six illustrative examples provide a proof of concept for computing globally optimal solutions with the tSLR. 
Future work will focus on developing numerical methods tailored to \eqref{eq:tSLR} and systematically comparing them, on suitable test problems, with the existing algorithms for pessimistic bilevel optimization reviewed in Section \ref{sec:Existing reformulations and numerical algorithms}. A major limitation, however, is the systematic failure of the classical second-order sufficient condition at Karush--Kuhn--Tucker points of \eqref{eq:tSLR} within the framework of this paper, as shown in Subsection \ref{sec:Second order sufficient conditions}. Consequently, second-order methods may struggle to solve \eqref{eq:tSLR} efficiently, as observed in Corollary \ref{cor:failure_second_order_method}.

As observed throughout Sections \ref{sec:Single-level-Reform} and \ref{sec:Optimality conditions}, some features of problem \eqref{eq:tSLR} require careful interpretation. In particular, 
variables $u$, $v$, and $w$ arise from the Wolfe-dual representation of
the intermediate problem, but they are not required to satisfy the
complementary-slackness conditions that would ordinarily be associated
with the corresponding constraints. Likewise, at an optimal solution
$(x,y,z,u,v,w)$ of problem \eqref{eq:tSLR}, the auxiliary point $y$ need not solve the
intermediate problem and $z$ need not solve the original lower-level
problem. Their role is to construct an exact single-level representation
from which the optimal leader decision $x$ can be recovered. Once an optimal $x$
has been obtained for problem \eqref{eq:Pessimistic-Problem}, corresponding optimal solutions for the intermediate- and lower-level problems \eqref{eq:IL} and \eqref{eq:LL}, respectively, can be generated (see Remark~\ref{rem:z-in-SL}). %

It is also important to note that the proposed true single-level reformulation  does not remove the fundamental existence difficulties of pessimistic bilevel optimization. The assumptions used
to establish the tSLR do not, by themselves, guarantee that either the
original problem or its reformulation possesses an optimal solution.
In particular, the feasible set of the tSLR is generally unbounded, and
the discontinuity and nonattainment phenomena caused, for example, by
a failure of lower semicontinuity of the lower-level optimal solution set-valued mapping
are inherited by the reformulation. This is unavoidable for an exact
reformulation: if the original pessimistic problem has a finite but
unattained infimum, the tSLR cannot be expected to manufacture an
optimal solution that does not exist.

Despite the failure of \eqref{eq:SOSC} for problem \eqref{eq:tSLR}, the broader message of this work is a very \textit{optimistic} one. Pessimistic bilevel programs possess genuine structural
difficulties and may fail to admit optimal solutions without additional
regularity assumptions. However, whenever an optimal solution does
exist, computing it need not be intrinsically more difficult than
solving the corresponding optimistic problem. Indeed, the tSLR suggests
that, for some problem classes, finding a pessimistic optimal solution
may even be easier, because it can be approached through a
reformulation without implicit variables, complementarity constraints, or optimal value functions, and which can satisfy the standard constraint qualifications that
are unavailable to conventional bilevel optimization reformulations.

These observations open several directions for future research. A first
priority could be the development of global and local algorithms specifically
tailored to the structure of problem~\eqref{eq:tSLR}, particularly, 
methods that exploit the close relationship between its objective
function and its constraint
\[
\nabla_y F(x, y) - \nabla_y g(x, y)^\top u - \nabla_y h(x, y)^\top v- w \nabla_y f(x, y) =0
\]
could be part of such future explorations. Specialized algorithms for the
reduced formulation~\eqref{eq:tSLRLP}, as well as extensions capable
of accommodating discrete or mixed-integer upper-level variables, also
deserve investigation. Finally, a broader computational study on
systematic benchmark collections is needed to compare these methods
with the existing algorithms reviewed in
Section~\ref{sec:Existing reformulations and numerical algorithms} and
to identify the classes of pessimistic bilevel programs for which the
tSLR provides the greatest advantage.

\section*{Statements and declarations}

\subsubsection*{Competing interests}
${}$\\[-5ex]
The authors have no competing interests to declare that are relevant to the content of this article.

\subsubsection*{Ethics approval}
${}$\\[-5ex]
Not applicable.

\subsubsection*{Consent to participate}
${}$\\[-5ex]
Not applicable.

\subsubsection*{Consent for publication}
${}$\\[-5ex]
Not applicable.

\subsubsection*{Data availability}
${}$\\[-5ex]
The data and code used to produce the numerical results in Section~\ref{sec:Algorithmic framework for numerical computations} are available from the corresponding author upon reasonable request.

\subsubsection*{Code availability}
${}$\\[-5ex]
The code used to produce the numerical results in Section~\ref{sec:Algorithmic framework for numerical computations} is available from the corresponding author upon reasonable request.

\subsubsection*{Author contributions}
${}$\\[-5ex]
Both authors contributed to the conception, analysis, and writing of this work, and read and approved the final manuscript.

\end{document}